\documentclass[leqno]{amsart}

 \usepackage{amssymb, amsmath, amsthm, chngcntr, enumerate, mathabx, mathrsfs, mathtools}
\usepackage[numbers]{natbib}

\numberwithin{equation}{section}

\newtheorem{thm}{Theorem}
\newtheorem{theo}[thm]{Theorem}

\newtheorem{lemma}[thm]{Lemma}

\theoremstyle{remark}
\newtheorem{rmk}[thm]{Remark} 

\title[On a problem of Pichorides]{On a problem of Pichorides}
\author{Odysseas Bakas}
\address{Department of Mathematics, Stockholm University, 106 91 Stockholm, Sweden}
\email{bakas@math.su.se}

 
\begin{document}

\begin{abstract} Let $S^{(\Lambda)}$ denote the classical Littlewood-Paley square function formed with respect to a lacunary sequence $\Lambda$ of positive integers. Motivated by a remark of Pichorides, we obtain sharp asymptotic estimates of the behaviour of  the operator norm of $S^{(\Lambda)}$ from the analytic Hardy space $H^p_A (\mathbb{T})$ to $L^p (\mathbb{T})$ and of the behaviour of the $L^p (\mathbb{T}) \rightarrow L^p (\mathbb{T})$ operator norm of $S^{(\Lambda)}$  ($1 < p < 2$) in terms of the  ratio of the lacunary sequence $\Lambda$. Namely, if $\rho_{\Lambda}$ denotes the ratio of $\Lambda$, then we prove that 
$$ \sup_{\substack{ \| f \|_{L^p (\mathbb{T})} = 1 \\ f \in H^p_A (\mathbb{T}) } }   \big\| S^{(\Lambda)} (f) \big\|_{L^p (\mathbb{T})} \lesssim \frac{1}{p-1}  (\rho_{\Lambda} - 1 )^{-1/2} \quad (1<p<2)$$
and
$$ \big\| S^{(\Lambda)}  \big\|_{L^p (\mathbb{T}) \rightarrow L^p (\mathbb{T})} \lesssim \frac{1}{(p-1)^{3/2}}  (\rho_{\Lambda} - 1)^{-1/2} \quad (1<p<2)$$
and that the exponents $r=1/2$ in $(\rho_{\Lambda} - 1)^{-1/2} $ cannot be improved in general. Variants in higher dimensions and in the Euclidean setting are also obtained.
\end{abstract}

\maketitle

%%%%%%%%%%%%%%%%
%%%%%%%%%%%%%%%%
\section{Introduction}\label{intro}
Given a strictly increasing sequence $\Lambda = (\lambda_j)_{j \in \mathbb{N}_0}$ of positive integers,   consider the corresponding Littlewood-Paley projections $\big(\Delta^{(\Lambda)}_j \big)_{j \in \mathbb{N}_0}$ given by
$$ 
\Delta^{(\Lambda)}_j :=
\begin{cases} 
 S_{\lambda_0}, \quad \textrm{if} \ j=0\\
 S_{\lambda_j} - S_{\lambda_{j-1}}, \quad \textrm{if} \ j \in \mathbb{N} 
\end{cases} 
$$
where, for $N \in \mathbb{N}$, $S_N$ denotes the multiplier operator acting on functions over $\mathbb{T}$ with symbol $\chi_{ \{ -N +1, \cdots, N -1 \}}$. Define the Littlewood-Paley square function $S^{(\Lambda)} (g)$ of a trigonometric polynomial $g$ by
$$ S^{(\Lambda)} (g) : = \Bigg(\sum_{j \in \mathbb{N}_0} \big| \Delta^{(\Lambda)}_j (g) \big|^2 \Bigg)^{1/2}.$$

It is well-known that in the case where $\Lambda = (\lambda_j)_{j \in \mathbb{N}_0}$ is a lacunary sequence in $\mathbb{N}$, namely the ratio $\rho_{\Lambda}:= \inf_{j \in \mathbb{N}_0} ( \lambda_{j+1}/ \lambda_j )$ is greater than $1$, the Littlewood-Paley square function $S^{(\Lambda)}$ can be extended as a sublinear $L^p (\mathbb{T})$ bounded operator for  $1<p <\infty$; see \cite{Edwards_Gaudry} or \cite{Zygmund_book} for the periodic case and \cite{Singular_integrals} for the Euclidean case.

 In 1989, in \cite{Bourgain_89}, Bourgain proved that if $\Lambda_D = (2^j)_{j \in \mathbb{N}_0}$, then the $L^p (\mathbb{T}) \rightarrow L^p (\mathbb{T})$ operator norm of the Littlewood-Paley square function $S^{(\Lambda_D)}$ behaves like $(p-1)^{-3/2}$ as $p \rightarrow 1^{+}$, namely
 \begin{equation}\label{Bourgain_est}
 \| S^{(\Lambda_D)} \|_{L^p (\mathbb{T}) \rightarrow L^p (\mathbb{T})} \sim (p-1 )^{-3/2} \quad (1<p<2).
 \end{equation}
 Other proofs of the aforementioned theorem of Bourgain were obtained by the author in \cite{Bakas} and by Lerner in \cite{Lerner}. 
 
In 1992, in \cite{Pichorides}, Pichorides showed  that if we restrict ourselves to the analytic Hardy spaces, then one has the improved behaviour $(p-1)^{-1}$ as $p \rightarrow 1^+$. More specifically, Pichorides proved in \cite{Pichorides} that if $\Lambda = (\lambda_j)_{j \in \mathbb{N}_0}$ is a lacunary sequence of positive integers, then one has
\begin{equation}\label{Pichorides_result}
\sup_{\substack{ \| f \|_{L^p (\mathbb{T}^d)} = 1 \\ f \in H^p_A (\mathbb{T}^d)} } \big\| S^{(\Lambda)} (f) \big\|_{L^p (\mathbb{T})} \sim \frac{1}{p-1 } \quad (1<p<2),
\end{equation}
where the implied constants in \eqref{Pichorides_result} depend only on the lacunary sequence $\Lambda$ and not on $p$. As remarked by Pichorides, see Remark (i) in \cite[Section 3]{Pichorides}, if $\Lambda = (\lambda_j)_{j \in \mathbb{N}_0}$ is a lacunary sequence in $\mathbb{N}$ with ratio $\rho_{\Lambda}  \in (1,2)$, then the argument in \cite{Pichorides} yields that for fixed $1<p <2$ the implied constant in the upper estimate in \eqref{Pichorides_result} is $O ( ( \rho_{\lambda} - 1 )^{-2} )$. 

Our main purpose in this paper is to solve the aforementioned problem implicitly posed by Pichorides in \cite{Pichorides} and more specifically, our goal is to improve the exponent $r=2$ in $(\rho_{\Lambda} -1)^{-2}$  obtained in \cite{Pichorides} to the following optimal estimate.

\begin{theo}\label{lac}
There exists an absolute constant $C_0>0$ such that for every $1<p <2$ and for every lacunary sequence $\Lambda  = (\lambda_j)_{j \in \mathbb{N}_0}$ in $\mathbb{N}$ with ratio $\rho_{\Lambda} \in (1,2)$ one has
\begin{equation}\label{main_result}
\sup_{\substack{ \| f \|_{L^p (\mathbb{T})} = 1 \\ f \in H^p_A (\mathbb{T})}} \big\| S^{(\Lambda)} (f) \big\|_{L^p (\mathbb{T})} \leq \frac{C_0}{p-1} (\rho_{\Lambda} - 1)^{-1/2}.
\end{equation}
\end{theo}
 
The exponent $r=1/2$ in $(\rho_{\Lambda} - 1)^{-1/2}$ in \eqref{main_result} is optimal in the sense that for every given $\lambda $ ``close'' to $1^+$, one can construct a lacunary sequence $\Lambda$ in $\mathbb{N}$ with ratio $\rho_{\Lambda}  \sim \lambda$ and choose a $p $ close to $1^+$ such that
\begin{equation}\label{lower_bound}
\sup_{\substack{ \| f \|_{L^p (\mathbb{T})} = 1 \\ f \in H^p_A (\mathbb{T}) }}  \big\| S^{(\Lambda)} (f) \big\|_{L^p (\mathbb{T})} \gtrsim \frac{1}{p-1} (\lambda - 1 )^{-1/2} \sim \frac{1}{p-1} (\rho_{\Lambda} - 1)^{-1/2}.
\end{equation}
Note also that for each fixed $1 < p < 2 $, Theorem \ref{lac} implies that  for every lacunary sequence $\Lambda$ in $\mathbb{N}$ with ratio  $ \rho_{\Lambda} \in (1, p)$ one has
 $$ \sup_{\substack{ \| f \|_{L^p (\mathbb{T})} = 1 \\ f \in H^p_A (\mathbb{T})}} \big\| S^{(\Lambda)} (f) \big\|_{L^p (\mathbb{T})} \lesssim (\rho_{\Lambda} - 1)^{-3/2} $$
 and this estimate, as remarked in \cite{Pichorides}, is best possible in general; see also Remark \ref{Pich_l_b} below.
 
Furthermore, we also establish  the sharp behaviour of the $L^p (\mathbb{T}) \rightarrow L^p (\mathbb{T}) $ operator norm of $S^{(\Lambda)}$ in terms of the ratio $\rho_{\Lambda}$ of $\Lambda$. In other words, the following non-dyadic version of the aforementioned result of Bourgain \eqref{Bourgain_est} is obtained in this paper.

\begin{theo}\label{lac_2}
There exists an absolute constant $C_0 > 0$ such that for every $1 < p < 2$ and for every lacunary sequence $\Lambda = (\lambda_j)_{j \in \mathbb{N}_0}$ in $\mathbb{N}$ with ratio $\rho_{\Lambda}  \in (1,2) $ one has
\begin{equation}\label{main_result_2}
 \big\| S^{(\Lambda)} \big\|_{L^p (\mathbb{T}) \rightarrow L^p (\mathbb{T})} \leq \frac{C_0}{(p-1)^{3/2}}   (\rho_{\Lambda} - 1)^{-1/2} .
\end{equation}
\end{theo}

Moreover, as in Theorem \ref{lac}, the exponent $r=1/2$ in $(\rho_{\Lambda} - 1)^{-1/2}$ in $\eqref{main_result_2}$ cannot be improved in general.

The proofs of \eqref{main_result} and \eqref{main_result_2} are based on the observation that it suffices to show
\begin{equation}\label{main_obs}
 \big\| S^{(\Lambda)} (f) \big\|_{L^{1, \infty} (\mathbb{T})} \lesssim (\rho_{\Lambda} -1)^{-1/2} \| f \|_X ,
 \end{equation}
where for the case of Theorem \ref{lac} one takes $X$ to be the real Hardy space $H^1 (\mathbb{T})$ and for the case of Theorem \ref{lac_2}, one takes $X$ to be the Orlicz space $L\log^{1/2} L (\mathbb{T})$. Indeed, regarding the proof of Theorem \ref{lac}, having established the aforementioned weak-type inequality, one then argues as in \cite{BRS}. More precisely, \eqref{main_result} is obtained by using \eqref{main_obs} for $X=H^1 (\mathbb{T})$, the trivial estimate $\big\| S^{(\Lambda)} \big\|_{L^2 (\mathbb{T}) \rightarrow L^2 (\mathbb{T})} = 1$ and  a result of Kislyakov and Xu on Marcinkiewicz-type interpolation between analytic Hardy spaces \cite{KX}. Regarding Theorem \ref{lac_2}, having shown \eqref{main_obs} for $X = L \log^{1/2} L (\mathbb{T})$, the proof of \eqref{main_result_2} is  obtained by arguing as in \cite{Bakas}, namely  by first interpolating between \eqref{main_obs} for $X = L \log^{1/2} L (\mathbb{T})$ and $\big\| S^{(\Lambda)}   \big\|_{L^2 (\mathbb{T}) \rightarrow L^2 (\mathbb{T})} = 1$ and then using Tao's converse extrapolation theorem \cite{Tao}. 

The proof of \eqref{main_obs} for $X = H^1 (\mathbb{T})$ and $X= L \log^{1/2} L (\mathbb{T})$ can be obtained by using the arguments of Tao and Wright \cite{TW} that establish the endpoint mapping properies of general Marcinkiewicz multiplier operators ``near'' $L^1 (\mathbb{R})$. However, we remark that for the case of Theorem \ref{lac}, namely to prove \eqref{main_obs} for $X=H^1 (\mathbb{T})$, one can just use a version of Stein's classical multiplier theorem on Hardy spaces \cite{Stein_1, Stein_2} obtained by Coifman and Weiss in \cite{CW}. 

Furthermore, an adaptation of the aforementioned argument to the Euclidean setting, where one uses the work of Tao and Wright \cite{TW} combined with a theorem of Peter Jones \cite{Peter_Jones} on Marcinkiewicz-type decomposition for functions in analytic Hardy spaces on the real line (instead of the theorem of Kislyakov and Xu mentioned above), gives a Euclidean version of Theorem \ref{lac}. Moreover, by using the work of Lerner \cite{Lerner}, one obtains a variant of Theorem \ref{lac_2} to the Euclidean setting as well as an alternative proof of Theorem \ref{lac_2}. 

At this point, it is worth noting that, given any strictly increasing sequence $\Lambda =  (\lambda_j)_{j \in \mathbb{N}_0}$ of positive integers, if $2 <p < \infty$ then $S^{(\Lambda)}$ has an $L^p (\mathbb{T}) \rightarrow L^p (\mathbb{T})$ operator norm that is independent of $\Lambda$. More specifically, Bourgain, by using duality and his extension \cite{Bourgain_85} of Rubio de Francia's theorem \cite{RdF}, proved  in \cite{Bourgain_89} that for every strictly increasing sequence $\Lambda $ of positive integers the $L^p (\mathbb{T}) \rightarrow L^p (\mathbb{T})$ operator norm of $S^{(\Lambda)}$ ($2<p<\infty$) behaves like
\begin{equation}\label{Bourgain_result}
 \big\|  S^{(\Lambda)} \big\|_{L^p (\mathbb{T}) \rightarrow L^p (\mathbb{T})} \sim p \quad (p \rightarrow \infty)
 \end{equation}
and the implied constants in \eqref{Bourgain_result} do not depend on $\Lambda$. In particular, if $\Lambda$ is a a lacunary sequence in $\mathbb{N}$, then the implied constants in \eqref{Bourgain_result} are independent of $\rho_{\Lambda}$. Moreover, it is well-known that, in general, Littlewood-Paley square functions formed with respect to arbitrary strictly increasing sequences might not be bounded on $L^p$ for $1 \leq p <2$; see e.g. \cite{Carleson}. 
For more details on Littlewood-Paley square functions of Rubio de Francia type, see \cite{Lacey} and the references therein.

The present paper is organised as follows. In Section \ref{Background}, we give some notation and background. In Section \ref{Proof}, we give a proof of Theorem \ref{lac} and present its optimality in the sense explained above. In Section \ref{Proof_2} we prove Theorem \ref{lac_2} and in Section \ref{finite_union} we extend our results to Littlewood-Paley square functions formed with respect to finite unions of lacunary sequences. In Section \ref{higher} we  extend \eqref{main_result} and \eqref{main_result_2} to higher dimensions and in the last section of this paper we obtain variants of \eqref{main_result} and \eqref{main_result_2} to the Euclidean setting. 

%%%%%%%%%%%%%%%%%%%%%%%%%%%%
%%%%%%%%%%%%%%%%%%%%%%%%%%%%
%%%%%%%%%%%%%%%%%%%%%%%%%%%%
%%%%%%%%%%%%%%%%%%%%%%%%%%%%
\section{Notation and Background}\label{Background}

%%%%%%%%
%%%%%%%%
\subsection{Notation} We denote the set of integers by $\mathbb{Z}$, the set of natural numbers  by $\mathbb{N}$ and the set of non-negative integers by $\mathbb{N}_0$. The real line is denoted by $\mathbb{R}$ and the complex plane by $\mathbb{C}$. We identify strictly increasing sequences of positive integers with subsets of $\mathbb{N}$ in a standard way.

If $x \in \mathbb{R}$, then $\lceil x \rceil $ stands for the unique integer such that $\lceil x \rceil -1 < x \leq \lceil x \rceil$ and $\lfloor x \rfloor$ denotes the integer part of $x$ namely, $ \lfloor x \rfloor$ is the unique integer satisfying $ \lfloor x \rfloor \leq x < \lfloor x \rfloor +1$. 

The notation $| \cdot |$ is used for either the one-dimensional Lebesgue measure of a Lebesgue measurable set $A \subseteq \mathbb{R}$ or for the modulus of a complex number $a \in \mathbb{C}$. 

The logarithm of $x>0$ to the base $\lambda >0$ is denoted by $\log_{\lambda} (x)$. If $\lambda = e$, we   write  $ \log x $.

If $A$ is a finite set, then $\# (A)$ denotes the number of its elements. Given $a,b \in \mathbb{Z}$ with $a<b$,  $\{ a, \cdots, b \}$ denotes the set $[a,b] \cap \mathbb{Z}$ and will occasionally be referred to as an ``interval'' in $\mathbb{Z}$.  
Moreover, a function $m : \mathbb{Z} \rightarrow \mathbb{R}$ is said to be ``affine'' in $\{ a, \cdots, b\}$ if, and only if, there exists a function $\mu : \mathbb{R} \rightarrow \mathbb{R}$ such that $\mu (n) = m (n)$ for all $n \in \{ a, \cdots, b\}$ and $\mu$ is affine in $[a,b]$. 

If there exists an absolute constant $C>0$ such that $A \leq C B$, we shall write $A \lesssim B$ or $B \gtrsim A$ and say that $A$ is $O (B)$. If $C>0$ depends on a given parameter $\eta$, we shall write $A \lesssim_{\eta} B$. If $A \lesssim B$ and $B \lesssim A$, we write $A \sim B$. Similarly, if $A \lesssim_{\eta} B$ and $B \lesssim_{\eta} A$, we write $A \sim_{\eta} B$.

If $J$ is an arc of the torus $\mathbb{T} := \mathbb{R}/(2 \pi \mathbb{Z}) $ with $|J| < \pi$, then $2J$ denotes the arc that is concentric to $J$ with length equal to $2|J|$. 

We identify functions on the torus $\mathbb{T}$ with $2\pi$-periodic functions defined over the real line. The Fourier coefficient of a function  $f$ in $ L^1 (\mathbb{T}^d)$ at $(n_1, \cdots, n_d) \in \mathbb{Z}^d$ is given by 
$$\widehat{f} (n_1, \cdots, n_d) := (2 \pi)^{-d} \int_{\mathbb{T}^d} f(x_1, \cdots, x_d) e^{-i (n_1 x_1 +\cdots + n_d x_d) } d x_1 \cdots d x_d .$$
If $f \in L^1 (\mathbb{T}^d)$ is such that $\mathrm{supp} (\widehat{f})$ is finite, then $f$ is said to be a trigonometric polynomial on $\mathbb{T}^d$. If $f$ is a trigonometric polynomial on $\mathbb{T}^d$ such that $\mathrm{supp} (\widehat{f}) \subseteq \mathbb{N}^d_0$, then we say that $f$ is an analytic trigonometric polynomial on $\mathbb{T}^d$. Similarly, if $f$ is a Schwartz function on $\mathbb{R}^d$ then its Fourier transform at $(\xi_1, \cdots, \xi_d) \in \mathbb{R}^d$ is given by
$$\widehat{f} (\xi_1, \cdots, \xi_d) := (2 \pi)^{-d} \int_{\mathbb{R}^d} f(x_1, \cdots, x_d) e^{-i (\xi_1 x_1 +\cdots + \xi_d x_d)} d x_1 \cdots d x_d.$$

Vector-valued functions on $\mathbb{T}$ are denoted by $\underline{F}$. Moreover, for $0 < p < \infty$, we use the notation 
$$ \| \underline{F} \|_{L^p (\mathbb{T}; \ell^2 (\mathbb{N}))} := (2 \pi)^{-1/p} \Big(  \int_{[-\pi,\pi)} \| \underline{F} (x) \|^p_{\ell^2 (\mathbb{N})} dx \Big)^{1/p}.$$

If $(X, \mathcal{A}, \mu)$ is a measure space, we use the standard notation 
$$ \| f \|_{L^{1,\infty} (X)} : = \sup_{\lambda>0} \big\{ \lambda \cdot \mu \big( \{ x \in X : |f(x)| > \lambda \} \big)  \big\} .$$

%%%%%%%%
%%%%%%%%
\subsection{Multipliers}\label{multipliers}
If $m \in \ell^{\infty} (\mathbb{Z}^d)$, then $T$ is said to be a multiplier operator with associated multiplier $m$ if, and only if, for every trigonometric polynomial $f$ on $\mathbb{T}^d$ one has the representation
$$ T (f) (x_1, \cdots, x_d) = \sum_{(n_1, \cdots, n_d) \in \mathbb{Z}^d} m(n_1, \cdots, n_d) \widehat{f} (n_1, \cdots, n_d) e^{ i(n_1 x_1+\cdots + n_d x_d) }$$
for every $(x_1, \cdots, x_d) \in \mathbb{T}^d$. In this case, we also say that $m $ is the symbol of $T$ and we write $T = T_m$. 

For $ j \in \{1, \cdots, d\}$, if $T_j$ is the multiplier operator acting on functions over $\mathbb{T}$ with symbol $m_j \in \ell^{\infty} (\mathbb{Z}) $, then $ T_1 \otimes \cdots \otimes T_d$ denotes the multiplier operator acting on functions over $\mathbb{T}^d$ whose associated symbol is given by $m (n_1, \cdots, n_d) = m_1 (n_1) \cdots m_d (n_d)$ for all $(n_1, \cdots, n_d) \in \mathbb{Z}^d$. 

Multiplier operators acting on functions over Euclidean spaces are defined similarly. Given a bounded function $m$ on  $\mathbb{R}$, we denote by $T_{m|_{\mathbb{Z}}}$ the multiplier operator acting on functions defined over the torus with associated symbol given by $\mu = m |_{\mathbb{Z}}$, i.e. $\mu : \mathbb{Z} \rightarrow \mathbb{C}$ and $\mu (n) = m(n)$ for all $n \in \mathbb{Z}$.

%%%%%%%%
%%%%%%%%
\subsection{Function spaces} The real Hardy space $H^1 (\mathbb{T})$ is defined to be the class of all functions $f \in L^1 (\mathbb{T})$ such that $H (f ) \in L^1 (\mathbb{T})$, where $H$ denotes the periodic Hilbert transform. For $f \in H^1 (\mathbb{T})$, we set $\| f \|_{H^1 (\mathbb{T})} : = \| f \|_{L^1 (\mathbb{T})} + \| H(f) \|_{L^1 (\mathbb{T})}$. One defines the real Hardy space $H^1 (\mathbb{R})$ on the real line in an analogous way.
 
It is well-known that $H^1(\mathbb{T})$ admits an atomic decomposition. More specifically, following \cite{CW}, a function $a$ is said to be an atom in $H^1(\mathbb{T})$ if it is either the constant function $a =  \chi_{\mathbb{T}}$ or there exists an arc $I$ in $\mathbb{T}$ such that $\mathrm{supp} (a) \subseteq I$, $\int_I a (x) dx = 0$, and $\| a \|_{L^2 (\mathbb{T})} \leq | I |^{-1/2}$. The characterisation of $H^1 (\mathbb{T})$ in terms of atoms asserts that  $f \in H^1 (\mathbb{T})$ if, and only if, there exists a sequence of atoms $(a_k)_{k \in \mathbb{N}}$ in $H^1(\mathbb{T})$ and a sequence $(\mu_k)_{k \in \mathbb{N}} \in \ell^1 (\mathbb{N})$  such that
$$ f = \sum_{k \in \mathbb{N}} \mu_k a_k ,$$ 
where the convergence is in the $H^1 (\mathbb{T})$-norm and moreover, if we define
$$ \| f \|_{H^1_{\mathrm{at}} (\mathbb{T})} := \inf \Bigg\{ \sum_{ k \in \mathbb{N} } |\mu_k| : f = \sum_{k \in \mathbb{N} } \mu_k a_k, \ a_k\ \mathrm{are}\  \mathrm{atoms}\ \mathrm{in}\ H^1 (\mathbb{T}) \Bigg\}, $$
then there exist absolute constants $c'_1, c_1 >0$ such that
\begin{equation}\label{equivalence}
c'_1 \| f \|_{H^1_{\mathrm{at}} (\mathbb{T})} \leq \| f \|_{H^1 (\mathbb{T})} \leq c_1 \| f \|_{H^1_{\mathrm{at}} (\mathbb{T})}.
\end{equation}
For more details on real Hardy spaces, see \cite{CW}, \cite{GR}, \cite{Grafakos_modern} and the references therein.

For $0< p < \infty$, one defines the $d$-parameter analytic Hardy space $H^p_A (\mathbb{R}^d)$ as follows. A function $f \in L^p (\mathbb{R}^d)$ belongs to $H^p_A (\mathbb{R}^d)$ if, and only if, there exists an analytic function $\widetilde{f}$ on $(\mathbb{R}_+^2)^d$ satisfying the condition
$$ \sup_{y_1, \cdots, y_d >0 } \int_{\mathbb{R}^d} \big| \widetilde{f} (x_1 + i y_1, \cdots, x_d + i y_d) \big|^p d x_1 \cdots d x_d < \infty $$
so that $f$ equals a.e. to the limit of $\widetilde{f}$ as one approaches the boundary $\mathbb{R}^d$ of $(\mathbb{R}_+^2)^d$, where $\mathbb{R}_+^2 := \{ x + iy : x \in \mathbb{R}, y > 0\}$. The Hardy space $H^{\infty}_A (\mathbb{R}^d)$ is the class of all functions $f \in L^{\infty} (\mathbb{R}^d)$ that are boundary values of  bounded analytic functions $\widetilde{f}$ on $( \mathbb{R}_+^2)^d$. One defines $H^p_A (\mathbb{T}^d)$ in an analogous way ($0 < p \leq \infty$). Also, it is well-known that for $1 \leq p \leq \infty$ the Hardy space $H^p_A (\mathbb{T}^d)$ coincides with the class of all functions $f \in L^p (\mathbb{T}^d)$ such that $\mathrm{supp} (\widehat{f} ) \subseteq \mathbb{N}^d_0 $. It thus follows that, in the one-dimensional case, $ H^1_A (\mathbb{T}) \subseteq H^1 (\mathbb{T})$ and $ \| f \|_{H^1 (\mathbb{T})} = 2 \| f \|_{L^1 (\mathbb{T})}$ for all $f \in H^1_A (\mathbb{T})$. Moreover, the class of analytic trigonometric polynomials on $\mathbb{T}^d$ is dense in $(H^p_A (\mathbb{T}^d), \| \cdot \|_{L^p (\mathbb{T}^d)})$ for all $0< p < \infty$.  
For more details on one-dimensional analytic Hardy spaces, we refer the reader to the book \cite{Duren}.

For $r>0$, $L \log^r L (\mathbb{T})$ denotes the class of all measurable functions $f$ on the torus satisfying 
$$ \int_{\mathbb{T}} |f(x)| \log^r (e + |f(x)|) dx < \infty .$$
If we equip $L \log^r L (\mathbb{T})$ with the norm 
$$ \| f \|_{L \log^r L (\mathbb{T})} := \inf \Bigg\{\lambda> 0: (2 \pi)^{-1} \int_{\mathbb{T}} \Phi_r (\lambda^{-1} |f(x)|) dx \leq 1 \Bigg\}, $$ 
where $\Phi_r (t) := t [1+\log (1+t)]^r$ ($t \geq 0$), then $ (L \log^r L (\mathbb{T}),\| \cdot \|_{L \log^r L (\mathbb{T})} )$ becomes a Banach space. For more details on Orlicz spaces, see \cite{K-R}.

%%%%%%%%
%%%%%%%%

\subsection{Maximal functions} We denote the centred Hardy-Littlewood maximal operator in the periodic setting by $M_c$, namely if $f$ is a measurable function over $\mathbb{T}$ then one has
$$ M_c (f) (x) : = \sup_{0 < r \leq \pi} r^{-1} \int_{|t| <r} |f(x-t)| dt, \quad x \in \mathbb{T} .$$
The symbol $M_d$ stands for the centred discrete Hardy-Littlewood maximal operator given by
$$ M_d  (a) (n): = \sup_{N \in \mathbb{N}} N^{-1} \sum_{|k| <N} |a(n-k)| \quad (n \in \mathbb{Z})$$
for any function $a : \mathbb{Z} \rightarrow \mathbb{C} $. 

It is well-known that $M_c$ is of weak-type $(1,1)$ and bounded on $L^p (\mathbb{T})$ for every $1< p \leq \infty$; see Chapter I in \cite{Singular_integrals} or Section 13 in Chapter I in \cite{Zygmund_book}. Analogous bounds hold for $M_d$; see \cite{S-W}.

%%%%%%%%
%%%%%%%%

\subsection{Khintchine's inequality} In several parts of this paper, we pass from square functions of the form $(\sum_j |s_j|^2 )^{1/2} $ to corresponding families of  functions $ \sum_j \pm s_j $ and vice versa by using a standard  randomisation argument involving Khintchine's inequality for powers $p \in [1,2)$. 

Recall that given a probability space $(\Omega, \mathcal{A}, \mathbb{P})$ and a countable set of indices $F$, a sequence $(r_n)_{n \in F}$ of independent random variables on $(\Omega, \mathcal{A}, \mathbb{P})$ satisfying $\mathbb{P} (r_n = 1) = \mathbb{P}( r_n  = -1)=1/2$, $n \in F$, is said to be a sequence of Rademacher functions on $\Omega$ indexed by $F$. Then, Khintchine's inequality asserts that for every finitely supported complex-valued sequence $(a_{j_1, \cdots, j_d} )_{j_1, \cdots, j_d \in F} $  one has
\begin{equation}\label{Khintchine}
 \Bigg\| \sum_{j_1, \cdots, j_d \in F} a_{j_1, \cdots, j_d} r_{j_1} \cdots r_{j_d} \Bigg\|_{L^p (\Omega^d)} \sim \Bigg( \sum_{j_1, \cdots, j_d \in F} |a_{j_1, \cdots, j_d}|^2 \Bigg)^{1/2} 
\end{equation}
for all $0 < p < \infty$, where the implied constants do not depend on $(a_{j_1, \cdots, j_d} )_{j_1, \cdots, j_d \in F}$.  In the special case where $p$ is ``close'' to $1^+$, for instance when $p \in [1,2)$, the implied constants in \eqref{Khintchine} can be taken to be independent of $p \in [1,2)$; see Appendix D in \cite{Singular_integrals}. 

%%%%%%%%
%%%%%%%%

%%%%%%%%%%%%%%%%%%%%%%%%%%%%
%%%%%%%%%%%%%%%%%%%%%%%%%%%%
%%%%%%%%%%%%%%%%%%%%%%%%%%%%
%%%%%%%%%%%%%%%%%%%%%%%%%%%%
 \section{Proof of Theorem \ref{lac}}\label{Proof} 

Let  $\Lambda  = (\lambda_j)_{j \in \mathbb{N}_0}$ be a lacunary sequence in $\mathbb{N}$ with ratio $\rho_{\Lambda} \in (1,2)$. To prove Theorem \ref{lac}, we shall first establish the weak-type inequality \eqref{main_obs} for $X=H^1 (\mathbb{T})$ and then use a Marcinkiewicz-type interpolation argument for analytic Hardy spaces on the torus. 
 
As mentioned in the introduction, the proof of \eqref{main_obs} for $X=H^1 (\mathbb{T})$ can  be obtained  by using either the work of Tao and Wright \cite{TW} or a classical result of Coifman and Weiss \cite{CW} and more specifically, by using the argument of Coifman and Weiss that establishes \cite[Theorem (1.20)]{CW}. As the former approach will be used in the proof of Theorem \ref{lac_2} in Section \ref{Proof_2}, we shall present here a proof of the desired weak-type inequality that uses arguments of \cite{CW}. 

To obtain the desired weak-type inequality following the latter method, we shall consider an appropriate sequence of ``smoothed-out'' replacements $ ( \widetilde{\Delta}^{(\Lambda)}_j )_{j \in \mathbb{N}_0}$ of $(\Delta^{(\Lambda)}_j)_{j \in \mathbb{N}_0}$ satisfying 
$$  \Delta^{(\Lambda)}_j \widetilde{\Delta}^{(\Lambda)}_j = \Delta^{(\Lambda)}_j .$$
The sequence of operators $ ( \widetilde{\Delta}^{(\Lambda)}_j )_{j \in \mathbb{N}_0}$ is defined as follows.
For $j = 0$, let $\widetilde{\Delta}^{(\Lambda)}_0$ denote the multiplier operator whose symbol $m^{(\Lambda)}_0 : \mathbb{Z} \rightarrow \mathbb{R}$ is such that:
\begin{itemize}
\item $ \mathrm{supp} \big(  m^{(\Lambda)}_0 \big) = \{ - \lambda_1 +1, \cdots, \lambda_1 -1 \} $. 
\item $ m^{(\Lambda)}_0 (n) = 1 $ for all $n \in \{ -\lambda_0, \cdots, \lambda_0 \} $.
\item $ m^{(\Lambda)}_0 $ is ``affine'' in $\{ - \lambda_1, \cdots, -\lambda_0 \} $ and in $ \{ \lambda_0, \cdots, \lambda_1 \}$.
\end{itemize}
For $j \in \mathbb{N}$, define $ \widetilde{\Delta}^{(\Lambda)}_j $ to be the multiplier operator whose associated multiplier $m^{(\Lambda)}_j : \mathbb{Z} \rightarrow \mathbb{R}$ is even and satisfies:
\begin{itemize}
\item $ \mathrm{supp} \big(  m^{(\Lambda)}_j \big) =  \{\lambda_{j-2} + 1, \cdots, \lambda_{j+1} -1\}$. 
\item $ m^{(\Lambda)}_j (n) = 1 $ for all $n \in \{ \lambda_{j-1} , \cdots, \lambda_j \} $.
\item $ m^{(\Lambda)}_j $ is ``affine'' in $\{ \lambda_{j-2}, \cdots, \lambda_{j-1} \} $ and in $ \{ \lambda_j, \cdots, \lambda_{j+1} \}$.
\end{itemize} 
 Here, in the case where $j=1$, we make the convention that $\lambda_{-1} := \lfloor \lambda_0 /  \rho_{\Lambda} \rfloor $. 
 
The following lemma is a consequence of the argument of Coifman and Weiss establishing \cite[Theorem (1.20)]{CW}.

%%%%
\begin{lemma}\label{main_lemma} Let $\Lambda = (\lambda_j)_{j \in \mathbb{N}_0}$ be a lacunary sequence in $\mathbb{N}$ with ratio $\rho_{\Lambda} \in (1,2)$. 

If $\big( \widetilde{\Delta}^{(\Lambda)}_j \big)_{j \in \mathbb{N}_0}$ is defined as above and $(\Omega, \mathcal{A}, \mathbb{P})$ is a probability space, for $\omega \in \Omega$, consider the operator $\widetilde{T}^{(\Lambda)}_{\omega}$ given by
$$  \widetilde{T}^{(\Lambda)}_{\omega} : = \sum_{j \in \mathbb{N}_0} r_j (\omega) \widetilde{\Delta}^{(\Lambda)}_j  ,$$
where $(r_j)_{j \in \mathbb{N}_0}$ denotes the set of Rademacher functions on $\Omega$ indexed by $\mathbb{N}_0$.
Then, there exists an absolute constant $A_0 >0$ such that  
\begin{equation}\label{main_estimate}
\big\| \widetilde{T}^{(\Lambda)}_{\omega} (f) \big\|_{L^1 (\mathbb{T})} \leq \frac{A_0}{(\rho_{\Lambda} - 1)^{1/2}} \| f \|_{H^1 (\mathbb{T})}  
\end{equation}
for every choice of $\omega \in \Omega$.
\end{lemma}
%%%%

%%%%
\begin{proof} Let $m_{\omega}^{(\Lambda)}$ denote the symbol of the operator $\widetilde{T}^{(\Lambda)}_{\omega} $  in the statement of the lemma, namely
$$  m_{\omega}^{(\Lambda)} : =  \sum_{j \in \mathbb{N}_0} r_j (\omega) m^{(\Lambda)}_j  , $$
where $m^{(\Lambda)}_j $ are as above, $j \in \mathbb{N}_0$. 
Our first task is to show that $ m_{\omega}^{(\Lambda)}$  satisfies the following ``Mikhlin-type'' condition
\begin{equation}\label{Mikhlin_cond}
\sup_{n \in \mathbb{Z}} \Big\{ | n | | m_{\omega}^{(\Lambda)} (n+1) - m_{\omega}^{(\Lambda)} (n) | \Big\} \leq  C_0 ( \rho_{\Lambda} - 1)^{-1},
\end{equation}
 where $C_0 > 0$ is an absolute constant, independent of $\omega $ and $\Lambda  $. The verification of \eqref{Mikhlin_cond} is elementary. Indeed, to show \eqref{Mikhlin_cond}, note that for every $n \in \mathbb{Z}$ there exist at most $3$ non-zero terms in the sum
 $$ m_{\omega}^{(\Lambda)} (n) =  \sum_{j \in \mathbb{N}_0} r_j (\omega) m^{(\Lambda)}_j (n) $$
and so, it suffices to show that for every $j \in \mathbb{N}_0$ one has
\begin{equation}\label{condition}
\delta_j (n) :=  \big| m_j^{(\Lambda)} (n+1) - m_j^{(\Lambda)} (n) \big| \lesssim \frac{1} { |n| (\rho_{\Lambda} - 1) }  
\end{equation}
for all $n \in \mathbb{Z} \setminus \{ 0\}$, where the implied constant is independent of $n$ and $\Lambda$. To show \eqref{condition}, fix a $j \in \mathbb{N}_0$ and take an $n \in \mathrm{supp} \big( m_j^{(\Lambda)} \big) $. We may assume that $j \in \mathbb{N}$, as the case $j=0$ is treated similarly and gives the same bounds. Suppose first that  $n \in \{- \lambda_{j+1}, \cdots, -\lambda_{j-1} \} \cup \{ \lambda_{j-1}, \cdots, \lambda_{j+1} \}$. Observe that in the subcase where $ \lambda_{j-1} \leq |n| < \lambda_j$ one has  $\delta_j (n) = 0$ and so, \eqref{condition} trivially holds. If we now assume that $\lambda_j \leq |n| \leq \lambda_{j+1}$, then one has
$$ \delta_j (n) = \frac{1}{ \lambda_{j+1} - \lambda_j} \leq \frac{1}{ \lambda_{j+1} (1 - \rho^{-1}_{\Lambda})}  \leq\frac{1}{ |n| (1 - \rho^{-1}_{\Lambda})} \leq \frac{2}{ |n| ( \rho_{\Lambda}  -1 )}, $$
as desired. The case where $n \in \{- \lambda_{j-1}, \cdots, -\lambda_{j-2} \} \cup \{ \lambda_{j-2}, \cdots, \lambda_{j-1} \} $ is handled similarly and also gives $ |n| \delta_j (n) \leq 2 (\rho_{\Lambda} - 1 )^{-1}$. Therefore, \eqref{condition} holds and so, \eqref{Mikhlin_cond} is valid with $C_0 = 6$. 

Consider now an arbitrary non-constant atom $a$ in $H^1 (\mathbb{T})$. Then, the proof of  \cite[Theorem (1.20)]{CW}, together with the estimate \eqref{Mikhlin_cond}, yields that
\begin{equation}\label{main_est}
 \big\| \widetilde{T}^{(\Lambda)}_{\omega} (a) \big\|_{L^1 (\mathbb{T})} \lesssim \big\| m_{\omega}^{(\Lambda)} \big\|_{\ell^{\infty} (\mathbb{Z})} + \big\| m_{\omega}^{(\Lambda)} \big\|_{\ell^{\infty} (\mathbb{Z})}^{1/2} (\rho_{\Lambda} - 1)^{-1/2} ,
\end{equation}
where the implied constant does not depend on $\widetilde{T}^{(\Lambda)}_{\omega}$ and $a$. Indeed, notice that it follows from \cite[(1.16)]{CW} that
$$  \int_{\mathbb{T}} \big| \widetilde{T}^{(\Lambda)}_{\omega} (a) (x) \big| dx  \leq 3 \pi  \Bigg( \int_{\mathbb{T}} \big|\widetilde{T}^{(\Lambda)}_{\omega} (a) (x) \big|^2 dx \Bigg)^{1/4}    \Bigg( \int_{\mathbb{T}} \big|\widetilde{T}^{(\Lambda)}_{\omega} (a) (x) \big|^2 |1- e^{ix}|^2 dx \Bigg)^{1/4} $$
and hence, by using \eqref{condition} and arguing exactly as on pp. 578--579 in \cite{CW}, one deduces that
\begin{align*}
& \Bigg( \int_{\mathbb{T}} \big| \widetilde{T}^{(\Lambda)}_{\omega} (a) (x) \big|^2 dx \Bigg)^{1/4}    \Bigg( \int_{\mathbb{T}} \big| \widetilde{T}^{(\Lambda)}_{\omega} (a) (x) \big|^2 |1 - e^{ix}|^2 dx \Bigg)^{1/4}  \lesssim \big\| m_{\omega}^{(\Lambda)} \big\|_{\ell^{\infty} (\mathbb{Z})} + \\
& \big\| m_{\omega}^{(\Lambda)} \big\|^{1/2}_{\ell^{\infty} (\mathbb{Z})} \| a \|^{1/2}_{L^2 (\mathbb{T})} \Bigg( \frac{1} { (\rho_{\Lambda} - 1)^2} \sum_{|n| \leq \| a \|^2_{L^2 (\mathbb{T})}} \| a \|^{-4}_{L^2 (\mathbb{T})} + \frac{1} { (\rho_{\Lambda} - 1)^2  } \sum_{|n| \geq \| a \|^2_{L^2 (\mathbb{T})}} \frac{1}{n^2}  \Bigg)^{1/4}
\end{align*}
and so, \eqref{main_est} follows. Note that since $ \big\|  m_{\omega}^{(\Lambda)} \big\|_{\ell^{\infty} (\mathbb{Z})} \leq 3$, we deduce from \eqref{main_est} that there exists an absolute constant $D_0 >0$ such that
\begin{equation}\label{CW_cor}
\big\| \widetilde{T}^{(\Lambda)}_{\omega} (a) \big\|_{L^1 (\mathbb{T})} \leq D_0 (\rho_{\Lambda} - 1)^{-1/2}    
\end{equation}
for any non-constant atom $a$ in $H^1 (\mathbb{T})$. Moreover, observe that the constant atom $a_0 = \chi_{\mathbb{T}}$ trivially satisfies \eqref{CW_cor}, since
$$ \| \widetilde{T}^{(\Lambda)}_{\omega} (a_0) \|_{L^1 (\mathbb{T})} = \big| m_{\omega}^{(\Lambda)} (0)   \widehat{a_0} (0) \big|  \leq D'_0 (\rho_{\Lambda} - 1)^{-1/2}, $$
where $D'_0 := \max \{ D_0, 3 \}$. Therefore, \eqref{CW_cor} holds for all atoms in $H^1 (\mathbb{T})$ and hence, by arguing as e.g. on pp. 129--130 in \cite{Grafakos_modern}, we deduce that \eqref{CW_cor}  holds in the whole of $H^1 (\mathbb{T})$ with $A_0 =  c_1 D'_0 $, $c_1$ being the constant in \eqref{equivalence}. 
\end{proof}
%%%%

Having established Lemma \ref{main_lemma}, we are now ready to prove \eqref{main_obs} for $X = H^1 (\mathbb{T})$. Note that by using \eqref{main_obs} for $X = H^1 (\mathbb{T})$ combined with the fact that $ \| g \|_{H^1 (\mathbb{T})} = 2 \| g \|_{L^1 (\mathbb{T})}$ for  $g \in H^1_A (\mathbb{T})$ and the density of analytic trigonometric polynomials in $(H^1_A (\mathbb{T}), \| \cdot \|_{L^1 (\mathbb{T})})$, one deduces that there exists an absolute constant $M_0 >0$ such that
\begin{equation}\label{weak_bound}
\big\| S^{(\Lambda)}  (g) \big\|_{L^{1, \infty} (\mathbb{T})} \leq M_0  \frac{1} {(\rho_{\Lambda} - 1)^{1/2}  }  \| g \|_{L^1 (\mathbb{T})} \quad (g \in H^1_A (\mathbb{T})).
\end{equation}
 Now, in order to prove \eqref{main_obs} for $X = H^1 (\mathbb{T})$, fix an arbitrary trigonometric polynomial $f $ and define $(h_j)_{j \in \mathbb{N}_0}$ by $h_j := \widetilde{\Delta}^{(\Lambda)}_j (f)$, where $\widetilde{\Delta}^{(\Lambda)}_j$ are the ``smoothed-out'' versions of $\Delta^{(\Lambda)}_j$, introduced above.
It follows from Corollary 2.13 on p. 488 in \cite{GR} and the definition of $(h_j)_{j \in \mathbb{N}_0}$ that
\begin{equation}\label{bound_a}
\big\| S^{(\Lambda)}  (f) \big\|_{L^{1, \infty} (\mathbb{T})} \lesssim \Bigg\| \Bigg( \sum_{j \in \mathbb{N}_0} |h_j|^2 \Bigg)^{1/2} \Bigg\|_{L^1 (\mathbb{T})} ,
\end{equation}
where the implied constant does not depend on $f$ and $\Lambda  $. If we fix a probability space $(\Omega, \mathcal{A}, \mathbb{P})$, observe that, by using the definition of $(h_j )_{j \in \mathbb{N}_0}$ together with Khintchine's inequality \eqref{Khintchine} and Fubini's theorem, one has
$$ \Bigg\| \Bigg( \sum_{j \in \mathbb{N}_0} | h_j  |^2 \Bigg)^{1/2} \Bigg\|_{L^1 (\mathbb{T})} \lesssim \int_{\Omega} \| \widetilde{T}^{(\Lambda)}_{\omega}  (f) \|_{L^1 (\mathbb{T})} d \mathbb{P} (\omega), $$
 where the implied constant is independent of $f$ and $ \Lambda $. Here, $ \widetilde{T}^{(\Lambda)}_{\omega} $ denotes the multiplier operator in the statement of Lemma \ref{main_lemma}. Hence, by using the last estimate, \eqref{main_estimate}, and \eqref{bound_a}, we obtain $\eqref{main_obs}$ for $X = H^1 (\mathbb{T})$. 

We shall now employ the following result, which is due to Kislyakov and Xu \cite{KX}. See also \cite[Proposition 1.6]{Bourgain_84} and \cite[Lemma 7.4.2]{Pavlovic}.

%%%%
\begin{lemma}[\cite{KX}]\label{Marcinkiewicz_dec} If $f \in H^1_A (\mathbb{T})$ and $\alpha >0$, then
 there exist functions $g_{\alpha} \in H^1_A (\mathbb{T})$ and $h_{\alpha} \in H^{\infty}_A (\mathbb{T})$ such that $f = g_{\alpha} + h_{\alpha}$ and
\begin{itemize}
\item $\| g_{\alpha} \|_{L^1 (\mathbb{T})} \leq C \int_{ \{|f| > \alpha \}} |f(x)| dx$
\item $ | h_{\alpha} (x) | \leq C \alpha \min \big\{ \alpha |f(x)|^{-1}, \alpha^{-1} |f(x)|\big\}$ for a.e. $x \in \mathbb{T}$,
\end{itemize}
where the constant $C>0$ does not depend on $f$ and $\alpha$.
\end{lemma}
%%%%

To complete the proof of Theorem \ref{lac}, one argues as in \cite{BRS}. More precisely, fix a $p \in (1,2)$ and take an arbitrary $f \in H^p_A (\mathbb{T})$ with $\| f \|_{L^p (\mathbb{T})} = 1$. Assume first that  $p$ is ``close'' to $1^+$, for instance, suppose that $ p \in (1, 3/2]$. 
Since  $f \in H^p_A (\mathbb{T}) \subseteq H^1_A (\mathbb{T})$, by using Lemma \ref{Marcinkiewicz_dec}, we may write
$$ \big\| S^{(\Lambda)} (f) \big\|^p_{L^p (\mathbb{T})} = (2 \pi)^{-1} \int_0^{\infty} p \alpha^{p-1} \Big| \Big\{ x \in \mathbb{T} : S^{(\Lambda)} (f) (x) > \alpha \Big\} \Big| d \alpha \leq I_1 +I_2 ,$$
where
$$ I_1 :=  (2 \pi)^{-1} p  \int_0^{\infty} \alpha^{p-1} \Big| \Big\{ x \in \mathbb{T} : S^{(\Lambda)} (g_{\alpha})  (x) > \alpha/2 \Big\} \Big|  d \alpha$$
and
$$ I_2 :=  (2 \pi)^{-1} p \int_0^{\infty} \alpha^{p-1} \Big| \Big\{ x \in \mathbb{T} : S^{(\Lambda)} (h_{\alpha})(x) > \alpha/2 \Big\} \Big|  d \alpha .$$
To handle $I_1$, we use \eqref{weak_bound}, the properties of $g_{\alpha} \in H^1_A (\mathbb{T})$, Fubini's theorem, and the assumption that $\| f \|_{L^p (\mathbb{T})} = 1$ as follows,
\begin{align*}
 I_1 & \leq   2  p M_0 (2 \pi)^{-1}  (\rho_{\Lambda} - 1)^{-1/2} \int_0^{\infty} \alpha^{p-2} \Bigg[ \int_{\mathbb{T}} |g_{\alpha} (x) | dx \Bigg]  d \alpha \\
 & \leq  2  p M_0 C (2 \pi)^{-1}  (\rho_{\Lambda} - 1)^{-1/2} \int_0^{\infty} \alpha^{p-2} \Bigg[ \int_{\{ |f| > \alpha \}} |f (x) | dx \Bigg]  d \alpha  \\
 &  \leq  C_0 (\rho_{\Lambda} - 1)^{-1/2} \frac{1} {p-1} ,
 \end{align*}
where $C_0  = 4 M_0 C$ with $M_0$ and $ C $ being the constants in \eqref{weak_bound} and in Lemma \ref{Marcinkiewicz_dec}, respectively. To handle $I_2$, we first use the fact that $ \| S^{(\Lambda)} \|_{L^2 (\mathbb{T}) \rightarrow L^2 (\mathbb{T})} =1$ and get
\begin{align*}
 I_2 & \leq  (2 \pi)^{-1} p \int_0^{\infty} \alpha^{p-1} \Big| \Big\{ x \in \mathbb{T} : S^{(\Lambda)}  (h_{\alpha})  (x) > \alpha/2 \Big\} \Big|  d \alpha \\
 &  \leq (2 \pi)^{-1} 4 p \int_0^{\infty} \alpha^{p-3} \Bigg[ \int_{\mathbb{T}} |h_{\alpha}  (x)|^2 dx\Bigg] d \alpha 
 \end{align*}
and then, arguing as e.g. in the proof of \cite[Theorem 7.4.1]{Pavlovic}, we write $ I_2 \leq I'_2 + I''_2,$
where 
$$ I'_2 := 2 \pi^{-1} C p \int_0^{\infty} \alpha^{p-3} \Bigg[ \int_{ \{ |f| \leq \alpha \}} | f (x) |^2 dx\Bigg] d \alpha  $$
and 
$$ I''_2:=  2 \pi^{-1} C p \int_0^{\infty} \alpha^{p+1} \Bigg[ \int_{\{ |f| > \alpha \}} |f  (x)|^{-2} dx\Bigg] d \alpha ,$$
with $C>0$ being the constant in Lemma \ref{Marcinkiewicz_dec}. Hence, by using Fubini's theorem and the assumption that $\| f \|_{L^p (\mathbb{T})} = 1$, we have
$$ I'_2 \leq \frac{C' }{ 2- p} $$
and
$$ I''_2 \leq \frac{C'}{ p+2} ,$$
where $C'>0$ is an absolute constant. Putting all the estimates together, we obtain
$$ \big\| S^{(\Lambda)}  (f) \big\|^p_{L^p (\mathbb{T})} \leq C_0 (\rho_{\Lambda} - 1)^{-1/2} \frac{1} { p-1 }   + \frac{ C' }{ 2- p} + \frac{ C' }{ p+2 }$$
and since $p  \in (1,3/2]$, we deduce that
$$ \big\| S^{(\Lambda)}  (f) \big\|_{L^p (\mathbb{T})} \lesssim \frac{1}{p-1} (\rho_{\Lambda} - 1)^{-1/2} ,$$
where the implied constant does not depend on $f$, $p$, and $\Lambda$. 

We remark that \eqref{main_result} also holds for $p \in (3/2,2)$. To see this, observe that \cite[Theorem 2]{Bourgain_89} implies that, e.g., $\big\| S^{(\Lambda)} \big\|_{L^3 (\mathbb{T}) \rightarrow L^3 (\mathbb{T})} \leq C  $, where $C>0$ does not depend on the lacunary sequence $\Lambda $. Hence, arguing as in the case where $p \in (1,3/2]$, namely by using Lemma \ref{Marcinkiewicz_dec} and, in particular, by interpolating between \eqref{weak_bound} and the $L^3(\mathbb{T}) \rightarrow L^3 (\mathbb{T})$ bound mentioned above, we deduce that \eqref{main_result} is also valid for $p \in (3/2,2)$.
Therefore, the proof of Theorem \ref{lac} is complete. 
 
%%%%%%%%
%%%%%%%%
\subsection{Optimality of Theorem \ref{lac}}\label{Sharpness} 
In this subsection, it is shown that the exponent  $r=1/2$ in $ (\rho_{\Lambda} - 1)^{-1/2} $ in \eqref{main_result} is optimal in the sense that for every $\lambda $ ``close'' to $1^+$ one can exhibit a lacunary sequence $\Lambda$ with ratio $\rho_{\Lambda} \in [ \lambda, \lambda^3) $ and choose a $p = p (\lambda)$ ``close'' to $1^+$ such that \eqref{lower_bound} holds. For this, the idea is to consider lacunary sequences $ \Lambda $ whose terms essentially behave like $ \lambda^{j  q_j} (\lambda - 1)^{-1}$ for all $j \in \mathbb{N}_0$, where $q_j \sim 1$, $j \in \mathbb{N}_0$. To be more precise, we need the following elementary construction.

%%%%
\begin{lemma}\label{construction} For every $\lambda >1$ such that $\lambda^3 <2$, there exists a lacunary sequence $\Lambda = (\lambda_j)_{j \in \mathbb{N}_0}$ of positive integers such that
\begin{equation}\label{bound1}
\frac{1}{\lambda-1} < \lambda_0 < \frac{4}{\lambda-1}
\end{equation}
and
\begin{equation}\label{bound2}
\lambda \leq  \frac{\lambda_{j+1}}{ \lambda_j} < \lambda^3 \quad \textrm{for}\ \textrm{all}\ j \in \mathbb{N}_0.
\end{equation}
In particular, the ratio $\rho_{\Lambda} $ of $\Lambda$ satisfies $\rho_{\Lambda} \in [\lambda, \lambda^3)$.
\end{lemma}
%%%%

%%%%
\begin{proof} Given a $\lambda > 1$ with $\lambda^3 <2$, fix a $\widetilde{\lambda} \in ( \lambda^{3/2},  \lambda^2 )$ and then consider the auxiliary sequence $(\alpha_j)_{j \in \mathbb{N}_0}$ given by $ \alpha_j := \lceil \widetilde{\lambda}^j \rceil $, $j \in \mathbb{N}_0$. Note that it follows from the definition of $(\alpha_j)_{j \in \mathbb{N}_0}$ that
\begin{equation}\label{bound_1}
 \frac{\alpha_{j+1}}{ \alpha_j} \geq \widetilde{\lambda} \cdot \frac{\widetilde{\lambda}^j}{\widetilde{\lambda}^j +1} 
 \end{equation}
for all $j \in \mathbb{N}_0$. Observe that, since $\widetilde{\lambda} > \lambda^{3/2} > \lambda >1$, there exists a $j_0 \in \mathbb{N}$ such that
\begin{equation}\label{bound_2}
    \frac{  \widetilde{\lambda}^{j_0}}{\widetilde{\lambda}^{j_0} +1} \geq \frac{\lambda} {\widetilde{\lambda}} >  \frac{  \widetilde{\lambda}^{j_0-1}}{\widetilde{\lambda}^{j_0 - 1} +1}.
  \end{equation}
  Note also that, since $\widetilde{\lambda} < \lambda^2$, the left-hand side of \eqref{bound_2} implies that 
\begin{equation}\label{bound_3}
\widetilde{\lambda}^j > \frac{1}{ \lambda - 1}  \quad \textrm{for}\ \textrm{all}\ j \geq j_0.
\end{equation}
Moreover, it follows from the right-hand side of \eqref{bound_2} that
\begin{align*}
\widetilde{\lambda}^{j_0 - 1} < \frac{ \lambda} {\widetilde{\lambda} - \lambda}
\end{align*}
and hence, using the facts that $\widetilde{\lambda} > \lambda^{3/2} $ as well as  $1<\lambda < 2$, we deduce that
\begin{equation}\label{bound_4}
\widetilde{\lambda}^{j_0} < \frac{1}{\lambda^{1/2} - 1} < \frac{3}{\lambda - 1}.
\end{equation}
Define now $\Lambda = (\lambda_j)_{j \in \mathbb{N}_0}$ by
$$ \lambda_j := \alpha_{j+j_0} = \big\lceil \widetilde{\lambda}^{j+j_0} \big\rceil \quad \textrm{for} \ j \in \mathbb{N}_0, $$
 $j_0$ being as above. We shall prove that $\Lambda $ satisfies the desired properties \eqref{bound1} and \eqref{bound2}. To prove \eqref{bound1}, note that
 \eqref{bound_3} gives
 $$ \lambda_0 \geq \widetilde{\lambda}^{j_0} > \frac{1}{\lambda -1 }. $$ 
 Moreover, \eqref{bound_4} implies that
 $$ \lambda_0 \leq \widetilde{\lambda}^{j_0} +1 < \frac{3}{\lambda -1} +1 < \frac{4}{\lambda - 1}$$
 and hence, we deduce that $\lambda_0$ satisfies \eqref{bound1}.
 
To prove \eqref{bound2}, note that for all $j \in \mathbb{N}$ one has
$$ \frac{\lambda_{j+1}} {\lambda_j} \geq \widetilde{\lambda} \cdot \frac{\widetilde{\lambda}^{j+j_0}} {\widetilde{\lambda}^{j + j_0} +1} \geq  \widetilde{\lambda} \cdot \frac{\widetilde{\lambda}^{j_0}} {\widetilde{\lambda}^{j_0}+1} \geq \lambda,$$
where we used the fact that the map $t \mapsto t (t+1)^{-1}$ is increasing on $[0,\infty)$ and \eqref{bound_2}. This completes the proof of the left-hand side of \eqref{bound2}. To prove the right-hand side of \eqref{bound2}, note that by the definition of $\Lambda $ one has
$$ \frac{\lambda_{j+1}}{\lambda_j} \leq \frac{ \widetilde{\lambda}^{j+ j_0+1} + 1} {\widetilde{\lambda}^{j + j_0}} = \widetilde{\lambda} + \frac{1}{\widetilde{\lambda}^{j+j_0}}$$
and so, \eqref{bound_3} gives
$$ \frac{\lambda_{j+1}}{\lambda_j} < \widetilde{\lambda} + \lambda - 1.$$
Since $\widetilde{\lambda} + \lambda - 1 < \lambda^2 + \lambda -1$ and the map $t \mapsto t^3 - t^2 - t +1$ is increasing on $[1,\infty)$, the proof of the right-hand side of \eqref{bound2} follows from the last estimate. 
 \end{proof}
%%%%

%%
Given a $\lambda > 1 $ with $\lambda^3 <2$, construct a lacunary sequence $\Lambda  $ in $\mathbb{N}$ as in Lemma \ref{construction} and then consider the corresponding square function $S^{(\Lambda)}$.  

Let $N \in \mathbb{N}$ be such that
\begin{equation}\label{choice}
  N  := \big\lceil e^{4/(\lambda-1)} \big\rceil  
 \end{equation}
 and observe that if we choose $p = \lambda $, then
 \begin{equation}\label{choice_p}
  e^4 \leq N^{p-1} \leq e^5 .
 \end{equation}
Consider now the de la Vall\'ee Poussin kernel  $V_N$ of order $N$, namely
$$ V_N := 2 K_{2N+1} - K_N,$$
$ K_n $ being the Fej\'er kernel of order $n$; $K_n (x) := \sum_{|j| \leq n} \big( 1 - |j|/(n+1) \big) e^{ijx} $, $x \in \mathbb{T}$. As in \cite{BRS}, consider the analytic trigonometric polynomial $f_N$ given by 
$$ f_N (x) : = e^{i (2N + 1) x} V_N (x), \quad x\in \mathbb{T} .$$
Using \eqref{choice_p}, one has
\begin{equation}\label{p_norm}
 \| f_N \|_{L^p (\mathbb{T})} \lesssim 1.
 \end{equation}
Indeed, arguing as on p. 424 in \cite{OSTTW}, observe that $\| f_N \|_{L^1 (\mathbb{T})} \lesssim 1$ and $ \| f_N \|_{L^{\infty} (\mathbb{T})} \lesssim N $ and hence, interpolation implies that $\| f_N \|_{L^p (\mathbb{T})} \lesssim N^{1-1/p}$. Therefore, \eqref{p_norm} follows from \eqref{choice_p}.  

Next, we claim that the set $ A^{(\Lambda)}_N = \big\{ j \in \mathbb{N}_0 : N \leq \lambda_j \leq 2N  \big\} $ is non-empty and 
has cardinality
\begin{equation}\label{cardinality}
\# ( A^{(\Lambda)}_N )  \sim ( \lambda - 1 )^{-1} \sim (\rho_{\Lambda} - 1)^{-1} .
\end{equation}
To prove that $A^{(\Lambda)}_N $ is non-empty, observe that it follows from \eqref{bound1} and \eqref{choice} that 
$$ \lambda_0 < \frac{4}{\lambda - 1} < N $$ and hence, there exists a $k_0 \in \mathbb{N} $ such that 
$$ \lambda_{k_0-1} < N \leq \lambda_{k_0} .$$
Using now the right-hand side of \eqref{bound2}, we obtain
 $$ \lambda_{k_0} < \lambda^3 \lambda_{k_0 - 1} < \lambda^3 N < 2N. $$
It thus follows that $k_0 \in A^{(\Lambda)}_N $ and so, $A^{(\Lambda)}_N $ is a non-empty set of indices. In order to prove \eqref{cardinality}, note that if we set  $ k := \# ( A^{(\Lambda)}_N ) $ and $k_0 \in \mathbb{N}$ is as above, then the definition of $A^{(\Lambda)}_N $ and the left-hand side of \eqref{bound2} give 
$$ 2N \geq \lambda_{ k_0 + k -1} \geq \lambda^{k-1} \lambda_{k_0} \geq \lambda^{k-1} N$$
and so, $ \lambda^{k-1} \leq 2$. Hence, $ k \leq 1+ \log_{\lambda} 2 \sim [ \log \lambda ]^{-1} $ and since $[ \log \lambda ]^{-1} \sim (\lambda - 1)^{-1}$, we get the upper estimate
\begin{equation}\label{upper}
k \lesssim (\lambda - 1)^{-1} .
\end{equation}
By the definitions of $ A^{(\Lambda)}_N $, $k_0$, $k$, and the right-hand side of \eqref{bound2}, we obtain
$$ 2N \leq \lambda_{ k_0 + k + 1} \leq \lambda^{3 (k+1)} \lambda_{k_0 - 1} < \lambda^{3 (k+1)} N$$
and so, $(\lambda - 1)^{-1} \sim \log_{\lambda} 2 < 3(k+1)$. Hence, we also get the lower estimate
\begin{equation}\label{lower}
k \gtrsim (\lambda - 1)^{-1} 
\end{equation}
and therefore, the proof of \eqref{cardinality} is  complete in view of \eqref{upper} and \eqref{lower}.

Going now back to the proof of \eqref{lower_bound}, observe that, thanks to the definition of $f_N$, one has
\begin{equation}\label{Dirichlet}
\Delta^{(\Lambda)}_j (f_N) (x) = \sum_{n= \lambda_{j-1}}^{\lambda_j - 1} e^{inx} 
\end{equation}
for all $j \in A^{(\Lambda)}_N $. Therefore, by using \eqref{Dirichlet} and \eqref{choice}, one deduces that
$$ \big\| \Delta^{(\Lambda)}_j (f_N) \big\|_{L^1 (\mathbb{T})} \sim \log (\lambda_j - \lambda_{j-1}) > \log ( \lambda_{j-1} (\lambda - 1)) \geq \log (N (\lambda - 1)) \sim  \log N $$ 
for all $j \in A^{(\Lambda)}_N$ with $j > k_0 +1$, $k_0$ being as above.  Hence, arguing again as in \cite{Bourgain_89}, by using Minkowski's inequality together with the last estimate and \eqref{cardinality}, we get
\begin{align*}
\big\| S^{(\Lambda)} (f_N) \big\|_{L^p (\mathbb{T})} \geq \Bigg( \sum_{j \in \mathbb{N}_0} \| \Delta^{(\Lambda)}_j (f_N) \|^2_{L^p (\mathbb{T})} \Bigg)^{1/2}  
&\geq \Bigg( \sum_{j \in A^{(\Lambda)}_N} \| \Delta^{(\Lambda)}_j (f_N) \|^2_{L^1 (\mathbb{T})} \Bigg)^{1/2} \\
&\gtrsim \Big( \# ( A^{(\Lambda)}_N ) [ \log  N  ]^2  \Big)^{1/2} \\
&\sim (\rho_{\Lambda} - 1)^{-1/2} \log  N.
 \end{align*}
 It thus follows from the last estimate combined with \eqref{p_norm}  that
 $$ \sup_{\substack{ \| f \|_{L^p (\mathbb{T})} = 1 \\ f \in H^p_A (\mathbb{T})}} \big\| S^{(\Lambda)} (f) \big\|_{L^p (\mathbb{T})} \gtrsim (\rho_{\Lambda} - 1)^{-1/2} \log  N $$
 and this implies the desired bound \eqref{lower_bound}, as \eqref{choice_p} gives $\log N \sim (p-1)^{-1}$.

%%%%
\begin{rmk}\label{Pich_l_b} Observe that for any fixed $ p \in (1,2) $, if one sets $\lambda := p^{1/3}$ and defines $\Lambda = ( \lambda_j )_{j \in \mathbb{N}_0}$, $N $,  and $f_N$ as above, then one has $\lambda \leq \rho_{\Lambda} < \lambda^3 = p$, $\log N \sim (p-1)^{-1} \sim (\lambda - 1)^{-1}$ and the previous argument shows that
 $$ \sup_{\substack{ \| f \|_{L^p (\mathbb{T})} = 1 \\ f \in H^p_A (\mathbb{T}) } } \big\| S^{(\Lambda)} (f) \big\|_{L^p (\mathbb{T})} \gtrsim (\lambda - 1)^{-3/2} \sim (\rho_{\Lambda} - 1)^{-3/2},$$
which is the lower estimate mentioned in Remark (i) in \cite[Section 3]{Pichorides}. Notice that the aforementioned lower bound also shows that the estimate in Theorem \ref{lac} cannot be improved in general.
\end{rmk}
%%%%

%%%%%%%% 
%%%%%%%%
\subsection{A classical inequality of Paley}  
A classical theorem of Paley \cite{Paley} asserts that if $\Lambda = (\lambda_j)_{j \in \mathbb{N}_0}$ is a lacunary sequence in $\mathbb{N}$, then for every $f \in H^1_A (\mathbb{T})$ the sequence $  ( \widehat{f} (\lambda_j)  )_{j \in \mathbb{N}_0} $ is square summable. Moreover, Paley's argument in \cite{Paley} yields that if $\Lambda  $ is a lacunary sequence in $\mathbb{N}$ with ratio $\rho_{\Lambda} \in (1,2)$ then
\begin{equation}\label{Paley_ineq}
\Bigg( \sum_{j \in \mathbb{N}_0} |\widehat{f} (\lambda_j) |^2 \Bigg)^{1/2} \lesssim (\rho_{\Lambda} - 1)^{-1/2}   \| f \|_{L^1 (\mathbb{T})} \quad (f \in H^1_A (\mathbb{T})). 
\end{equation}
Paley's inequality was extended by D. Oberlin's in \cite{Oberlin}; see Corollary on p. 45 in  \cite{Oberlin}.
We remark that by using Lemma \ref{main_lemma}, iteration and multi-dimensional Khintchine's inequality \eqref{Khintchine}, one recovers D. Oberlin's extension of Paley's inequality, namely if $\Lambda_n =  ( \lambda^{(n)}_{j_n}  )_{j_n  \in \mathbb{N}_0}$ is a lacunary sequence in $\mathbb{N}$ with ratio $\rho_{\Lambda_n} \in (1,2)$ for $n \in \{1, \cdots, d\}$, then
\begin{equation}\label{Paley_ineq_d}
\Bigg( \sum_{j_1, \cdots, j_d \in \mathbb{N}_0} \Big| \widehat{f} (\lambda^{(1)}_{j_1}, \cdots, \lambda^{(d)}_{j_d}) \Big|^2 \Bigg)^{1/2} \lesssim_d \Bigg[ \prod_{n=1}^d (\rho_{\Lambda_n} - 1)^{-1/2} \Bigg] \| f \|_{L^1 (\mathbb{T}^d)} \quad 
\end{equation}
for all $f \in H^1_A (\mathbb{T}^d)$. Furthermore, an adaptation of  the argument presented in the previous subsection shows that the exponents $r =1/2$ in $\prod_{n=1}^d (\rho_{\Lambda_n} - 1)^{-1/2}$ in \eqref{Paley_ineq_d} cannot be improved in general.

%%%%%%%%%%%%%%%%%%%%%%%%%%%%
%%%%%%%%%%%%%%%%%%%%%%%%%%%%
%%%%%%%%%%%%%%%%%%%%%%%%%%%%
%%%%%%%%%%%%%%%%%%%%%%%%%%%%
\section{Proof of Theorem \ref{lac_2}}\label{Proof_2}

As mentioned in the introduction, the first step in the proof of Theorem \ref{lac_2} is to establish the weak-type inequality \eqref{main_obs} for $X= L \log^{1/2} L (\mathbb{T})$ and this inequality will be obtained by using the work of Tao and Wright on the endpoint behaviour of Marcinkiewicz multiplier operators acting on functions defined over the real line \cite{TW}.

To be more precise, let $\eta$ be a fixed Schwartz function that is even, supported in $(-8, -1/8 ) \cup ( 1/8, 8)$ and such that $\eta|_{[1/4,4]} \equiv 1$. For $j \in \mathbb{N}$, set $\eta_j (\xi) := \eta (2^{-j} \xi)$ for $\xi \in \mathbb{R}$ and denote by $\widetilde{\Delta}_j$ the periodic multiplier operator whose symbol is given by $\eta_j |_{\mathbb{Z}}$, i.e. $\widetilde{\Delta}_j = T_{\eta_j |_{\mathbb{Z}}}$.
We shall also consider the sequence of functions $( \phi_j )_{j \in \mathbb{N}}$ given by 
$$ \phi_j (x) : =  2^j \big[1+ 2^{2j} \sin^2 (x/2) \big]^{-3/4} , \quad x\in \mathbb{T} .$$
By arguing exactly as on pp. 547--549 in \cite{TW} one deduces that \cite[Proposition 9.1]{TW} implies that for every $f \in L \log^{1/2} L (\mathbb{T})$  there exists a sequence $(F_j )_{ j \in \mathbb{N} }$ of  non-negative functions such that for every $j \in \mathbb{N}$ one has  
\begin{equation}\label{majorisation}
 \big| \widetilde{\Delta}_j (f) (x) \big| \lesssim \big( F_j \ast \phi_j \big) (x) \quad \textrm{for}\ \textrm{all} 
 \ x \in \mathbb{T}
 \end{equation} 
and
\begin{equation}\label{substitute}  \Bigg\| \Bigg(\sum_{j \in \mathbb{N}} |F_j |^2 \Bigg)^{1/2} \Bigg\|_{L^1 (\mathbb{T})} \lesssim 1 + \int_{\mathbb{T}} |f (x)| \log^{1/2} (e + |f(x)|) dx.
\end{equation}
We  omit the details. Notice that \eqref{substitute} can be regarded as a periodic analogue of \cite[Proposition 4.1]{TW}.

Fix now a lacunary sequence $\Lambda = (\lambda_k)_{k \in \mathbb{N}_0}$ of positive integers with $\rho_{\Lambda} \in (1,2)$. Note that it follows from the mapping properties of the periodic Hilbert transform that for every $a , b \in \mathbb{R}$ such that $a < b$, the multiplier operator $T_{\chi_{ \{ a, \cdots, b\} }}$ is of weak-type $(1,1)$ and bounded on $L^p (\mathbb{T})$ for all $1< p <\infty$ with corresponding operator norm bounds that are independent of $a,b$. We may thus assume, without loss of generality, that $\lambda_0 \geq 8$.  

For technical reasons, in order to suitably adapt the relevant arguments of \cite{TW} to the periodic setting and prove that $S^{(\Lambda)}$ satisfies \eqref{main_obs} for $X= L \log^{1/2} L (\mathbb{T})$, it would be more convenient to work with $S^{(\widetilde{\Lambda})}$, where $\widetilde{\Lambda}$ is a strictly increasing sequence in $\mathbb{N}$, which is associated to $\Lambda$ and satisfies the properties of the following lemma.

%%%%
\begin{lemma}\label{refinement} Let $\Lambda = (\lambda_k)_{k \in \mathbb{N}_0}$ be a lacunary sequence in $\mathbb{N}$ with $\rho_{\Lambda} \in (1,2)$ and $\lambda_0 \geq 8$.

There exists a $\Lambda'  \subseteq \mathbb{N}$ such that if we regard $\widetilde{\Lambda} := \Lambda \cup \Lambda' \cup (2^{j+3})_{j \in \mathbb{N}_0}$ as a strictly increasing sequence in $\mathbb{N}$ and write $\widetilde{\Lambda}= (\widetilde{\lambda}_k )_{k \in \mathbb{N}_0}$, then $\widetilde{\lambda}_0 = 8$ and moreover, $\widetilde{\Lambda}$ has the following properties: 
 \begin{enumerate}
 \item For every $k \in \mathbb{N}$ there exists a positive integer $j_k \geq 4$ such that 
 $$ [ \widetilde{\lambda}_{k-1},  \widetilde{\lambda}_k ) \subseteq [ 2^{j_k-1}, 2^{j_k} ) $$
 and  $\widetilde{\lambda}_k - \widetilde{\lambda}_{k-1} \leq 2^{j_k - 3}$.
 \item For all $N \in \mathbb{N}$, one has 
 $$ \#(\widetilde{\Lambda} \cap \{ 2^{N-1}, \cdots, 2^N\}) \leq 9 (\#(\Lambda \cap \{ 2^{N-1}, \cdots, 2^N\}) + 2) . $$ 
 \end{enumerate}
 \end{lemma}
 %%%%
 
 %%%%
 \begin{proof} The desired construction is elementary. First of all, note that if we regard the set $\Lambda_d := \Lambda \cup (2^{j+3})_{j \in \mathbb{N}_0}$ as a strictly increasing sequence in $\mathbb{N}$ and write $\Lambda_d = ( \lambda^{(d)}_k)_{k \in \mathbb{N}_0}$, then
 $ \lambda^{(d)}_0 = \min \{ \lambda_0 , 8 \}  = 8$ and for every $k \in \mathbb{N} $ there exists a unique positive integer $j_k \geq 4$ such that $ I^{(d)}_k : = [\lambda^{(d)}_{k-1}, \lambda^{(d)}_k) \subseteq [ 2^{j_k -1 }, 2^{j_k} )$.
 
 We shall now construct $\Lambda' = \bigcup_{k \in \mathbb{N}_0} \Lambda'_k$, where the sets $ \Lambda'_k \subseteq \mathbb{N}$ are defined inductively as follows.
 \begin{itemize}
 \item For $k=0$, define $\Lambda_0' : = \{ \lambda^{(d)}_0 \} $. 
 \item For $k \in \mathbb{N}$, let  $j_k$ and $I_k^{(d)}$  be as above. \newline
 \textbf{Case 1.} If $| I_k^{(d)} | \leq 2^{j_k -3} $ then $\Lambda'$ contains no terms between $\lambda^{(d)}_{k-1} $ and $ \lambda^{(d)}_k$, namely we set $\Lambda'_k : = \emptyset$. \newline
 \textbf{Case 2.}  If $ | I_k^{(d)} | > 2^{j_k -3} $ then, since $ I_k^{(d)} \subseteq [2^{j_k -1}, 2^{j_k} ) $, there exist positive integers $a^{(1)}_k \leq a^{(2)}_k \leq \cdots \leq a^{(8)}_k  $, suitably chosen, such that the intervals $[\lambda^{(d)}_{k-1}, a^{(1)}_k ]$,  $[a^{(1)}_k , a^{(2)}_k ]$, $\cdots$, $[a^{(7)}_k, a^{(8)}_k]$, and $[a^{(8)}_k, \lambda^{(d)}_k )$ have lengths that are less or equal than $ 2^{j_k -3}$. Here, we make the standard convention that $[a,a] = \{ a \} $ for $a \in \mathbb{R}$. For such a choice of $a^{(1)}_k,   \cdots, a^{(8)}_k  $, we define
 $$ \Lambda'_k : = \{ a^{(1)}_k, \cdots, a^{(8)}_k  \}.$$
 \end{itemize}
If we set $\widetilde{\Lambda} : = \Lambda \cup \Lambda' \cup (2^{j+3})_{j \in \mathbb{N}_0}$ and regard it as a sequence i.e. $\widetilde{\Lambda}= (\widetilde{\lambda}_k )_{k \in \mathbb{N}_0}$, then $\widetilde{\lambda}_0 = \lambda^{(d)}_0 = 8$ and for every $k \in \mathbb{N}$ there exists a $k' \in \mathbb{N}$ such that $[ \widetilde{\lambda}_{k-1}, \widetilde{\lambda}_k ) \subseteq [\lambda^{(d)}_{k'-1} , \lambda^{(d)}_{k'} ) \subseteq [ 2^{j_{k'} -1 }, 2^{j_{k'}} ) $. Hence, property $(1)$ is satisfied for $j_k = j_{k'}$. To verify property $(2)$, observe that by the definition of $\Lambda'$ one has
$$ \#( \Lambda' \cap \{ 2^{N-1}, \cdots, 2^N\} ) \leq 8( \#(\Lambda \cap \{ 2^{N-1}, \cdots, 2^N\}) + 2) $$
and so,
\begin{align*}
\#( \widetilde{\Lambda} \cap \{ 2^{N-1}, \cdots, 2^N\} )   
& \leq  \#( \Lambda' \cap \{ 2^{N-1}, \cdots, 2^N\} ) +  \#(  \Lambda   \cap \{ 2^{N-1}, \cdots, 2^N\} ) +2 \\
& \leq 9 ( \#( \Lambda  \cap \{ 2^{N-1}, \cdots, 2^N\} ) + 2) .
\end{align*}
Hence, the proof of the lemma is complete.  \end{proof}
%%%%
 
Note that if $\widetilde{\Lambda}$ is as in Lemma \ref{refinement}, then one has
\begin{equation}\label{pointwise_square}
 S^{(\Lambda)} (f) (x) \lesssim S^{(\widetilde{\Lambda})} (f) (x) \quad (x \in \mathbb{T})
 \end{equation}
for every trigonometric polynomial $f$ on $\mathbb{T}$. Hence, it is enough to work with the operator $S^{(\widetilde{\Lambda})}$ instead of $S^{(\Lambda)}$ in view of \eqref{pointwise_square}. 

We now fix a trigonometric polynomial $f$ on $\mathbb{T}$. For every given $k \in \mathbb{N}$, we consider $j_k \in \mathbb{N}$ as in Lemma \ref{refinement}, namely $j_k $ is the positive integer satisfying the property $[ \widetilde{\lambda}_{k-1}, \widetilde{\lambda}_k ) \subseteq [ 2^{j_k -1} , 2^{j_k})$ and we then set 
\begin{equation}\label{rep_def}
 \widetilde{F}_k := F_{j_k} ,
 \end{equation}
 where $( F_j )_{j \in \mathbb{N}}$ is the sequence of non-negative functions associated to $f$ such that \eqref{majorisation} and \eqref{substitute} hold. Then, an adaptation of the proof of \cite[Proposition 5.1]{TW} to the periodic setting, where one uses \eqref{majorisation} and \eqref{substitute} instead of \cite[Proposition 4.1]{TW}, yields that
\begin{equation}\label{main_w-t}
\Bigg\| \sum_{k \in \mathbb{N}} r_k (\omega) \Delta_k^{(\widetilde{\Lambda})} (f) \Bigg\|_{L^{1,\infty} (\mathbb{T})} \lesssim  \Bigg\| \Bigg(\sum_{k \in  \mathbb{N}} \big| \widetilde{F}_k \big|^2 \Bigg)^{1/2} \Bigg\|_{L^1 (\mathbb{T})},
\end{equation}
where the implied constant does not depend on $\Lambda$ and $\omega$. As some of the technicalities in the periodic setting are slightly more involved than the ones in the Euclidean case, for the convenience of the reader, in Subsection \ref{details} we present how \eqref{main_w-t} can be obtained by adapting the arguments of Tao and Wright \cite{TW} to our case.   

Assuming now that \eqref{main_w-t} holds, note that it easily follows from the definition \eqref{rep_def} of $ \big( \widetilde{F}_k \big)_{k \in \mathbb{N}_0} $ and the second property of $\widetilde{\Lambda}$ in Lemma \ref{refinement} that there exists an absolute constant $A >0 $ such that
\begin{equation}\label{observation}
  \Bigg(\sum_{k \in  \mathbb{N}} \big| \widetilde{F}_k (x) \big|^2 \Bigg)^{1/2} \leq A   (\rho_{\Lambda} - 1 )^{-1/2} \Bigg(\sum_{j \in  \mathbb{N}} |F_j (x) |^2 \Bigg)^{1/2} \quad \mathrm{for}\ \mathrm{all}\  x \in \mathbb{T}.
\end{equation}
Hence, \eqref{substitute}, \eqref{main_w-t}, and  \eqref{observation}, combined with the fact that $\Delta^{(\widetilde{\Lambda})}_0 = T_{\chi_{\{ -\widetilde{\lambda}_0 +1, \cdots, \widetilde{\lambda}_0 - 1 \}}}$ is of weak-type $(1,1)$  with corresponding operator norm independent of $\Lambda$, imply that  
\begin{equation}\label{final_w-t}
\Bigg\| \sum_{k \in \mathbb{N}_0} r_k (\omega) \Delta_k^{(\widetilde{\Lambda})} (f) \Bigg\|_{L^{1,\infty} (\mathbb{T})} \leq \frac{C_0} { (\rho_{\Lambda} - 1 )^{1/2} }  \Bigg[ 1 + \int_{\mathbb{T}} |f (x)| \log^{1/2} (e + |f(x)|) dx \Bigg],
\end{equation}
where $C_0 >0$ is an absolute constant that is independent of $f$, $\Lambda$, and  $\omega$.

We shall now argue as in \cite{Bakas}. More precisely, by using \eqref{final_w-t} as well as the trivial estimate 
\begin{equation}\label{trivial_2}
 \Bigg\| \sum_{k \in \mathbb{N}_0} r_k (\omega) \Delta_k^{(\widetilde{\Lambda})}  \Bigg\|_{L^2 (\mathbb{T}) \rightarrow L^2 (\mathbb{T})} =1,
 \end{equation}
 then a Marcinkiewicz-type interpolation argument implies that
\begin{equation}\label{final_s-t}
\Bigg\| \sum_{k \in \mathbb{N}_0} r_k (\omega) \Delta_k^{(\widetilde{\Lambda})} (f) \Bigg\|_{L^1 (\mathbb{T})} \leq  \frac{A_0} { (\rho_{\Lambda} - 1 )^{1/2} }  \Bigg[ 1 + \int_{\mathbb{T}} |f (x)| \log^{3/2} (e + |f(x)|) dx \Bigg] ,
\end{equation}
where $A_0 >0$ does not depend on $f$, $\Lambda$, and the choice of $\omega \in \Omega$. Indeed, arguing as in the proof of \cite[Lemma 3.2]{Bakas}, write
$$ \Bigg\| \sum_{k \in \mathbb{N}_0} r_k (\omega) \Delta_k^{(\widetilde{\Lambda})} (f) \Bigg\|_{L^1 (\mathbb{T})} \leq  (2 \pi)^{-1} (  I_1 + I_2 ),  $$
where
$$ I_1 : = \int_0^{\infty} \Big| \Big\{ x \in \mathbb{T} : \Bigg| \sum_{k \in \mathbb{N}_0} r_k (\omega) \Delta_k^{(\widetilde{\Lambda})} \big( f \chi_{ \{ |f| > \alpha \} } \big) (x) \Bigg| > \alpha/2 \Big\} \Big| d \alpha$$
and
$$ I_2 : = \int_0^{\infty} \Big| \Big\{ x \in \mathbb{T} : \Bigg| \sum_{k \in \mathbb{N}_0} r_k (\omega) \Delta_k^{(\widetilde{\Lambda})} \big( f \chi_{\{ |f| \leq \alpha \} } \big) (x) \Bigg| > \alpha/2 \Big\} \Big| d \alpha .$$
To bound $I_1$, we only use \eqref{final_w-t} and Fubini's theorem,
\begin{align*}
I_1 &\leq 2 \pi + \frac{4C_0}{(\rho_{\Lambda} - 1)^{ 1/2}} \int_1^{\infty} \frac{1}{\alpha} \Bigg[ \int_{ \{ |f| > \alpha \} } |f (x)| \log^{1/2} (e + |f(x)|) dx  \Bigg] d \alpha \\
& \leq    \frac{2 \pi + 4C_0}{(\rho_{\Lambda} - 1)^{ 1/2}} \Bigg[ 1 +  \int_{\mathbb{T} } |f (x)| \log^{3/2} (e + |f(x)|) dx  \Bigg] 
  \end{align*}
whereas for $I_2$, we use \eqref{trivial_2} and the fact that the map $ t \mapsto t (\log^{3/2} t)^{-1} $ is increasing on $[e^{3/2}, \infty)$,
\begin{align*}
I_2 &\leq 6 \pi  e^{ 3/2} +  4 \int_{e^{3/2}}^{\infty} \frac{1}{\alpha^2} \Bigg[ \int_{\{ e^{3/2} \leq |f| \leq \alpha \}} |f (x)|^2   dx  \Bigg] d \alpha \\
& \leq 6 \pi e^{3/2} + 4  \Bigg[ \int_{e^{3/2}}^{\infty} \Big(  \frac{1}{\alpha \log^{3/2} \alpha} + \frac{1}{\alpha \log^{5/2}  \alpha} \Big) d \alpha \Bigg] \Bigg[ \int_{\mathbb{T} } |f (x)| \log^{3/2} (e + |f(x)|) dx  \Bigg] \\
&\leq M_0 \Bigg[ 1 + \int_{\mathbb{T} } |f (x)| \log^{3/2} (e + |f(x)|) dx  \Bigg].
  \end{align*}
Here, we set $M_0 := 6 \pi  e^{3/2} + 4 \int_{e^{3/2}}^{\infty} ( [\alpha \log^{3/2} \alpha ]^{-1} + [\alpha \log^{5/2} \alpha]^{-1} ) d \alpha < \infty$. We thus conclude that \eqref{final_s-t} holds for $A_0 := \max \{ 2 \pi + 4C_0, M_0 \}$, $C_0$ being as in \eqref{final_w-t}.

Therefore, arguing again as in \cite{Bakas}, it follows from \eqref{trivial_2}, \eqref{final_s-t}, and Tao's converse extrapolation theorem \cite{Tao} that
\begin{equation}\label{desired_est}
 \Bigg\| \sum_{k \in \mathbb{N}_0} r_k (\omega) \Delta_k^{(\widetilde{\Lambda})}  \Bigg\|_{L^p (\mathbb{T}) \rightarrow L^p (\mathbb{T})} \lesssim \frac{1}{ (p-1)^{3/2}} (\rho_{\Lambda} - 1)^{-1/2}  \quad (1<p<2),
\end{equation}
where the implied constant does not depend on $\omega$ and $\Lambda$.

To justify the last step, notice that if $T$ is a linear, translation-invariant operator acting on functions defined over the torus and  such that $ \| T \|_{L^{p_0} (\mathbb{T}) \rightarrow L^{p_0} (\mathbb{T}) } \leq 1$ for some $p_0 >1$ and $ \| T \|_{L \log^r L (\mathbb{T}) \rightarrow L^1 (\mathbb{T})} \leq D_0$ for some $r>0$, then by carefully examining the proof of \cite[Theorem 1.1]{Tao}, one deduces that 
\begin{equation}\label{Tao_proof}
\| T \|_{L^p (\mathbb{T}) \rightarrow L^p (\mathbb{T})} \leq D_0 M_{r, p_0} (p-1)^{-r} \quad (p \rightarrow 1^+), 
\end{equation}
 where $ M_{r, p_0} > 0 $ is a  constant depending only on $r$ and $p_0$, but not on $D_0 $. Indeed, if $T$ is above, then note that in the proof of \cite[Theorem 1.1]{Tao}, $ \| T \|_{L^p (\mathbb{T}) \rightarrow L^p (\mathbb{T})} $ is estimated by the sum of the quantities on the right-hand sides of \cite[(13)]{Tao} and \cite[(14)]{Tao}. In the proof of \cite[(13)]{Tao} only the fact that $T$ is bounded on $L^{p_0} (\mathbb{T})$ is used in \cite{Tao}. Moreover, it follows from the argument in \cite[Section 3]{Tao} that the implicit constant on the right-hand side of \cite[(14)]{Tao} depends linearly on the implicit constant in \cite[(1)]{Tao}. In turn, the proof of \cite[Lemma 2.1]{Tao} yields  that the implicit constant in \cite[(1)]{Tao} depends linearly on $ \| T \|_{L \log^r L (\mathbb{T}) \rightarrow L^1 (\mathbb{T})} $ and one can thus conclude that \eqref{Tao_proof} holds.
To complete the justification of \eqref{desired_est}, observe that in our case we have $T = \sum_{k \in \mathbb{N}_0} r_k (\omega) \Delta_k^{(\widetilde{\Lambda})} $ and so, by employing \eqref{trivial_2} (i.e. $p_0 =2$) and \eqref{final_s-t}, we may take $r=3/2$ and $D_0 = A_0 (\rho_{\Lambda} - 1 )^{-1/2}$. We thus see that \eqref{desired_est} indeed holds, in view of \eqref{Tao_proof}. 

Therefore,  the proof of \eqref{main_result_2} is now obtained by using \eqref{desired_est}, Khintchine's inequality \eqref{Khintchine}, and \eqref{pointwise_square}.

%%%%
\begin{rmk}\label{summary_2}  By using Lemma \ref{construction} and a modification of the argument presented in Subsection \ref{Sharpness}, one can show that  the exponent $r=1/2$ in $(\rho_{\Lambda} - 1)^{-1/2}$ in \eqref{main_result_2} is best possible in general. 
  \end{rmk}
%%%%

%%%%
\begin{rmk}\label{second_proof}
 The argument presented in this section can also be used to give an alternative proof of Theorem \ref{lac}. More specifically, one can prove $\eqref{main_obs}$ for $X= H^1 (\mathbb{T})$ by adapting the proof of Tao and Wright that establishes \cite[Proposition 5.1]{TW} to the periodic setting, where instead of \cite[Proposition 9.1]{TW} one uses directly the square function characterisation of $H^1 (\mathbb{T})$. More precisely, given an $f \in H^1 (\mathbb{T})$, if one defines $F_j : = |\widetilde{\Delta}_j (f)|$ then it follows from the work of Tao and Wright adapted to the torus (see also the next subsection) that for every choice of $\omega \in \Omega$ one has
 \begin{equation}\label{alt_w-t_ineq}
 \Bigg\| \sum_{k \in \mathbb{N}_0} r_k (\omega) \Delta^{(\widetilde{\Lambda})}_k (f) \Bigg\|_{L^{1,\infty} (\mathbb{T})} \lesssim (\rho_{\Lambda} - 1)^{-1/2} \Bigg\|  \Bigg( \sum_{j \in \mathbb{N}_0} |\widetilde{\Delta}_j (f) |^2 \Bigg)^{1/2} \Bigg\|_{L^1 (\mathbb{T})}.
 \end{equation}
 Hence, by using the following Littlewood-Paley inequality 
  \begin{equation}\label{s_f_char}
 \Bigg\|  \Bigg( \sum_{j \in \mathbb{N}_0} |\widetilde{\Delta}_j (f) |^2 \Bigg)^{1/2} \Bigg\|_{L^1 (\mathbb{T})} \lesssim \| f \|_{H^1 (\mathbb{T})},
 \end{equation}
 together with \eqref{alt_w-t_ineq}, \eqref{Khintchine}, and \eqref{pointwise_square}, one obtains $\eqref{main_obs}$ for $X= H^1 (\mathbb{T})$. Note that
  \eqref{s_f_char} is a consequence  of e.g. \cite[Theorem 1.20]{CW} or the work of Stein \cite{Stein_1, Stein_2}.
\end{rmk}
%%%%

%%%%%%%% 
%%%%%%%% 
\subsection{Proof of \eqref{main_w-t}}\label{details} In this subsection, we show how one can adapt the argument on pp. 535--540 in \cite{TW} to the periodic setting and establish \eqref{main_w-t}.

For this, fix a trigonometric polynomial $f$ on $\mathbb{T}$ and  let $\underline{\widetilde{F} } :=  (\widetilde{F}_k)_{k \in \mathbb{N}}$ be as above. We shall prove that for every $\alpha > 0$ one has
\begin{equation}\label{w-t_alpha}
  \Big| \Big\{ x \in \mathbb{T} : \Bigg|  \sum_{k \in \mathbb{N}} r_k (\omega) \Delta_k^{( \widetilde{\Lambda} )} (f) (x) \Bigg| > \alpha \Big\} \Big| \leq \frac{C_0}{\alpha} \big\|\underline{\widetilde{F} }  \big\|_{L^1 (\mathbb{T} ; \ell^2 (\mathbb{N}))} ,
 \end{equation}
where $C_0 >0$ is an absolute constant, independent of $\omega$, $\Lambda$, $\alpha$, and $f$. 
 Towards this aim, fix an $\alpha>0$ and consider the set  
$$ \Omega_{\alpha} : = \big\{ x \in \mathbb{T} : M_c (\big\| \underline{\widetilde{F}} \big\|_{\ell^2 (\mathbb{N}) }) (x) > \alpha \big\}, $$
where $M_c$ denotes the centred Hardy-Littlewood maximal function acting on functions defined over $\mathbb{T}$.  Then one has
\begin{equation}\label{prop_1}
| \Omega_{\alpha} | \leq \frac{N_0}{\alpha} \big\| \underline{\widetilde{F}} \big\|_{L^1 (\mathbb{T}; \ell^2 (\mathbb{N}))},
\end{equation}
where $N_0 := \| M_c \|_{L^{1,\infty} (\mathbb{T}) \rightarrow L^1 (\mathbb{T})} < \infty$. By using a Whitney-type covering lemma; see e.g. pp. 167--168 in \cite{Singular_integrals}, it follows that there exists a countable collection of arcs $(J)_{J \in \mathcal{J}}$ in $\mathbb{T}$ whose interiors are mutually disjoint, satisfy $|J| \leq 1/8$ for all $J \in \mathcal{J}$ as well as the properties
\begin{equation}\label{prop_2}
\Omega_{\alpha}  = \bigcup_{J \in \mathcal{J}} J
\end{equation}
and  
\begin{equation}\label{prop_3}
c_0^{-1} |J| \leq \mathrm{dist} (\mathbb{T} \setminus \Omega_{\alpha}, J) \leq c_0 |J| \quad \mathrm{for}\ \mathrm{all}\   J \in \mathcal{J}
\end{equation}
where $c_0>0$ is an absolute constant. We thus have
\begin{equation}\label{arcs}
\sum_{J \in \mathcal{J}} |J| \leq \frac{N_0}{\alpha} \big\| \underline{\widetilde{F} }  \big\|_{L^1 (\mathbb{T};  \ell^2 (\mathbb{N}) )} ,
\end{equation}
where the absolute constant $N_0 > 0$ is as in \eqref{prop_1}.
 Moreover, it follows from the definition of $\Omega_{\alpha}$ that 
\begin{equation}\label{prop_4}
\big\| \underline{\widetilde{F} } (x) \big\|_{\ell^2 (\mathbb{N})} \leq \alpha \quad \mathrm{for}\ \mathrm{all} \ x \in \mathbb{T} \setminus \Omega_{\alpha}.
\end{equation}
Furthermore, \eqref{prop_2} and \eqref{prop_3} imply that
\begin{equation}\label{prop_5}
\frac{1}{|J|} \int_J \big\| \underline{\widetilde{F} } (x) \big\|_{\ell^2 (\mathbb{N})} dx \lesssim \alpha \quad \mathrm{for}\ \mathrm{all} \ J \in \mathcal{J}
\end{equation}
where the implied constant in \eqref{prop_5} is independent of $\underline{F}$ and $\alpha$. Define $\underline{G} = (g_k )_{k \in \mathbb{N}}$ by
$$  g_k (x) : =\begin{cases}
|J|^{-1} \int_J \widetilde{F}_k (y) dy, \quad \mathrm{if} \ x \in J \\
\widetilde{F}_k (x), \quad \mathrm{if}\ x \in \mathbb{T} \setminus \Omega_{\alpha} .
\end{cases} 
$$
Using \eqref{prop_4}, \eqref{prop_5} as well as Minkowski's integral inequality one deduces that
\begin{equation}\label{prop_G}
 \| \underline{G}   \|^2_{L^2 (\mathbb{T}; \ell^2 (\mathbb{N}) )} \lesssim \alpha  \big\| \underline{\widetilde{F}} \big\|_{L^1 (\mathbb{T}; \ell^2 (\mathbb{N}) ) } ,
\end{equation}
where the implied constant in \eqref{prop_G} does not depend on $\underline{F}$ and $\alpha$. We also define $\underline{B} = (b_k)_{k \in \mathbb{N}}$ by 
$$ b_k : = \widetilde{F}_k - g_k , \quad k \in \mathbb{N}.$$
 Then it easily follows from \eqref{prop_5} that
\begin{equation}\label{prop_B}
\int_J \| \underline{B} (x) \|_{ \ell^2 (\mathbb{N})} dx \lesssim \alpha |J| \quad \mathrm{for} \ \mathrm{all} \ J \in \mathcal{J},
\end{equation}
where the implied constant in \eqref{prop_B} is independent of $\underline{F}$ and $\alpha$. Notice that we also have that $\int_J b_k (x) dx =0$ for all $k \in \mathbb{N}$, but we will not exploit this property here.

Next, we write $\widetilde{\Lambda} = (\widetilde{\lambda}_k)_{k \in \mathbb{N}_0} $ and for $k \in \mathbb{N}$ define the intervals
$$ I_k  : = [ \widetilde{\lambda}_{k-1}, \widetilde{\lambda}_k ) \quad \mathrm{and} \quad  I'_k  : = (- \widetilde{\lambda}_k, - \widetilde{\lambda}_{k-1} ] .$$
Observe that
$$ \Delta_k^{(\widetilde{\Lambda})} =  \Delta_k^{(\widetilde{\Lambda}, +)} +  \Delta_k^{(\widetilde{\Lambda},-)} , $$
where
$  \Delta_k^{(\widetilde{\Lambda},+)} : = T_{\chi_{I_k |_{  \mathbb{Z}}}} $ and $\Delta_k^{(\widetilde{\Lambda},-)} : = T_{\chi_{I'_k |_{\mathbb{Z} }}} $. Hence, to prove \eqref{w-t_alpha}, it suffices to show that
\begin{equation}\label{w-t_alpha_+}
\Big| \Big\{ x \in \mathbb{T} : \Bigg|  \sum_{k \in \mathbb{N}} r_k (\omega) \Delta_k^{(\widetilde{\Lambda}, +)} (f) (x) \Bigg| > \alpha /2 \Big\} \Big| \leq \frac{C_0}{2 \alpha} \big\| \underline{\widetilde{F} }  \big\|_{L^1 (\mathbb{T} ; \ell^2 (\mathbb{N}))} 
\end{equation}
and
\begin{equation}\label{w-t_alpha_-}
\Big| \Big\{ x \in \mathbb{T} : \Bigg|  \sum_{k \in \mathbb{N}} r_k (\omega) \Delta_k^{(\widetilde{\Lambda}, -)} (f) (x) \Bigg| > \alpha /2 \Big\} \Big| \leq \frac{ C_0}{2 \alpha}  \big\| \underline{\widetilde{F} }  \big\|_{L^1 (\mathbb{T} ; \ell^2 (\mathbb{N}))} ,
\end{equation}
where $C_0$ is the constant in \eqref{w-t_alpha}. 

We shall only focus on the proof of \eqref{w-t_alpha_+}, as the proof of \eqref{w-t_alpha_-} is completely analogous. To prove \eqref{w-t_alpha_+}, for $k \in \mathbb{N}$ consider the functions
$$ \phi_{I_k} (x) : = | I_k |  \big[ 1 +  |I_k|^2 \sin^2 (x/2) \big]^{-3/4} , \quad x \in \mathbb{T} $$
and for a fixed Schwartz function $\psi$ that is even, supported in $[-2,2]  $ and such that $\psi|_{[-1,1]} \equiv 1$, let 
$$ \psi_k (\xi) : =  \psi \big(  |I_k|^{-1} \big[ \xi - \xi_{I_k} \big] \big), \quad \xi \in \mathbb{R},$$
where $\xi_{I_k}$ denotes the centre of $I_k$. Consider now the multiplier $ T_{\psi_k |_{ \mathbb{Z}}}$ and note that it follows from the definition of $\psi_k$ and Lemma \ref{refinement} that
$$ T_{\psi_k |_{\mathbb{Z}}} (f) = T_{\psi_k |_{ \mathbb{Z}}}  \big( \widetilde{\Delta}_{j_k}  (f) \big) ,$$
where $j_k \in \mathbb{N}$ is such that $I_k \subseteq [2^{j_k -1}, 2^{j_k } ) $ and $\widetilde{\Delta}_{j_k} =  T_{\eta_{j_k}|_{ \mathbb{Z}}}  $ with $\eta_j$ being as in the previous section. Hence, it follows from \eqref{majorisation} and the smoothness of $\psi$ that there exists an absolute constant $M_0 >0$ such that
\begin{equation}\label{pointwise_ineq}
 | T_{\psi_k |_{ \mathbb{Z}}} (f) (x) | \leq M_0   \big( \widetilde{F}_k \ast \phi_{I_k} \ast \phi_{j_k} \big) (x) \quad \mathrm{for} \ \mathrm{all} \ x \in \mathbb{T},
 \end{equation}
where $j_k \in \mathbb{N}$ is as above, i.e. $I_k \subseteq [2^{j_k -1}, 2^{j_k })$ and $\phi_{j_k}$ is as in the previous section, namely $ \phi_{j_k} (x) = 2^{j_k} [1+ 2^{2 j_k} \sin^2 (x/2) ]^{-3/4}$, $x \in \mathbb{T}$.  For $k \in \mathbb{N}$, if $\widetilde{F}_k$ is not identically zero, define  the function $a_k$ by
$$ a_k  (x) :=   T_{\psi_k |_{ \mathbb{Z}}} (f) (x) \Big[ \big( \widetilde{F}_k \ast \phi_{I_k} \ast \phi_{j_k} \big) (x) \Big]^{-1}, \quad x \in \mathbb{T}.$$
Otherwise, define $a_k (x) :=0$, $x \in \mathbb{T}$. Hence, the definition of $a_k$ and \eqref{pointwise_ineq} imply that $ \|a_k \|_{L^{\infty} (\mathbb{T})} \leq M_0$ for all $k \in \mathbb{N}$. For each $k \in \mathbb{N}$, we thus have that 
$$ \Delta_k^{(\widetilde{\Lambda}, +)} (f) = \Delta_k^{(\widetilde{\Lambda}, +)} ( T_{\psi_k |_{ \mathbb{Z}}} (f) ) =  \Delta_k^{(\widetilde{\Lambda}, +)} \big( a_k (\widetilde{F}_k \ast \phi_{I_k} \ast  \phi_{j_k}) \big)
$$
and so, to prove \eqref{w-t_alpha_+}, it suffices to show that
$$ \Big| \Big\{ x \in \mathbb{T} : \Bigg|  \sum_{k \in \mathbb{N} } r_k (\omega)    \Delta_k^{(\widetilde{\Lambda},+)} \big( a_k (\widetilde{F}_k \ast \phi_{I_k} \ast \phi_{j_k}) \big) (x) \Bigg| > \frac{\alpha}{2} \Big\} \Big| \leq 
 \frac{C'_0}{  \alpha}  \big\| \underline{\widetilde{F} } \big\|_{L^1 (\mathbb{T} ; \ell^2 (\mathbb{N}))} ,
 $$
where $C'_0 > 0$ is an absolute constant that does not depend on $\omega$, $\Lambda$, $\alpha$, and $f$. Since $\widetilde{F}_k = g_k + b_k$, it is enough to prove that
$$ \Big| \Big\{ x \in \mathbb{T} : \Bigg|  \sum_{k \in \mathbb{N} } r_k (\omega)    \Delta_k^{(\widetilde{\Lambda}, +)} \big( a_k  ( g_k \ast \phi_{I_k} \ast \phi_{j_k}) \big) (x) \Bigg| >\frac{\alpha}{4} \Big\} \Big| \leq \frac {C''_0}{\alpha}  \big\| \underline{\widetilde{F} }  \big\|_{L^1 (\mathbb{T} ; \ell^2 (\mathbb{N}))}   
$$
 and
$$ \Big| \Big\{ x \in \mathbb{T} : \Bigg|  \sum_{k \in \mathbb{N} } r_k (\omega)    \Delta_k^{(\widetilde{\Lambda}, +)} \big( a_k (b_k \ast \phi_{I_k} \ast \phi_{j_k}) \big) (x) \Bigg| > \frac{\alpha}{4} \Big\} \Big| \leq \frac{ C''_0} {\alpha}   \big\| \underline{\widetilde{F} }  \big\|_{L^1 (\mathbb{T} ; \ell^2 (\mathbb{N}))}  
$$
for some absolute constant $C''_0 > 0$.

To establish the first bound involving $\underline{G} = (g_k)_{k \in \mathbb{N}}$, note that by using  Chebyshev's inequality, Parseval's identity twice as well as the fact that $ \|a_k \|_{L^{\infty} (\mathbb{T})} \leq M_0$ for all $k \in \mathbb{N}$, one obtains
\begin{align*}
& \Big| \Big\{ x \in \mathbb{T} : \Bigg| \sum_{k \in \mathbb{N} } r_k (\omega)    \Delta_k^{(\widetilde{\Lambda}, +)} \big( a_k ( g_k \ast \phi_{I_k} \ast \phi_{j_k}) \big) (x) \Bigg| >  \frac{\alpha}{4} \Big\} \Big| \lesssim \\
 & \frac{1} { \alpha^2 } \sum_{k \in \mathbb{N} } \big\| g_k \ast \phi_{I_k} \ast \phi_{j_k} \big\|_{L^2 (\mathbb{T})}^2 ,
\end{align*}
where the implied constant is independent of $\omega$, $\Lambda$, $\alpha$, and $f$. Using now Young's inequality twice combined with the fact that there exists an absolute constant $C>0$ such that $ \| \phi_{I_k} \|_{L^1 (\mathbb{T})} \leq C $ and $ \| \phi_{j_k} \|_{L^1 (\mathbb{T})} \leq C $ for all $k \in \mathbb{N}$, we deduce that
$$ \Big| \Big\{ x \in \mathbb{T} : \Bigg| \sum_{k \in \mathbb{N} } r_k (\omega)    \Delta_k^{(\widetilde{\Lambda}, +)} \big( a_k ( g_k \ast \phi_{I_k} \ast \phi_{j_k} ) \big) (x) \Bigg| >  \frac{\alpha}{4} \Big\} \Big| \lesssim \frac{1} {\alpha^2 } \| \underline{G} \|_{L^2 (\mathbb{T}; \ell^2 (\mathbb{N}))}^2,
$$
where the implied constant does not depend on $\omega$, $\Lambda$, $\alpha$, and $f$. Hence, the desired inequality for $ \underline{G} = (g_k)_{k \in \mathbb{N}} $ follows from the last estimate combined with \eqref{prop_G}. 

To prove the desired weak-type inequality involving $\underline{B} = (b_k)_{k \in \mathbb{N}}$, for $J \in \mathcal{J}$ we write 
$$ b_{k,J} := \chi_J   b_k, \quad k \in \mathbb{N} $$
and, as in \cite{TW}, we have
$$ \Big| \Big\{ x \in \mathbb{T} : \Bigg|  \sum_{k \in \mathbb{N} } r_k (\omega)    \Delta_k^{(\widetilde{\Lambda},+)} \big( a_k (b_k \ast \phi_{I_k} \ast \phi_{j_k}) \big) (x) \Bigg| > \frac{\alpha}{4} \Big\} \Big| \leq I_1 + I_2,
 $$ 
where
$$ I_1 := \Big| \Big\{ x \in \mathbb{T} : \Bigg| \sum_{\substack{k \in \mathbb{N}, J \in \mathcal{J}: \\ |I_k| | J| \leq 1 }} r_k (\omega) \Delta_k^{(\widetilde{\Lambda},+)} \big( a_k ( b_{k,J} \ast \phi_{I_k} \ast \phi_{j_k} ) \big) (x) \Bigg| > \frac{\alpha}{8} \Big\} \Big|  $$
and 
$$ I_2 := \Big| \Big\{ x \in \mathbb{T} : \Bigg|  \sum_{\substack{k \in \mathbb{N}, J \in \mathcal{J}: \\ |I_k| | J| > 1 }} r_k (\omega) \Delta_k^{(\widetilde{\Lambda},+)} \big( a_k ( b_{k,J} \ast \phi_{I_k} \ast \phi_{j_k} ) \big) (x) \Bigg| > \frac{\alpha}{8} \Big\} \Big|  .$$
We shall prove separately that
\begin{equation}\label{bound_I1} I_1 \lesssim \frac{1} {\alpha} \big\| \underline{\widetilde{F}} \big\|_{L^1 (\mathbb{T} ; \ell^2 (\mathbb{N}))} 
\end{equation}
and
\begin{equation}\label{bound_I2} I_2 \lesssim \frac{1} {\alpha} \big\| \underline{\widetilde{F}} \big\|_{L^1 (\mathbb{T} ; \ell^2 (\mathbb{N}))} ,
\end{equation}
where the implied constants in \eqref{bound_I1} and \eqref{bound_I2} do not depend on $\omega$, $\Lambda$, $\alpha$, and $f$.

%%%%
%%%%
\subsubsection{Proof of \eqref{bound_I1}}
To prove the desired bound for $I_1$, observe that by using Chebyshev's inequality, Parseval's identity twice, the bound $\| a_k \|_{L^{\infty} (\mathbb{T})} \leq M_0$ and then Young's inequality together with the fact that there exists an absolute constant $C>0$ such that $ \| \phi_{j_k} \|_{L^1 (\mathbb{T})} \leq C$, one has
\begin{equation}\label{ineq_I1}
 I_1 \lesssim \frac{1}{\alpha^2}  \sum_{k \in \mathbb{N}}  \Big\| \sum_{\substack{ J \in \mathcal{J} : \\ |I_k| |J| \leq 1}} b_{k,J} \ast \phi_{I_k}    \Big\|_{L^2 (\mathbb{T})}^2 ,
 \end{equation}
where the implied constant is independent of $\omega$, $\Lambda$, $\alpha$, and $f$. To get an appropriate bound for the right-hand side of \eqref{ineq_I1}, we shall use the following lemma.

%%%%
\begin{lemma}\label{pointwise_claim} Let $b_{k,J}$ and  $\phi_{I_k}$ be as above.

If $|I_k||J| \leq 1$, then there exists an absolute constant $A_0>0$ such that 
\begin{equation}\label{pointwise_I1}
| ( b_{k,J} \ast \phi_{I_k} ) (x)| \leq A_0 \| b_{k,J} \|_{L^1 (\mathbb{T})} |J|^{-1} ( \chi_J \ast \phi_{I_k} ) (x)
 \end{equation}
for all $x \in \mathbb{T}$.
\end{lemma}
%%%%

%%%%
\begin{proof} We may suppose, without loss of generality, that $2J$ can be regarded as an interval in $[- \pi, \pi)$. To prove \eqref{pointwise_I1}, take an $x \in [ - \pi, \pi)$ and consider two cases; $x \in 2J$ and $x \notin 2J$. 

 Assume first that $x \notin 2J$. In the subcase where $|x-y| \leq 3 \pi/2$ for all $y \in J$ (with $- \pi \leq y < \pi$) note that for all $y, y' \in J$ we have $|x-y| \geq |x-y'|/3$. Hence, if $y_J$ denotes the centre of $J$, we get
\begin{align*} 
| b_{k,J}   \ast \phi_{I_k}(x) | & \lesssim \int_J | b_{k,J} (y) | |I_k|  ( 1+ | I_k|^2 | x-y |^2 )^{-3/4}  dy   \\
& \lesssim  \| b_{k,J} \|_{L^1 (\mathbb{T})} |I_k| (1+ | I_k|^2 | x-y_J |^2 )^{-3/4}    \\
& \lesssim  \| b_{k,J} \|_{L^1 (\mathbb{T})} |J|^{-1} \int_J |I_k|  ( 1 + | I_k|^2 | x - y |^2 )^{-3/4}   dy  \\
& \lesssim  \| b_{k,J} \|_{L^1 (\mathbb{T})}  |J|^{-1} ( \chi_J \ast \phi_{I_k} ) (x),
\end{align*}
where the implied constants in the above chain of inequalities do not depend on $k$, $J$ and we used the fact that
\begin{equation}\label{simple_fact}
 |I_k| (1 + | I_k|^2  |x|^2 )^{-3/4} \leq  \phi_{I_k} (x) \leq 64 |I_k| ( 1+ | I_k|^2 |x|^2 )^{-3/4} 
 \end{equation}
for all $x \in [-3\pi/2, 3\pi/2]$, which is a direct consequence of the elementary inequalities $y/8 \leq \sin (y/2) $ for $ y \in [0 , 3\pi/2 ]$ and $\sin (y) \leq y$ for all $y \geq 0$. Notice that the inequality on the left-hand side of \eqref{simple_fact} holds for all $x\in \mathbb{R}$. If we now consider the subcase where $|x-y| > 3 \pi/2$ for some $y \in J$ (with $ - \pi \leq y < \pi $), then $|x-y| \sim |x| \sim |y|$ for all $ y \in J$ (noting that $|J| \leq 1/8$) and so, by using \eqref{simple_fact}, one has
 \begin{align*} 
| b_{k,J} \ast \phi_{I_k}(x) | & \lesssim \int_{[-\pi , \pi) }  | b_{k,J} (x-y) | | I_k | ( 1+ | I_k|^2 |y|^2 )^{-3/4} dy \\
& \lesssim \| b_{k,J} \|_{L^1 (\mathbb{T})} |I_k| (1+ | I_k|^2  |x|^2 )^{-3/4}  \\
& \lesssim \| b_{k,J} \|_{L^1 (\mathbb{T})} |J|^{-1} \int_J |I_k| (1 + | I_k|^2 | x - y |^2 )^{-3/4} dy \\
&  \lesssim \| b_{k,J} \|_{L^1 (\mathbb{T})}  |J|^{-1} ( \chi_J \ast \phi_{I_k}) (x).
\end{align*}

If $x \in 2J$, observe that $ |I_k| |x- y| \leq 2 |I_k| |J| \leq 2$ for all $y \in J$ and hence, by using the trivial estimate $ | b_{k,J} \ast \phi_{I_k}(x)| \leq \|  b_{k,J} \|_{L^1 (\mathbb{T})} |I_k|$ and \eqref{simple_fact}, we get
\begin{align*}
 | b_{k,J} \ast \phi_{I_k}(x)| &\leq \| b_{k,J} \|_{L^1(\mathbb{T})} |J|^{-1} \int_J |I_k| dy \lesssim \|b_{k,J} \|_{L^1 (\mathbb{T})} |J|^{-1} (\phi_{I_k} \ast \chi_J ) (x),  
\end{align*}
where the implied constant is independent of $k$, $J$.
\end{proof}
%%%%

Using \eqref{pointwise_I1}, \eqref{ineq_I1} becomes
\begin{equation}\label{ineq_I1_2} I_1 \leq \frac{A_0' }{\alpha^2} \sum_{k \in \mathbb{N}} \Bigg\| \sum_{J \in \mathcal{J}} \| b_{k,J} \|_{L^1 (\mathbb{T})} |J|^{-1} ( \phi_{I_k} \ast \chi_J ) \Bigg\|^2_{L^2 (\mathbb{T})},
\end{equation}
where $A_0' > 0$ is an absolute constant. Hence, by using  \eqref{ineq_I1_2} and Young's inequality combined with the fact that $\| \phi_{I_k} \|_{L^1 (\mathbb{T})} \leq C$,  we get
$$
 I_1 \leq \frac{A_0'' }{\alpha^2} \sum_{k \in \mathbb{N}} \int_{\mathbb{T}} \Bigg( \sum_{J \in \mathcal{J}} \|  b_{k,J} \|_{L^1 (\mathbb{T})} |J|^{-1} \chi_J  (x) \Bigg)^2  dx ,
$$
where $A_0''>0$ is an absolute constant.  Since the arcs $(J)_{J \in \mathcal{J}}$ have mutually disjoint interiors, we have
\begin{equation}\label{fact_ineq_I1}
 I_1 \leq \frac{B_0 }{\alpha^2} \sum_{J \in \mathcal{J}} |J|^{-1} \sum_{k \in \mathbb{N}} \|  b_{k,J} \|_{L^1 (\mathbb{T})}^2 ,
\end{equation}
where $B_0 >0$ is an absolute constant. Observe that by using Minkowski's integral inequality and \eqref{prop_B} one has
\begin{equation}\label{prop_B_2}
\sum_{k \in \mathbb{N}} \|  b_{k,J} \|^2_{L^1 (\mathbb{T})} \leq B'_0 \alpha^2 |J|^2  
\end{equation}
where $B'_0  >0$ is an absolute constant. Therefore, by using \eqref{fact_ineq_I1} together with  \eqref{prop_B_2} and \eqref{arcs}, \eqref{bound_I1} follows.

%%%%
%%%%
\subsubsection{Proof of \eqref{bound_I2}}
We shall now estimate the remaining term $I_2$. As in \cite{TW}, for $k \in \mathbb{N}_0$ and $J \in \mathcal{J}$ with $|I_k| |J| >1$, we write
$$  b_{k,J} \ast \phi_{I_k} \ast \phi_{j_k} =  \chi_{\mathbb{T} \setminus (2 J) } (x) (b_{k,J} \ast \phi_{I_k} \ast \phi_{j_k}) (x) +  \chi_{2J} (x)   ( b_{k,J} \ast \phi_{I_k} \ast \phi_{j_k} ) (x)  $$
for $x \in \mathbb{T}$ and hence, we have the inequality
\begin{equation}\label{I_2_split}
I_2 \leq  I^{(i)}_2 +  I^{(ii)}_2,
\end{equation}
where
$$  I^{(i)}_2 : =  \Big| \Big\{ x \in \mathbb{T} : \Bigg|  \sum_{\substack{k \in \mathbb{N}, J \in \mathcal{J}: \\ |I_k| | J| > 1 }} r_k (\omega) \Delta_k^{(\widetilde{\Lambda},+)} \big(  a_k    \chi_{\mathbb{T} \setminus (2 J) }  ( b_{k,J} \ast \phi_{I_k} \ast \phi_{j_k}) \big) (x) \Bigg| > \frac{\alpha} {16} \Big\} \Big| $$
and
$$  I^{(ii)}_2 : =  \Big| \Big\{ x \in \mathbb{T} : \Bigg|  \sum_{\substack{k \in \mathbb{N}, J \in \mathcal{J}: \\ |I_k| | J| > 1 }} r_k (\omega) \Delta_k^{(\widetilde{\Lambda},+)} \big(  a_k  \chi_{  2 J  }  ( b_{k,J} \ast \phi_{I_k} \ast \phi_{j_k}) \big) (x) \Bigg| > \frac{\alpha} {16} \Big\} \Big| .$$

We shall first handle the term $I^{(i)}_2$. For this, we need the following lemma.

%%%%
\begin{lemma}\label{pointwise_bound_second_term}
Assume that $|I_k| |J| > 1$.

There exists an absolute constant $C_0 >0$ such that
$$ \big|  \chi_{\mathbb{T} \setminus (2 J) } (x)  \big( b_{k,J} \ast \phi_{I_k} \ast \phi_{j_k} \big) (x ) \big| \leq C_0   \| b_{k,J} \|_{L^1 (\mathbb{T})} |J|^{-1} \big[ M_c (\chi_J) (x) \big]^{3/2} $$
for all $x \in \mathbb{T} $.
\end{lemma}
%%%%

%%%%
\begin{proof} To prove the desired estimate, note that there exists an absolute constant $C>0$ such that
\begin{equation}\label{pointwise_approximation}
\phi_{I_k} \ast \phi_{j_k} (x) \leq C \big[ \phi_{I_k} (x) + \phi_{j_k} (x) \big]
\end{equation}
for all $x  \in [ -\pi, \pi)$. To see this,  consider first the case where  $x \in [-\pi/2 ,\pi/2]$ and write
$$ \phi_{I_k} \ast \phi_{j_k} (x) = I_1 (x) + I_2 (x) + I_3 (x),$$
where
$$ I_1 (x) : = (2 \pi)^{-1}  \int_{\substack{ - \pi  \leq y <  \pi : \\ |x| > 2|y| }}  \phi_{j_k} (y) \phi_{I_k} (x- y)  dy ,$$
$$ I_2 (x) : = (2 \pi)^{-1}  \int_{\substack{ - \pi  \leq y <  \pi : \\ 2 |x|  < | y|  } } \phi_{j_k} (y) \phi_{I_k} (x- y) dy,  $$
and
$$ I_3 (x) : = (2 \pi)^{-1}  \int_{\substack{ - \pi  \leq y <  \pi : \\ |x|/2 \leq |y| \leq 2|x| }} \phi_{j_k} (y)  \phi_{I_k} (x- y) dy  .$$
Observe that in this subcase one has $|x-y| \leq 3 \pi/2$ for all $y \in [-\pi, \pi)$ and so, by using \eqref{simple_fact} twice, one gets
\begin{align*}
 I_1 (x) & \lesssim  \int_{\substack{ - \pi  \leq y <  \pi : \\ |x| > 2|y| }}  \phi_{j_k} (y)   |I_k| (1 + |I_k|^2 |x-y|^2)^{-3/4}   dy   \\
& \lesssim |I_k|  (1 + |I_k|^2 |x|^2)^{-3/4} \int_{\substack{ - \pi  \leq y <  \pi : \\ |x| > 2|y| }} \phi_{j_k} (y) dy \\
& \lesssim \phi_{I_k} (x)  \int_{\mathbb{T}}  \phi_{j_k} (y) dy .
\end{align*}
Since  $\int_{\mathbb{T}} \phi_{j_k} (y) dy  \lesssim 1 $, one deduces that $I_1 (x)$ is $O(\phi_{I_k} (x))$.  Similarly, one shows that $ I_2 (x)$ is $O( \phi_{I_k} (x) )$ and $  I_3 (x) $ is $ O (  \phi_{j_k} (x) ) $. Therefore, by putting the above estimates together, it follows that \eqref{pointwise_approximation} holds when $x \in [-\pi/2, \pi/2]$. If we now assume that $x \in [-\pi, \pi) $ and $|x| > \pi/2$, we write
\begin{align*}
  \phi_{I_k} \ast \phi_{j_k} (x) &=   (2 \pi)^{-1}  \int_{\substack{ - \pi  \leq y <  \pi : \\ |x-y| \leq 3\pi/2 }}  \phi_{j_k} (y) \phi_{I_k} (x- y)  dy \\
  &+ (2 \pi)^{-1}  \int_{\substack{ - \pi  \leq y <  \pi : \\ |x-y| > 3 \pi/2 }}  \phi_{j_k} (y) \phi_{I_k} (x- y)  dy.
\end{align*}
The first integral is handled as in the first case. For the second integral, we have $|y| \sim |x|$ for all $y \in [ -\pi, \pi)$ with $|x-y| > 3\pi/2$ and so, by using a version of \eqref{simple_fact} for $\phi_{j_k}$ as well as the fact that $\int_{\mathbb{T}} \phi_{I_k} (y) dy \lesssim 1$, we get
\begin{align*} 
\int_{\substack{ - \pi  \leq y <  \pi : \\ |x-y| > 3 \pi/2 }}  \phi_{j_k} (y) \phi_{I_k} (x- y)  dy 
& \lesssim  \int_{\substack{ - \pi  \leq y <  \pi : \\ |x-y| > 3 \pi/2 }}  2^{j_k} (1 + 2^{2j_k} |y|^2)^{-3/4}  \phi_{I_k} (x- y)  dy \\
& \lesssim 2^{j_k} (1 + 2^{2 j_k} |x|^2)^{-3/4} \int_{\substack{ - \pi  \leq y <  \pi : \\ |x-y| > 3 \pi/2 }}     \phi_{I_k} (x- y)  dy \\
& \lesssim \phi_{j_k} (x),  
\end{align*}
as desired.

Therefore, by using \eqref{pointwise_approximation}, one has
\begin{equation}\label{ineq_two}
 | b_{k,J} \ast \phi_{I_k} \ast \phi_{j_k}  (x) | \leq C \big[ | b_{k,J} | \ast \phi_{I_k} (x) + | b_{k,J} | \ast \phi_{j_k} (x) \big] \quad (x \in \mathbb{T}).
 \end{equation}
We may suppose that $2J $ can be regarded as an interval in $[ -\pi , \pi)$. To estimate the first term on the right-hand side of \eqref{ineq_two}, take an $x \in [- \pi, \pi) \setminus 2J$ and consider first the case where $|x-y| \leq 3 \pi/2$ for all $y \in J$. Note that, by using \eqref{simple_fact}, one has 
\begin{align*}
| b_{k,J} | \ast \phi_{I_k} (x) & \lesssim \int_J |b_{k,J} (y)|    |I_k|  (1 + |I_k|^2 |x-y|^2)^{-3/4}  dy \\
& \lesssim \| b_{k,J} \|_{L^1 (\mathbb{T})}   |I_k|^{-1/2}  [ \mathrm{dist} (x,J) ]^{-3/2} \\
& \lesssim \| b_{k,J} \|_{L^1 (\mathbb{T})} |J|^{-1} [ M_c (\chi_J) (x) ]^{3/2}  
\end{align*}
where the constants in the above chain of inequalities do not depend of $k$, $J$ and we used the assumption that $ |I_k| |J| > 1 $ as well as the fact that  $ M_c (\chi_J) (x) \sim |J| [ \mathrm{dist} (x,J)]^{-1} $ when $x \notin   2J$. In the case where $|x-y| > 3 \pi/2$ for some $y \in J$, one has $|x| \sim |y| \sim |x-y| \sim 1$ for all $y \in J$ and hence,
\begin{align*}
| b_{k,J} | \ast \phi_{I_k} (x) & \lesssim \int_{- \pi \leq y < \pi} |b_{k,J} (x- y)| |I_k|  (1 + |I_k|^2 |x|^2)^{-3/4} dy \\
& \lesssim \| b_{k,J} \|_{L^1 (\mathbb{T})}   |I_k|^{-1/2} [ \mathrm{dist} (x,J) ]^{-3/2} \\
& \lesssim \| b_{k,J} \|_{L^1 (\mathbb{T})} |J|^{-1} [ M_c (\chi_J) (x) ]^{3/2}  .
\end{align*}
Since we have $2^{j_k} | J |  \geq  |I_k| |J| > 1 $, a similar argument shows that $| b_{k,J} | \ast \phi_{j_k} (x)$ is $O(\| b_{k,J} \|_{L^1 (\mathbb{T})}   |J|^{-1}  [ M_c (\chi_J) (x) ]^{3/2} )$ for all $x  \notin 2J $ and hence, the proof of the lemma is complete.
\end{proof}
%%%%

Having established Lemma \ref{pointwise_bound_second_term}, observe that it follows from Chebyshev's inequality and Parseval's identity that
$$ I^{(i)}_2 \leq \frac{C} {\alpha^2} \sum_{k \in \mathbb{N}} \Bigg\| \sum_{J \in \mathcal{J}} \| b_{k,J} \|_{L^1 (\mathbb{T})} |J|^{-1} [ M_c (\chi_J) ]^{3/2} \Bigg\|^2_{L^2 (\mathbb{T})}  $$
where $C>0$ is an absolute constant. Since we may write
\begin{align*}
& \Bigg\| \sum_{J \in \mathcal{J}} \| b_{k,J} \|_{L^1 (\mathbb{T})} |J|^{-1} \big[ M_c (\chi_J) \big]^{3/2}  \Bigg\|^2_{L^2 (\mathbb{T})} = \\
& (2 \pi)^{-1} \int_{\mathbb{T}} \Bigg(  \sum_{J \in \mathcal{J} } \big[ M_c \big( \|  b_{k,J} \|^{2/3}_{L^1 (\mathbb{T})} |J|^{-2/3} \chi_J \big)  (x) \big]^{3/2} \Bigg)^2 dx ,
\end{align*}
an application of a periodic version of the Fefferman-Stein maximal theorem \cite[Theorem 1 (1)]{F-S} (for $r=3/2$ and $p=3$) gives us that
\begin{align*}\Bigg\| \sum_{J \in \mathcal{J}} \| b_{k,J} \|_{L^1 (\mathbb{T})} |J|^{-1}  [ M_c (\chi_J) ]^{3/2} \Bigg\|^2_{L^2 (\mathbb{T})} & \lesssim \int_{\mathbb{T}} \Bigg( \sum_{J \in \mathcal{J} } \|  b_{k,J} \|_{L^1 (\mathbb{T})} |J|^{-1} \chi_J (x) \Bigg)^2 dx  \\
& = \sum_{J \in \mathcal{J} } \| b_{k,J} \|^2_{L^1 (\mathbb{T})} |J|^{-1} , 
\end{align*}
where in the last step we used the fact that the arcs $( J )_{ J \in \mathcal{J}}$ have mutually disjoint interiors. We thus deduce that there exists an absolute constant $A_0 >0$ such that
\begin{equation}\label{ineq_I2_i}
 \sum_{k \in \mathbb{N}} \Bigg\| \sum_{J \in \mathcal{J}} \|  b_{k,J} \|_{L^1 (\mathbb{T})}   |J|^{-1} \big[ M_c (\chi_J) \big]^{3/2}  \Bigg\|^2_{L^2 (\mathbb{T})} \leq A_0 \sum_{J \in \mathcal{J} }  |J|^{-1}   \sum_{k \in \mathbb{N}} \|  b_{k,J} \|^2_{L^1 (\mathbb{T})} .
 \end{equation}
Hence, by using \eqref{prop_B_2}, \eqref{ineq_I2_i}, and \eqref{arcs}, it follows that 
\begin{equation}\label{bound_I2_a} I^{(i)}_2 \lesssim \frac{1} {\alpha}  \big\| \underline{\widetilde{F} }  \big\|_{L^1 (\mathbb{T} ; \ell^2 (\mathbb{N}))} ,
\end{equation}
where the implied constant does not depend on $\omega$, $\Lambda$, $\alpha$, and $f$. 

We shall now prove that
\begin{equation}\label{bound_I2_b} I^{(ii)}_2 \lesssim \frac{1} {\alpha}  \big\| \underline{\widetilde{F} }  \big\|_{L^1 (\mathbb{T} ; \ell^2 (\mathbb{N}))},
\end{equation}
To this end,  for $|I_k| |J| >1 $, we decompose the projection $\Delta_k^{(\widetilde{\Lambda}, +)}$ as
$$ \Delta_k^{(\widetilde{\Lambda}, +)} = \Delta_k^{(\widetilde{\Lambda}, +)} \circ T_{	P_{k,J}} +   T_{Q_{k,J}} + \Delta_k^{(\widetilde{\Lambda}, +)} \circ T_{	\widetilde{P}_{k,J}} , $$
where the operators $T_{P_{k,J}}$, $T_{Q_{k,J}}$, $T_{\widetilde{P}_{k,J}}$ are defined as follows. The operators $T_{P_{k,J}}$ and $T_{\widetilde{P}_{k,J}}$ are of convolution-type and their kernels $p_{k,J}$ and $\widetilde{p}_{k,J}$ are trigonometric polynomials given by
$$ p_{k,J} (x) : = e^{i \widetilde{\lambda}_{k-1} x} K_{\lfloor |J|^{-1} \rfloor} (x)  \quad (x \in \mathbb{T}) $$
and
$$ \widetilde{p}_{k,J} (x) : = e^{i (\widetilde{\lambda}_k -1 ) x} K_{\lfloor |J|^{-1} \rfloor} (x)  \quad (x \in \mathbb{T}) $$
respectively, where $K_N$ denotes the Fej\'er kernel of order $N$. Observe that $T_{P_{k,J}} $ and $T_{\widetilde{P}_{k,J}}$ are multiplier operators and their symbols  $ P_{k, J} $  and $\widetilde{P}_{k,J}$ are supported in  $ \big\{   \widetilde{\lambda}_{k-1} - \lfloor |J|^{-1} \rfloor , \cdots,  \widetilde{\lambda}_{k-1} + \lfloor |J|^{-1} \rfloor \big\}$ and $ \big\{   \widetilde{\lambda}_k  -1 - \lfloor |J|^{-1} \rfloor , \cdots,  \widetilde{\lambda}_k -1  +  \lfloor |J|^{-1} \rfloor \big\}$, respectively. The symbol of the multiplier operator  $T_{Q_{k,J}}$ is given by
$$ Q_{k,J} (n) : = \chi_{I_k} (n) \big[ 1 - P_{k,J} (n) - \widetilde{P}_{k,J} (n) \big], \quad n \in \mathbb{Z}. $$
Hence, if we set 
$$ c_{k,J} (x) := a_k (x)  \chi_{2J} (x) ( b_{k,J} \ast \phi_{I_k} \ast \phi_{j_k}) (x) \quad (x \in \mathbb{T}) , $$
 we have 
\begin{equation}\label{dec_ABA'}
 I^{(ii)}_2 \leq  A^{(ii)}_2 +  B^{(ii)}_2 + \widetilde{A}^{(ii)}_2,
 \end{equation}
where
$$ A^{(ii)}_2 :=  \Big| \Big\{ x \in \mathbb{T} : \Bigg|  \sum_{\substack{k \in \mathbb{N}, J \in \mathcal{J}: \\ |I_k| | J| > 1 }} r_k (\omega)  \Delta_k^{(\widetilde{\Lambda}, +)} \big( T_{P_{k,J}} ( c_{k,J} ) \big)  (x) \Bigg| > \frac{\alpha} {48} \Big\} \Big|   , $$
$$ B^{(ii)}_2 :=  \Big| \Big\{ x \in \mathbb{T} : \Bigg|  \sum_{\substack{k \in \mathbb{N}, J \in \mathcal{J}: \\ |I_k| | J| > 1 }} r_k (\omega) T_{Q_{k,J}} ( c_{k,J} ) (x) \Bigg| > \frac {\alpha} {48} \Big\} \Big|  ,$$
and
$$ \widetilde{A}^{(ii)}_2 :=  \Big| \Big\{ x \in \mathbb{T} : \Bigg|  \sum_{\substack{k \in \mathbb{N}, J \in \mathcal{J}: \\ |I_k| | J| > 1 }} r_k (\omega) \Delta_k^{(\widetilde{\Lambda}, +)} \big( T_{\widetilde{P}_{k,J}} ( c_{k,J} ) \big)  (x) \Bigg| > \frac{\alpha} {48} \Big\} \Big|. $$

To estimate $A^{(ii)}_2$, note that by using Chebyshev's inequality and Parseval's identity twice, as well as the fact that the for every given $n \in \mathbb{Z}$ there are at most $3$ non-zero terms $ ( T_{P_{k,J}} (c_{k,J}) )^{\widehat{ \ } } (n)$ for $|I_k| |J| > 1$, one gets
\begin{equation}\label{bound_A}
 A^{(ii)}_2 \leq   \frac{C}{\alpha^2} \sum_{k \in \mathbb{N}} \Big\| \sum_{\substack{ J \in \mathcal{J}:\\ |I_k| |J| >1}} T_{P_{k,J} } (c_{k,J})\Big\|^2_{L^2 (\mathbb{T})} 
 \end{equation}
 for some absolute constant $C>0$. The right-hand side of \eqref{bound_A} will essentially be estimated as in the previous case. More precisely, we have the following pointwise bound. 
 
%%%%
 \begin{lemma} 
If  $T_{P_{k,J}}$  and $c_{k,J}$ are as above, then one has
\begin{equation}\label{claim_P}
 |T_{P_{k,J} } (c_{k,J}) (x)| \leq C_0 \|b_{k,J} \|_{L^1 (\mathbb{T})} |J|^{-1} [ M_c (\chi_J) (x) ]^2 \quad \mathrm{for} \ \mathrm{all}\ x\in \mathbb{T} ,
 \end{equation}
where $C_0>0$ is an absolute constant. 
\end{lemma}
%%%%
 
%%%%
\begin{proof} First of all, note that since $c_{k,J}$ is supported in $2J$, it follows from the definition of $T_{P_{k,J}}$ that we may write
\begin{equation}\label{trivial_P}
 |T_{P_{k,J} } (c_{k,J}) (x)| \leq (2 \pi)^{-1} \int_{2J} |c_{k,J} (y) | K_{\lfloor |J|^{-1} \rfloor} (x-y) dy. 
 \end{equation}
To prove \eqref{claim_P}, we shall consider two cases; $x \in 2J$ and $x \notin 2J$. 

If $ x \in 2J$ then we use the trivial bound $\| K_{\lfloor |J|^{-1} \rfloor} \|_{L^{\infty} (\mathbb{T})} \lesssim |J|^{-1}  $ and so, \eqref{trivial_P} gives
$$  |T_{P_{k,J} } (c_{k,J}) (x)| \lesssim |J|^{-1} \| c_{k,J} \|_{L^1 (\mathbb{T})} . $$
Hence, by first using the bound $ \| a_k \|_{L^{\infty} (\mathbb{T})} \lesssim 1 $ and then Young's inequality combined with the facts that $ \| \phi_{I_k} \|_{L^1 (\mathbb{T})} \lesssim 1$ and $ \| \phi_{j_k} \|_{L^1 (\mathbb{T})} \lesssim 1$, one gets
$$ |T_{P_{k,J} } (c_{k,J}) (x)| \lesssim  |J|^{-1}   \| b_{k,J} \|_{L^1 (\mathbb{T})} . $$
Since $ M_c ( \chi_J ) (x) \sim 1$ when $x \in 2J$, the desired inequality \eqref{claim_P} follows from the last estimate. 

In the case where $x \notin 2J$, by using \eqref{trivial_P} and the definition of the Fej\'er kernel, then one deduces, by arguing as in the proof of Lemma \ref{pointwise_bound_second_term}, that
\begin{equation}\label{est_P}
 |T_{P_{k,J} } (c_{k,J}) (x)| \lesssim   \| c_{k,J} \|_{L^1 (\mathbb{T})}  \lfloor |J|^{-1} \rfloor^{-1}   \big[\mathrm{dist} (x,J) \big]^{-2} .
 \end{equation}
Since $\| c_{k,J} \|_{L^1 (\mathbb{T})} \lesssim \| b_{k,J} \|_{L^1 (\mathbb{T})}  $ and $M_c (\chi_J) (x) \sim |J|   [ \mathrm{dist} (x, J) ]^{-1}$ when $x \notin  2J$, the desired estimate \eqref{claim_P} follows from \eqref{est_P}.
\end{proof}
%%%%

By using \eqref{claim_P}, \eqref{bound_A} gives
$$ A^{(ii)}_2 \lesssim \frac{1}{\alpha^2} \sum_{k \in \mathbb{N}} \int_{\mathbb{T}} \Bigg( \sum_{J \in \mathcal{J}} \| b_{k,J} \|_{L^1 (\mathbb{T})} |J|^{-1} \big[  M_c ( \chi_J ) (x) \big]^2  \Bigg)^2 dx $$
and so, by arguing as in the previous case, namely by using a periodic version of \cite[Theorem 1 (i)]{F-S} (this time for $r=2$ and $p=4$) as well as  \eqref{prop_B_2}, we deduce that
$$ A^{(ii)}_2 \lesssim  \sum_{J \in \mathcal{J}} |J| .$$
It thus follows from the last etimate combined with \eqref{arcs} that
\begin{equation}\label{final_boundA}
A^{(ii)}_2 \lesssim \frac{1} {\alpha}  \big\| \underline{\widetilde{F} }  \big\|_{L^1 (\mathbb{T} ; \ell^2 (\mathbb{N}))} ,
\end{equation}
where the implied constant does not depend on $\omega$, $\Lambda$, $\alpha$, and $f$. Similarly, one proves that
\begin{equation}\label{final_boundA'}
\widetilde{A}^{(ii)}_2 \lesssim \frac{1} {\alpha}  \big\| \underline{\widetilde{F} }  \big\|_{L^1 (\mathbb{T} ; \ell^2 (\mathbb{N}))} ,
\end{equation}
where the implied constant does not depend on $\omega$, $\Lambda$, $\alpha$, and $f$.

To obtain an appropriate estimate for $ B^{(ii)}_2$, observe that by using Chebyshev's inequality, the triangle inequality as well as by exploiting the translation-invariance of the operator
$$ \sum_{\substack{k \in \mathbb{N} :\\ | I_k | |J| >1}} r_k (\omega) T_{Q_{k,J}},  $$
one gets
\begin{equation}\label{bound_B_2'}
 B^{(ii)}_2 \leq 4 \sum_{J \in \mathcal{J} } | J | + \frac{48}{\alpha} \sum_{J \in \mathcal{J} } \Bigg\| \sum_{\substack{k \in \mathbb{N} :\\ | I_k | |J| >1}}  r_k (\omega) T_{Q_{k,J}} ( \widetilde{c}_{k,J} )   \Bigg\|_{L^1  ( \mathbb{T} \setminus (2 \widetilde{J}) )} ,
 \end{equation}
where $\widetilde{J}$ denotes the arc  $[-|J|, |J|]$ in $\mathbb{T}$ (noting that by our construction $|J| \leq 1/8$) and $\widetilde{c}_{k,J}$ is a translated copy of $c_{k,J}$ such that  $\mathrm{supp} (\widetilde{c}_{k,J}) \subseteq \widetilde{J}$. We thus see that, in view of \eqref{arcs}, it suffices to handle the second term on the right-hand side of \eqref{bound_B_2'}. To this end,  fix a Schwartz function $s$ supported in $[-1,1]$ with $s (0) = 1$ and set $s_J (x) : = s (|J|^{-1} x)$. Hence,  $\mathrm{supp} (s_J) \subseteq \widetilde{J} =  [-  |J|, |J|]$ and $s_J (0) = 1$. So, if $\sigma_J$ denotes the periodisation of $s_J$, then $\mathrm{supp} (\sigma_J) \subseteq  \widetilde{J}$ and moreover,  the Fourier coefficients of $\sigma_J$ satisfy the properties
\begin{equation}\label{prop_s_1}
 \widehat{\sigma_J} (n) = |J| \widehat{s} (|J| n) \quad \mathrm{for}\ \mathrm{all}\ n \in \mathbb{Z}
 \end{equation}
and
\begin{equation}\label{prop_s_2}
 \sum_{n \in \mathbb{Z}} \widehat{\sigma_J} (n) = 1. 
 \end{equation}
Consider now the multiplier operator $T_{R_{k,J}}$ whose symbol $R_{k,J}$ is given by
$$ R_{k,J} (n) : = Q_{k,J} (n) - \big( Q_{k,J} \ast \widehat{ \sigma_J } \big) (n), \quad n \in \mathbb{Z} .$$
Notice that  if $q_{k,J}$ denotes the kernel of $T_{Q_{k,J}}$, then one has
\begin{align*} T_{R_{k,J}} (\widetilde{c}_{k,J}) (x) &= (2 \pi)^{-1} \int_{\widetilde{J}}  \widetilde{c}_{k,J} (y) q_{k,J} (x-y) \big[ 1  - \sigma_J (x-y) \big] dy \\
&= (2 \pi)^{-1} \int_{\widetilde{J}}  \widetilde{c}_{k,J} (y) q_{k,J} (x-y)  dy
\end{align*}
for all $x \in \mathbb{T} \setminus ( 2 \widetilde{J} )$ and so, it is enough to show that for every $J \in \mathcal{J}$ one has
\begin{equation}\label{suf_B_2}
\Bigg\| \sum_{\substack{k \in \mathbb{N} :\\ | I_k | |J| >1}}  r_k (\omega) T_{R_{k,J}} ( \widetilde{c}_{k,J} ) \Bigg\|_{L^1 ( \mathbb{T} \setminus (2 \widetilde{J} ) )} \leq D_0 |J| \alpha ,
\end{equation}
where $D_0 >0$ is an absolute constant. Towards this aim, consider the trigonometric polynomial 
$$ \gamma (x) : = e^{i  x} - 1  \quad (x \in \mathbb{T})$$
and observe that, by using the Cauchy-Schwarz inequality, the left-hand side of \eqref{suf_B_2} is majorised by
\begin{align*}
  &  \| ( \gamma )^{-1} \|_{L^2  ( \mathbb{T} \setminus (2 \widetilde{J})  )}  \Bigg\|  \gamma  \cdot \Bigg( \sum_{\substack{k \in \mathbb{N} :\\ | I_k | |J| >1}}  r_k (\omega)  T_{R_{k,J}} ( \widetilde{c}_{k,J} )   \Bigg) \Bigg\|_{L^2 ( \mathbb{T} )} \\
  & \lesssim |J|^{-1/2 } \Bigg\|  \gamma \cdot  \Bigg( \sum_{\substack{k \in \mathbb{N} :\\ | I_k | |J| >1}}  r_k (\omega)  T_{R_{k,J}} ( \widetilde{c}_{k,J} ) \Bigg) \Bigg\|_{L^2 ( \mathbb{T} )}.
\end{align*}
Hence, by using Parseval's theorem, one has
\begin{equation}\label{suf_R}
 \Bigg\| \sum_{\substack{k \in \mathbb{N} :\\ | I_k | |J| >1}}  r_k (\omega) T_{R_{k,J}} ( \widetilde{c}_{k,J} )  \Bigg\|_{L^1 ( \mathbb{T} \setminus (2 \widetilde{J} )  )} \lesssim |J|^{-1/2}  B_J ,
\end{equation}
where
$$ B_J : =  \Bigg[ \sum_{ n \in \mathbb{Z} } \Bigg( \sum_{\substack{k \in \mathbb{N} :\\ | I_k | |J| >1} } \big|  R_{k,J} (n +1 )  ( \widetilde{c}_{k,J}  )^{ \widehat{ \ } } (n +1) - R_{k,J} (n)  ( \widetilde{c}_{k,J}  )^{ \widehat{ \ } } (n) \big| \Bigg)^2  \Bigg]^{1/2}.  $$
We shall prove that there exists an absolute constant $C_0 >0$ such that
\begin{equation}\label{bound_BJ}
B_J \leq C_0 |J|^{3/2} \alpha
\end{equation}
for every $J \in \mathcal{J}$. To show \eqref{bound_BJ}, given a $k \in \mathbb{N}$ such that $|I_k| |J| > 1$, consider the following subsets of $I_k $, 
$$ L_{k,J} : = \big\{ \widetilde{\lambda}_{k-1} , \cdots,  \widetilde{\lambda}_{k-1}  +  \lfloor   | J |^{-1} \rfloor \big\} $$
and
$$ L'_{k,J} : = \big\{ \widetilde{\lambda}_k - 1- \lfloor | J |^{-1} \rfloor, \cdots, \widetilde{\lambda}_k -1 \big\} .$$ 
The desired estimate \eqref{bound_BJ} will easily be obtained by using the following lemma.

%%%%
\begin{lemma}\label{R_bounds}
Let $k \in \mathbb{N}$ and $J \in \mathcal{J}$  be such that $|I_k| |J| >1$.
If $L_{k,J}$, $L'_{k,J}$ are as above, then for every $n \in \mathbb{Z}$ one has
\begin{equation}\label{R_bound_pointwise}
 | R_{k,J} (n) | \lesssim \big[ M_d (\chi_{L_{k,J}}) (n) \big]^2 + \big[ M_d (\chi_{L'_{k,J}}) (n) \big]^2 
 \end{equation}
 and
 \begin{equation}\label{R_bound_pointwise_dif}
 | R_{k,J} (n+1) - R_{k,J} (n) | \lesssim |J| \big(  \big[ M_d (\chi_{L_{k,J}}) (n) \big]^2 + \big[ M_d (\chi_{L'_{k,J}}) (n) \big]^2 \big) ,
 \end{equation}
where  the implied constants do not depend on $I_k$, $J$. Here, $M_d$ denotes the centred discrete maximal function.  
\end{lemma}
%%%%

%%%%
\begin{proof} We shall establish \eqref{R_bound_pointwise} first. To this end, fix an arbitrary $n \in \mathbb{Z}$ and  consider the following cases; $n \notin I_k$ and $n \in I_k$. 

\textbf{Case 1.} If $n \notin I_k$, then either $n < \widetilde{\lambda}_{k-1}$ or $n \geq \widetilde{\lambda}_k$. Assume first that $n <  \widetilde{\lambda}_{k-1}$. Since $Q_{k,J}$ is supported in $I_k$, one has
$$ R_{k,J} (n) = \sum_{l \in I_k}  Q_{k,J} ( l )  \widehat{ \sigma_J } (n - l ). $$
Moreover, note that since $s $ is a Schwartz function, we deduce from \eqref{prop_s_1} that there exists an absolute constant $A_0>0$ such that
\begin{equation}\label{sigma_decay}
 | \widehat{\sigma_J} ( l ) | \leq A_0 \frac{|J|}{ (1+|J|^2 l^2 )^2 } \quad \mathrm{for}\ \mathrm{every} \ l \in \mathbb{Z}.
 \end{equation}
Hence, by using the fact $\| Q_{k,J} \|_{\ell^{\infty} (\mathbb{Z})} = 1$, it follows from \eqref{sigma_decay} that
\begin{align*}
 |R_{k,J} (n) | &\leq  A_0  \frac{ |J|^{-2} } {\big[ \mathrm{dist} (n, L_{k,J}) \big]^2}   \sum_{ l \in I_k } \frac{|J|}{ 1+|J|^2 ( l - n)  ^2  }  \\
 & \leq  A_0'  \big[ M_d (\chi_{L_{k,J}}) (n) \big]^2   \sum_{ l \geq n } \frac{|J|}{ 1+|J|^2 ( l - n)  ^2 } ,
\end{align*}
where $A_0' >0 $ is an absolute constant and in the last step we used the fact if $A$ is an ``interval" in $\mathbb{Z}$ and $ n \notin A$ then one has $ M_d (A) (n) \sim \# ( A ) ( \mathrm{dist} (n, A) )^{-1} $. Hence, by using the last estimate,  together with the elementary bound
$$ \sum_{ l \geq n } \frac{|J|}{ 1+|J|^2 ( l - n)  ^2 }  \leq 1 +  \int_n^{\infty }  \frac{|J|}{ 1+|J|^2 ( x - n)  ^2 }  dx = 1 + \int_0^{\infty} \frac{1}{u^2+1} du =1 + \frac{\pi} {2}, $$
it follows that $ |R_{k,J} (n) | $ is $O (  [ M_d (\chi_{L_{k,J}}) (n) ]^2 )$ and so, \eqref{R_bound_pointwise} holds when $n < \lambda_{k-1}$. If $n \geq \lambda_k$, then a similar argument shows that $ |R_{k,J} (n) | $ is $ O( [ M_d (\chi_{L'_{k,J}}) (n) ]^2 ) $ and hence, \eqref{R_bound_pointwise} holds when $n \notin I_k$.

\textbf{Case 2.} We shall now prove \eqref{R_bound_pointwise} in  the case where $n \in I_k$. Towards this aim, observe that if $n \in L_{k,J} \cup L'_{k,J}$, then  
$$ |R_{k,J} (n) | \leq 1 +  \sum_{m \in \mathbb{N}} | \sigma_J (m) | \lesssim 1 \lesssim \big[ M_d (\chi_{L_{k,J}}) (n) \big]^2 + \big[ M_d (\chi_{L'_{k,J}}) (n) \big]^2, $$
where we used \eqref{sigma_decay} and the fact that if $A$ is an ``interval" in $\mathbb{Z}$ and $ n \in A$ then one has $M_d (\chi_A) (n)  \sim 1$. 

If we now assume that $I_k \setminus \big( L_{k,J} \cup L'_{k,J} \big) \neq \emptyset$ and $ n \in I_k \setminus \big( L_{k,J} \cup L'_{k,J} \big) $, then $Q_{k,J} (n) = 1$ and so, \eqref{prop_s_2} gives
\begin{align*}
 | R_{k,J} (n) | &\leq  \sum_{l < \widetilde{\lambda}_{k-1} } | \widehat{\sigma_J} (n-l) | +  \sum_{l \in L_{k,J}} 
 | 1 - Q_{k,J} (l) | | \widehat{\sigma_J} (n-l)  |  \\
  &+ \sum_{l \in L'_{k,J} } | 1 - Q_{k,J} (l) | | \widehat{\sigma_J} (n-l) | + \sum_{l \geq \widetilde{\lambda}_k } |  \widehat{\sigma_J} (n-l)  |.
\end{align*}
To bound the second term, note that by using \eqref{sigma_decay} and the fact that $ \| Q_{k,J} \|_{\ell^{\infty} (\mathbb{Z})} = 1$, one has
\begin{align*}
  \sum_{l \in L_{k,J}}  | 1 - Q_{k,J} (l) | |  \widehat{\sigma_J} (n-l)  |  \lesssim \sum_{ l \in L_{k,J}} \frac{|J|}{1 + |J|^2 (n-l)^2 }  &\lesssim  \frac{ |J|^{-1} } { \big[ \mathrm{dist} (n, L_{k,J}) \big]^2}  \# ( L_{k,J} ) \\
 & \sim  \frac{ |J|^{-2}  } { \big[ \mathrm{dist} (n, L_{k,J}) \big]^2 } \\
 & \sim  \big[ M_d ( \chi_{L_{k,J}}  ) (n)  \big]^2,
 \end{align*}
 where the implied constants are independent of $k$ and $J$. The third term is handled similarly and is $O [ M_d ( \chi_{L'_{k,J}} ) (n) ]^2 $. For the first term, notice that by using \eqref{sigma_decay} one has
\begin{align*}
  \sum_{l < \widetilde{\lambda}_{k-1} }   | \widehat{\sigma_J} (n-l)  |  \lesssim \sum_{ l < \widetilde{\lambda}_{k-1} } \frac{|J|}{( 1 + |J|^2 (n-l)^2 )^2 } & \lesssim \frac{|J|^{-2}}{ \big[ \mathrm{dist} (n, L_{k,J})  \big]^2} \sum_{ m \in \mathbb{Z}} \frac{|J|}{1 + |J|^2 m^2  } \\
 & \lesssim  \big[ M_d ( \chi_{L_{k,J}}  ) (n)  \big]^2.
 \end{align*}
 One shows that the fourth term is $O ([ M_d ( \chi_{L'_{k,J}}  ) (n) ]^2 )$ in a completely analogous way. Hence, by putting all the estimates together, we deduce that \eqref{R_bound_pointwise} also holds in this subcase. 
 
 %%%%
 
We shall now prove \eqref{R_bound_pointwise_dif}. Towards this aim, we  argue as in the proof of \eqref{R_bound_pointwise} and, for a given $n \in \mathbb{Z}$, consider the cases where $n \notin I_k$ and $n \in I_k$. 
 
\textbf{Case 1'.} If $n \notin I_k$, then $Q_{k,J} (n+1) = Q_{k,J} (n) =0$ and so, one has
\begin{align*}
   | R_{k,J} (n+1) - R_{k,J} (n) | &= \Bigg|  \sum_{l \in I_k}  Q_{k,J} (l) \big[  \widehat{\sigma_J} (n+1 - l) - \widehat{\sigma_J} (n-l) \big]  \Bigg|  \\
   & \leq \sum_{l \in I_k}  | \widehat{\sigma_J} (n+1 - l) - \widehat{\sigma_J} (n-l) | .
\end{align*}
 Notice that, since for all $k \in \mathbb{Z}$ one has $\widehat{\sigma_J} (k) = |J| \widehat{s} (|J| k)$ and $s$ is a Schwartz function, it follows from the last estimate and the mean value theorem that
 \begin{equation}\label{dif_R_bound_1}
   | R_{k,J} (n+1) - R_{k,J} (n) | \leq A  |J| \sum_{l \in I_k } \frac{|J|}{(1 + |J|^2 (n-l)^2 )^2} , 
\end{equation}
 where $A>0$ is an absolute constant. Therefore, by arguing as in case 1, if $n < \widetilde{\lambda}_{k-1}$, the quantity on the right-hand side of \eqref{dif_R_bound_1} is $O ( |J|  [ M_d (\chi_{L_{k,J}}) (n)  ]^2 )$, whereas if $n \geq \widetilde{\lambda}_k$, the quantity on the right-hand side of \eqref{dif_R_bound_1} is $O( |J| [ M_d (\chi_{L'_{k,J}}) (n) ]^2 )$. We thus deduce that \eqref{R_bound_pointwise_dif} holds when $n \notin I_k$. 
 
 \textbf{Case 2'.} If $n \notin I_k$, then either $n \in L_{k,J} \cup L'_{k,J}$ or $n  \in I_k \setminus ( L_{k,J} \cup L'_{k,J} )$ (assuming that $ I_k \setminus \big( L_{k,J} \cup L'_{k,J} \big) \neq \emptyset$). 
 
 If $n \in L_{k,J} \cup L'_{k,J}$, we have  $ [ M_d (\chi_{L_{k,J}}) (n)  ]^2 + [ M_d (\chi_{L'_{k,J}}) (n)  ]^2 \sim 1$ and hence,
\begin{equation}\label{dif_Q}
 | Q_{k, J} (n+1) - Q_{k,J} (n) | \lesssim |J| \sim |J| \big( \big[ M_d (\chi_{L_{k,J}}) (n) \big]^2 + \big[ M_d (\chi_{L'_{k,J}}) (n) \big]^2 \big) ,
\end{equation}
where the implied constants do not depend on $k$, $J$. As in case 1', by using the definition of $\sigma_J$ and the mean value theorem, one has
$$ \sum_{l \in I_k} |Q_{k,J} (l) | | \widehat{\sigma_J} (n+1 - l) - \widehat{\sigma_J} (n-l) |  \leq A |J| \sum_{m \in \mathbb{Z}} \frac{| J |} {(1 + |J|^2 m^2 )^2 } .$$
We thus deduce that the right-hand side of the last inequality is $O(|J|)$, which combined with \eqref{dif_Q}, gives \eqref{R_bound_pointwise_dif} in the subcase
where $n \in L_{k,J} \cup L'_{k,J}$. 

It remains to treat the subcase where $n \in I_k \setminus ( L_{k,J} \cup L'_{k,J} )$. To this end, note that, in view of \eqref{prop_s_2}, we may write
$$ R_{k,J} (n+1) - R_{k,J} (n) = S_{k,J} (n) + T_{k,J} (n) + U_{k,J} (n), $$
where
$$ S_{k,J} (n): = Q_{k,J} (n+1) \sum_{l < \widetilde{\lambda}_{k-1}} \widehat{\sigma_J} (n+1 - l) - Q_{k,J} (n) \sum_{l < \widetilde{\lambda}_{k-1}} \widehat{\sigma_J} (n  - l), $$
$$ T_{k,J} (n): = Q_{k,J} (n+1) \sum_{l \geq \widetilde{\lambda}_k } \widehat{\sigma_J} (n+1 - l) - Q_{k,J} (n) \sum_{l \geq \widetilde{\lambda}_k } \widehat{\sigma_J} (n  - l),  $$
and
$$ U_{k,J} (n): =   \sum_{l  \in I_k } \big\{ [ Q_{k,J} (n+1) - Q_{k,J} (l) ] \widehat{\sigma_J} (n+1 - l) -    [ Q_{k,J} (n ) - Q_{k,J} (l) ] \widehat{\sigma_J} (n  - l) \big\}. $$
For the first term, we have
\begin{align*}
| S_{k,J} (n) | &\leq | Q_{k,J} (n+1) - Q_{k,J} (n) | \sum_{l < \widetilde{\lambda}_{k-1}} | \widehat{\sigma_J} (n+1 - l)|  \\
&+ |Q_{k,J} (n)| \sum_{l < \widetilde{\lambda}_{k-1}} |\widehat{\sigma_J} (n +1 - l) - \widehat{\sigma_J} (n - l )|
\end{align*}
and, by arguing as above, we see that both terms on the right-hand side of the last estimate are $O(|J| [M_d (\chi_{L_{k,J}}) (n)]^2)$. Similarly, one shows that $|T_{k,J} (n)|$ is $O(|J| [M_d (\chi_{L'_{k,J}}) (n)]^2)$. For the last term, observe that if $n  \in I_k \setminus ( L_{k, J} \cup L'_{k,J} )$, then $Q_{k,J} (n+1) = Q_{k,J} (n) =1$  and so,
$$ |U_{k,J} (n)| \leq \sum_{ l \in L_{k,J} \cup L'_{k,J}} |1- Q_{k,J} (l) | | \widehat{\sigma_J} (n + 1 - l) - \widehat{\sigma_J} (n - l ) | .$$
Hence, by arguing as above, one deduces that 
$$ |U_{k,J} (n)| \lesssim |J| \big(  \big[M_d (\chi_{L_{k,J}}) (n) \big]^2 +  \big[ M_d (\chi_{L'_{k,J}}) (n) \big]^2 \big) , $$
where the implied constants do not depend on $k$, $J$. Therefore, \eqref{R_bound_pointwise_dif} also holds in this subcase and hence, the proof of the lemma is complete.
 \end{proof}
%%%%

Going now back to the proof of \eqref{bound_BJ}, observe that for every $n \in \mathbb{Z}$ one has  
$$ \big|  ( \widetilde{c}_{k,J})^{ \widehat{ \ } }  (n +1) -   ( \widetilde{c}_{k,J}  )^{ \widehat{ \ } }  (n ) \big| \leq \int_{\widetilde{J}} |e^{ix}  - 1 | |\widetilde{c}_{k,J} (x)| dx \leq 2 |J| \| c_{k,J} \|_{L^1 (\mathbb{T})}  \lesssim |J| \| b_{k,J} \|_{L^1 (\mathbb{T})} $$
and $  | ( \widetilde{c}_{k,J}  )^{ \widehat{ \ } } (n) | \lesssim \| b_{k,J} \|_{L^1 (\mathbb{T})}$. It thus follows from Lemma \ref{R_bounds}  that there exists an absolute constant $C_0>0$ such that
$$ B_J \leq C_0 \big( B_J^{(1)} + B_J^{(2)} \big), $$
where
$$ B_J^{(1)} :=  |J| \Bigg( \sum_{n \in \mathbb{Z} } \Bigg[  \sum_{ \substack{k \in \mathbb{N}: \\ | I_k | | J | > 1 } } \big[ M_d \big( \| b_{k, J} \|^{1/2}_{L^1 (\mathbb{T})} \chi_{L_{k,J}} \big) (n) \big]^2  \Bigg]^2 \Bigg)^{1/2}   
$$
and
$$ B_J^{(2)} :=  |J| \Bigg( \sum_{n \in \mathbb{Z} } \Bigg[  \sum_{ \substack{k \in \mathbb{N}: \\ | I_k | | J | > 1 } } \big[ M_d \big( \| b_{k, J} \|^{1/2}_{L^1 (\mathbb{T})} \chi_{L'_{k,J}} \big) (n) \big]^2  \Bigg]^2 \Bigg)^{1/2} .
$$
By using a version of the Fefferman-Stein theorem \cite{F-S} for the discrete maximal operator $M_d$; see e.g.  \cite[Theorem 1.2]{G-L-Y} as well as the fact that the ``intervals'' $L_{k,J}$ are mutually disjoint and have cardinality $O(|J|^{-1})$,  we deduce that
\begin{align*}
 B_J^{(1)}  &\lesssim |J| \Bigg( \sum_{n \in \mathbb{Z} }  \Bigg[  \sum_{ \substack{k \in \mathbb{N}: \\ | I_k | | J | > 1 } }  \| b_{k, J} \|_{L^1 (\mathbb{T})}  \chi_{L_{k,J}}   (n)  \Bigg]^2  \Bigg)^{1/2}   \\
 &=  |J| \Bigg( \sum_{n \in \mathbb{Z} }  \sum_{ \substack{k \in \mathbb{N}: \\ | I_k | | J | > 1 } }  \| b_{k, J} \|^2_{L^1 (\mathbb{T})} \chi_{L_{k,J}} (n) \Bigg)^{1/2} \sim  |J|^{1/2} \Bigg(  \sum_{ \substack{k \in \mathbb{N}: \\ | I_k | | J | > 1 } }  \| b_{k, J} \|^2_{L^1 (\mathbb{T})} \Bigg)^{1/2} , 
 \end{align*}
 where in the above chain of inequalities the implied constants are independent of $k$ and $J$. Similarly, one shows that
$$ B_J^{(2)} \lesssim | J |^{1/2}  \Bigg(  \sum_{ \substack{k \in \mathbb{N}: \\ | I_k | | J | > 1 } }  \| b_{k, J} \|^2_{L^1 (\mathbb{T})} \Bigg)^{1/2}  $$
and we thus have 
$$ B_J \leq D_0 | J |^{1/2}  \Bigg(  \sum_{ \substack{k \in \mathbb{N}: \\ | I_k | | J | > 1 } } \| \widetilde{b}_{k, J} \|^2_{L^1 (\mathbb{T})}    \Bigg)^{1/2}  $$
where $D_0 >0$ is an absolute constant. Hence, by using the last estimate combined with \eqref{prop_B_2}, one obtains \eqref{bound_BJ}.  

Therefore, by using \eqref{suf_R} and \eqref{bound_BJ}, we deduce that \eqref{suf_B_2} holds. It thus follows from \eqref{arcs}, \eqref{bound_B_2'}, and \eqref{suf_B_2} that
$$ B_2^{(ii)} \lesssim \frac{1}{\alpha} \big\| \underline{\widetilde{F} }  \big\|_{L^1 (\mathbb{T} ; \ell^2 (\mathbb{N}))} $$
and so, by combining the last estimate with \eqref{final_boundA}, \eqref{final_boundA'}, and \eqref{dec_ABA'}, we obtain \eqref{bound_I2_b}. Hence, the proof of \eqref{bound_I2} is complete.

We have thus shown \eqref{w-t_alpha_+}. One establishes \eqref{w-t_alpha_-} in a completely analogous way and hence, \eqref{w-t_alpha} holds for every $\alpha >0$. Therefore, the proof of \eqref{main_w-t} is complete. 

%%%%%%%%%%%%%%%%%%%%%%%%%%%%
%%%%%%%%%%%%%%%%%%%%%%%%%%%%
%%%%%%%%%%%%%%%%%%%%%%%%%%%%
%%%%%%%%%%%%%%%%%%%%%%%%%%%%
\section{Variants and remarks in the one-dimensional case}\label{finite_union}

The argument of Tao and Wright \cite{TW} adapted to the torus, as presented in the previous section, can actually be used to handle Littlewood-Paley square functions formed with respect to finite unions of lacunary sequences in $\mathbb{N}$. To be more specific, observe that an increasing sequence $\Lambda$ of positive integers can be written as a finite union of lacunary sequences in $\mathbb{N}$ if, and only if, the quantity
\begin{equation}\label{structural_constant}
\sigma_{\Lambda} : = \sup_{N \in \mathbb{N}  } \# \big( \Lambda \cap \{  2^{N-1}, \cdots, 2^N  \} \big)  
\end{equation}
is finite.  Hence, given a finite union of lacunary sequences    $\Lambda$,  if $f$ is an arbitrary trigonometric polynomial on $\mathbb{T}$ and $(F_j )_{j \in \mathbb{N}}$ and $(\widetilde{F}_k)_{k \in \mathbb{N}} $ are as in the previous section, it follows that there exists an absolute constant $A>0$ such that
\begin{equation}\label{observation_2}
  \Bigg(\sum_{k \in  \mathbb{N}} \big| \widetilde{F}_k  (x) \big|^2 \Bigg)^{1/2} \leq A   \sigma_{\Lambda}^{1/2}  \Bigg(\sum_{j \in  \mathbb{N}} |F_j (x) |^2 \Bigg)^{1/2}   \quad \mathrm{for}\ \mathrm{all}\  x \in \mathbb{T}.
  \end{equation}
 Note that \eqref{observation_2} can be regarded as the analogue of \eqref{observation} in this case. Therefore, by arguing exactly as in Subsection \ref{details}, one shows that
 \begin{equation}\label{w-t_union}
 \| S^{(\Lambda)} (f) \|_{L^{1, \infty} (\mathbb{T})} \lesssim \sigma_{\Lambda}^{1/2}   \Bigg[ 1 + \int_{\mathbb{T}} |f(x) | \log^{1/2} (e + |f(x)|) dx \Bigg].
 \end{equation}
 Hence, by using interpolation (where in this case, one uses \eqref{w-t_union} instead of \eqref{final_w-t}) and Tao's converse extrapolation theorem \cite{Tao}, one deduces that $ \| S^{(\Lambda)}   \|_{L^p (\mathbb{T}) \rightarrow L^p (\mathbb{T})} $ is $O( (p-1)^{-3/2}   \sigma_{\Lambda}^{1/2} )$ for $p \in (1,2) $. Similarly, by arguing as in Remark \ref{second_proof}, one shows that the $H^p_A (\mathbb{T}) \rightarrow L^p (\mathbb{T})$ operator norm of $S^{(\Lambda)}$ is   $O ( (p-1)^{-1}   \sigma_{\Lambda}^{1/2} )$ for $p \in (1,2)  $. We thus obtain the following theorem.

\begin{theo}\label{union_variant}
There exists an absolute constant $C_0>0$ such that for every strictly increasing sequence $\Lambda  = (\lambda_j)_{j \in \mathbb{N}_0}$   in $\mathbb{N}$ with $ \sigma_{\Lambda} < \infty $ one has
\begin{equation}\label{variant_Hp}
\sup_{\substack{ \| f \|_{L^p (\mathbb{T})} = 1 \\ f \in H^p_A (\mathbb{T}) } } \| S^{(\Lambda)} (f) \|_{L^p (\mathbb{T})} \leq  \frac{C_0}{p-1}   \sigma_{\Lambda}^{1/2}
\end{equation}
and
\begin{equation}\label{variant_Lp}
   \| S^{(\Lambda)}   \|_{L^p (\mathbb{T}) \rightarrow L^p (\mathbb{T})} \leq  \frac{C_0}{(p-1)^{3/2}}   \sigma_{\Lambda}^{1/2}
\end{equation}
for every $1<p <2$.
\end{theo}

If $\Lambda$ is a lacunary sequence in $\mathbb{N}$ with ratio $\rho_{\Lambda} \in (1,2)$ then one can easily check that 
$$\sigma_{\Lambda} \lesssim (\rho_{\Lambda} - 1 )^{-1}$$
and hence, \eqref{main_result} and \eqref{main_result_2} are direct corollaries of \eqref{variant_Hp} and \eqref{variant_Lp}, respectively.

 Moreover, by arguing as in Subsection \ref{Sharpness}, one can show that both estimates in Theorem \ref{union_variant} are optimal in the sense that 
  for every $\sigma \in \mathbb{N}$ ``large'', there exist strictly increasing sequences $\Lambda_1$, $\Lambda_2$ of positive integers with $\sigma_{\Lambda_i} = \sigma$ ($i=1,2$) and $p \in (1,2)$ such that
 \begin{equation}\label{lower_bound_main_1}
   \sup_{\substack{ \| f \|_{L^p (\mathbb{T})} = 1 \\ f \in H^p_A (\mathbb{T}) } }   \| S^{(\Lambda_1)} (f) \|_{L^p (\mathbb{T})} \gtrsim  \sigma^{1/2}   (p-1)^{-1}  
   \end{equation}
  and
   \begin{equation}\label{lower_bound_main_2}
    \| S^{(\Lambda_2)}   \|_{L^p (\mathbb{T}) \rightarrow L^p (\mathbb{T})}  \gtrsim  \sigma^{1/2}  (p-1)^{-3/2}   ,
    \end{equation}
  respectively.
 Indeed, to prove \eqref{lower_bound_main_1} fix a ``large'' $\sigma \in \mathbb{N}$ and choose $ p : = 1 + \sigma^{-1} $. Also, let $N \in \mathbb{N}$ be such that $2^{N-1} \leq e^{\sigma} < 2^N$ and set $M : = 2^{2N} $. We thus have
 \begin{equation}\label{initial_cond_p}
  \log M \sim (p-1)^{-1}.
  \end{equation}
 We shall now construct an increasing sequence $\Lambda = (\lambda_j)_{j \in \mathbb{N}_0}$ such that $\sigma_{\Lambda} = \sigma$ and  \eqref{lower_bound_main_1}  holds. To this end, define $\lambda_j := M + j \big\lfloor M/ \sigma \big \rfloor$ for $j =0, \cdots, \sigma -1$ and $\lambda_{j+1} := 3 \lambda_j $ for all $j \geq \sigma$. Hence, if we set
 $$   s_{\Lambda} (k) := \#  \big(  \big\{ j \in \mathbb{N}_0 : 2^{k-1} \leq \lambda_j  < 2^k \big\}  \big),$$
 then $s_{\Lambda} (k) =0$ for all $k \leq  2N  $, $s_{\Lambda} (2N +1 ) = \sigma$ and $s_{\Lambda} (k) \leq 2$ for all $k >  2N +1 $. Therefore, our sequence $\Lambda$ satisfies $\sigma_{\Lambda} = \sigma$. 
 
Next, as in Subsection \ref{Sharpness}, consider the analytic trigonometric polynomial $f_M$ given by $f_M (x) := e^{i (2M+1) x} V_M (x)$, $x \in \mathbb{T}$ and note that, by our choice of $M$ and $p$, one has $\| f_M \|_{L^p (\mathbb{T})} \sim 1$. Hence, by arguing as in Subsection \ref{Sharpness}, one has
  \begin{align*}
  \sup_{\substack{ \| f \|_{L^p (\mathbb{T})} = 1 \\ f \in H^p_A (\mathbb{T}) } }   \| S^{(\Lambda)} (f) \|_{L^p (\mathbb{T})}   \sim   \| S^{(\Lambda)} (f_M) \|_{L^p (\mathbb{T})} & \geq \Bigg(  \sum_{\substack{ j \in \mathbb{N}_0: \\ M \leq \lambda_j < 2M }} \| \Delta^{(\Lambda)}_j (f_M) \|^2_{L^1 (\mathbb{T})} \Bigg)^{1/2} \\
& \gtrsim \Bigg(  \sum_{j=0}^{\sigma-2}  \log^2 (\lambda_{j+1} - \lambda_j)  \Bigg)^{1/2} \\
& \gtrsim \sigma^{1/2} \log M \\
& \sim \sigma^{1/2} (p-1)^{-1},
 \end{align*}
 where in the last step \eqref{initial_cond_p} was used. This completes the proof of \eqref{lower_bound_main_1}. The proof of \eqref{lower_bound_main_2} is completely analogous.

%%%%%%%%
%%%%%%%%
\subsection{Zygmund's classical inequality revisited}
Let  $\Lambda = (\lambda_j)_{j \in \mathbb{N}_0} $ be an increasing sequence that can be written as a finite union of lacunary sequences in $\mathbb{N}$ or, equivalently, $\sigma_{\Lambda} < \infty$ with $\sigma_{\Lambda} \in \mathbb{N}$ being as in \eqref{structural_constant}.  

Notice that \eqref{majorisation}, \eqref{substitute}, \eqref{observation_2}, and the argument on pp. 530--531 in \cite{TW}  yield the following version of a classical inequality due to Zygmund (see Theorem 7.6 in Chapter XII of \cite{Zygmund_book})
\begin{equation}\label{Zygmund_ineq}  
\Bigg( \sum_{j \in \mathbb{N}_0} | \widehat{f} (\lambda_j)|^2 \Bigg)^{1/2} \leq A \sigma_{\Lambda}^{1/2} \Bigg[ 1 + \int_{\mathbb{T}} |f(x)| \log^{1/2} (e + |f(x)|) dx \Bigg] ,
\end{equation}
where $A>0$ is an absolute constant. Therefore, by using \eqref{Zygmund_ineq} and duality; see e.g. Remarque on pp. 350--351 in \cite{Bonami}, one deduces that for every trigonometric polynomial $g$ such that $\mathrm{supp} (g) \subseteq \Lambda$ one has
\begin{equation}\label{Lambda_p}
\| g \|_{L^p (\mathbb{T})}  \leq C   \sigma_{\Lambda}^{1/2}  \sqrt{ p }  \| g \|_{L^2 (\mathbb{T})} \quad \mathrm{for} \ \mathrm{all}\ 2 \leq p < \infty,
\end{equation}
where $C >0$ is an absolute constant.

We remark that the exponent $r=1/2$ in $\sigma_{\Lambda}^{1/2}$ in \eqref{Lambda_p} cannot be improved and therefore, the exponent $r=1/2$ in $\sigma_{\Lambda}^{1/2}$ on the right-hand side of \eqref{Zygmund_ineq} is sharp. Indeed, this can be shown by using the work of Bourgain \cite{Bourgain_Sidon}. To be more specific, since we assume that $\sigma_{\Lambda} < \infty$, we may write $\Lambda$ as a disjoint union of lacunary sequences $\Lambda_1, \cdots, \Lambda_N$ of positive integers with $\rho_{\Lambda_j} \geq 2$ for all $j=1, \cdots, N$ and $ N \sim  \sigma_{\Lambda}$. Since lacunary sequences with ratio greater or equal than $2$ are quasi-independent sets in $\mathbb{Z}$; see \cite[Definition 4]{Bourgain_A_d_Sidon}, it follows that we may decompose $\Lambda$ as a disjoint union of $ N \sim \sigma_{\Lambda}$ quasi-independent sets. 
We thus conclude that for every finite subset $A$ of $\Lambda$ there exists a quasi-independent subset $B$ of $A$ such that 
$$ \#(B) \geq N^{-1} \#(A) \sim   \sigma_{\Lambda}^{-1} \#(A) $$
and hence, \eqref{Lambda_p} as well as its sharpness (in terms of the dependence with respect to $\sigma_{\Lambda}$) follows from \cite[Lemma 1]{Bourgain_Sidon}, see also pp. 101--102 in \cite{Bourgain_A_d_Sidon}.  

For variants of the aforementioned result of Zygmund in higher dimensions, see \cite{Bakas_PZ}, \cite{Bakas_MI} and the references therein. 

%%%%%%%%%%%%%%%%%%%%%%%%%%%%
%%%%%%%%%%%%%%%%%%%%%%%%%%%%
%%%%%%%%%%%%%%%%%%%%%%%%%%%%
%%%%%%%%%%%%%%%%%%%%%%%%%%%%
\section{Higher-dimensional versions of Theorems \ref{lac} and \ref{lac_2}}\label{higher}

In this section we extend Theorem \ref{lac}  to the $d$-torus. To state our result, fix a $d\in \mathbb{N}$ and consider  strictly increasing sequences $ \Lambda_n =  ( \lambda_{j_n}^{(n)}  )_{j_n \in \mathbb{N}_0}$ of positive integers, $n \in \{ 1,\cdots, d \}$. Then, the $d$-parameter Littlewood-Paley square function formed with respect to $\Lambda_n $  ($n \in \{ 1,\cdots, d \}$) is given by
$$ S^{ (\Lambda_1, \cdots, \Lambda_d)} (f) := \Bigg(  \sum_{j_1, \cdots, j_d \in \mathbb{N}_0} \Big| \Delta^{(\Lambda_1)}_{j_1} \otimes \cdots \otimes \Delta^{(\Lambda_d)}_{j_d} (f) \Big|^2  \Bigg)^{1/2} $$
and is initially defined over the class of all trigonometric polynomials $f$ on $\mathbb{T}^d$.

%%%%
\begin{theo}\label{higher_dim} Given a $d \in \mathbb{N}$, there exists an absolute constant $C_d >0$ such that for every $1<p<2$ and for all lacunary sequences $ \Lambda_n  = ( \lambda_{j_n}^{(n)}  )_{j_n \in \mathbb{N}_0}$ in $\mathbb{N}$ with ratio $\rho_{\Lambda_n} \in (1,2)$,  $ n \in \{1,\cdots, d \}$, one has
\begin{equation}\label{main_result_d}
  \sup_{\substack{ \| f \|_{L^p (\mathbb{T}^d)} = 1 \\ f \in H^p_A (\mathbb{T}^d) } } \Big\| S^{ (\Lambda_1, \cdots, \Lambda_d)} (f) \Big\|_{L^p (\mathbb{T}^d)} \leq \frac{C_d}{(p-1)^d} \prod_{n=1}^d  (\rho_{\Lambda_n} - 1)^{-1/2}   
\end{equation}
and
\begin{equation}\label{main_result_d_2}
\Big\|  S^{ (\Lambda_1, \cdots, \Lambda_d)} \Big\|_{L^p (\mathbb{T}^d) \rightarrow L^p (\mathbb{T}^d)} \leq \frac{C_d}{(p-1)^{3d/2}} \prod_{n=1}^d (\rho_{\Lambda_n} - 1)^{-1/2}.
\end{equation}
\end{theo}
%%%%

%%%%
\begin{proof} We shall prove \eqref{main_result_d} first. To this end, the idea is to fix a probability space $(\Omega, \mathcal{A}, \mathbb{P})$ and show that there exists an absolute constant $C>0$ such that for every choice of $\omega \in \Omega$ one has
\begin{equation}\label{randomised}
 \Bigg\| \sum_{j \in \mathbb{N}_0} r_j (\omega) \Delta_j^{(\Lambda)} (g)  \Bigg\|_{L^p (\mathbb{T})}  \leq \frac{C}{p-1}  (\rho_{\Lambda} - 1)^{-1/2}  \| g \|_{L^p (\mathbb{T})} .
 \end{equation}
 for every analytic trigonometric polynomial $g$ on $\mathbb{T}$. Then the proof of \eqref{main_result_d} is obtained by iterating \eqref{randomised}.  
 
To prove \eqref{randomised}, note that by using estimate $(3.2)$ in Bourgain's paper \cite{Bourgain_89}, which follows from Bourgain's extension \cite{Bourgain_85} of Rubio de Francia's theorem \cite{RdF}, one deduces that there exists an absolute constant $M>0$ such that for every $q \in [ 1 , 2 ] $ one has
\begin{equation}\label{Bourgain_ineq}
 \| h \|_{L^q (\mathbb{T})} \leq M \big\| S^{ (\Lambda ) } (h)  \big\|_{L^q (\mathbb{T})}  
 \end{equation}
for every analytic trigonometric polynomial $h$ on $\mathbb{T}$. Hence, in order to establish \eqref{randomised}, fix a trigonometric polynomial $g$ and consider the  trigonometric polynomial $h_{\omega}$ given
by
$$ h_{\omega} : =  \sum_{j \in \mathbb{N}_0} r_j (\omega) \Delta_j^{(\Lambda)} (g). $$
Observe that \eqref{Bourgain_ineq} applied to $h_{\omega}$ implies that
$$ \Bigg\| \sum_{j \in \mathbb{N}_0} r_j (\omega) \Delta_j^{(\Lambda)} (g)  \Bigg\|_{L^p (\mathbb{T})} \leq M \big\| S^{ (\Lambda ) } (g)  \big\|_{L^p (\mathbb{T})} $$
for all $p \in (1,2)$ and for every analytic trigonometric polynomial $g$ on $\mathbb{T}$. Hence, the desired estimate \eqref{randomised} follows from the last inequality combined with \eqref{main_result}. Alternatively, \eqref{randomised} can be obtained by using \eqref{alt_w-t_ineq} and Marcinkiewicz interpolation for analytic Hardy spaces for the operator $T_{\omega}^{(\Lambda)}:= \sum_{j \in \mathbb{N}_0} r_j (\omega) \Delta_j^{(\Lambda)}$, as explained in Section \ref{Proof}.

To complete the proof of \eqref{main_result_d}, fix a $p \in (1,2)$ and take an analytic trigonometric polynomial $f$ on $\mathbb{T}^d$. By iterating \eqref{randomised}, we get
\begin{align*}
& \Bigg\| \sum_{ j_1, \cdots, j_d \in \mathbb{N}_0} r_{j_1} (\omega_1) \cdots r_{j_d} (\omega_d) \Delta^{(\Lambda_1)}_{j_1} \otimes \cdots \otimes \Delta^{(\Lambda_d)}_{j_d} (f)  \Bigg\|_{L^p (\mathbb{T}^d)}  \leq \\
&\frac{C^d}{(p-1)^d} \Bigg[ \prod_{n=1}^d  (\rho_{\Lambda_n} - 1)^{-1/2} \Bigg]   \| f \|_{L^p (\mathbb{T}^d)}
\end{align*}
where $C>0$ is the constant in \eqref{randomised}. Hence, by using the last estimate together with multi-dimensional Khintchine's inequality \eqref{Khintchine} and the density of analytic trigonometric polynomials on $\mathbb{T}^d$ in $(H^p_A (\mathbb{T}^d), \| \cdot \|_{L^p (\mathbb{T}^d)})$, \eqref{main_result_d} follows.

The proof of \eqref{main_result_d_2} is similar. Indeed, by iterating \eqref{desired_est}, one gets
\begin{align*}
& \Bigg\| \sum_{ j_1, \cdots, j_d \in \mathbb{N}_0} r_{j_1} (\omega_1) \cdots r_{j_d} (\omega_d) \Delta^{(\widetilde{\Lambda}_1)}_{j_1} \otimes \cdots \otimes \Delta^{(\widetilde{\Lambda}_d)}_{j_d} (f)  \Bigg\|_{L^p (\mathbb{T}^d)}  \leq \\
&\frac{C^d}{(p-1)^{3d/2}} \Bigg[ \prod_{n=1}^d  (\rho_{\Lambda_n} - 1)^{-1/2} \Bigg]    \| f \|_{L^p (\mathbb{T}^d)}
\end{align*}
for every trigonometric polynomial $f $ on $\mathbb{T}^d$, where $\widetilde{\Lambda}_n$ is a strictly increasing sequence in $\mathbb{N}$ associated to $\Lambda_n$ and is constructed as in Lemma \ref{refinement}, $n \in \{ 1, \cdots, d \}$. Hence, the proof of \eqref{main_result_d_2} is obtained by using  the last estimate, \eqref{Khintchine}, the fact that for every trigonometric polynomial $f$ on $\mathbb{T}^d$ one has
$$ S^{( \Lambda_1, \cdots,  \Lambda_d)} (f) (x) \lesssim_d S^{(\widetilde{\Lambda}_1, \cdots, \widetilde{\Lambda}_d)} (f) (x) \quad (x \in \mathbb{T}^d) ,$$
and the density of trigonometric polynomials on $\mathbb{T}^d$ in $(L^p (\mathbb{T}^d), \| \cdot \|_{L^p (\mathbb{T}^d)})$.
\end{proof}
%%%%

%%%%
\begin{rmk} The estimates \eqref{main_result_d} and \eqref{main_result_d_2} in Theorem \ref{higher_dim} can be regarded as refined versions of \cite[Corollary 2]{BRS} and \cite[Proposition 4.1]{Bakas}, respectively. \end{rmk}
%%%%

An adaptation of the argument presented in Subsection \ref{Sharpness} to higher dimensions shows that \eqref{main_result_d} is optimal in the following sense. For every choice of $\lambda $ ``close'' to $1^+$, there exists a lacunary sequence $\Lambda = \big(\lambda_j \big)_{j \in \mathbb{N}_0}$ in $\mathbb{N} $ with ratio $\rho_{\Lambda} \in [ \lambda, \lambda^3)$ and a $p$ ``close" to $1^+$ such that  
$$  \sup_{\substack{ \| f \|_{L^p (\mathbb{T}^d)} = 1 \\ f \in H^p_A (\mathbb{T}^d) } } \Big\|  S^{ (\Lambda, \cdots, \Lambda)} (f) \Big\|_{L^p (\mathbb{T}^d)} \gtrsim_d  \frac{1}{(p-1)^d}  (\lambda -1 )^{-d/2} \sim_d  \frac{1}{(p-1)^d} (\rho_{\Lambda} -1 )^{-d/2}. $$
Similarly, one deduces that \eqref{main_result_d_2} is also optimal.

%%%%
\begin{rmk} By arguing as in the previous section, one shows that if $ \Lambda_n =  ( \lambda_{j_n}^{(n)}  )_{j_n \in \mathbb{N}_0}$ are finite unions of lacunary sequences ($n=1, \cdots, d$), then 
$$  \sup_{\substack{ \| f \|_{L^p (\mathbb{T}^d)} = 1 \\ f \in H^p_A (\mathbb{T}^d) } }  \Big\|  S^{ (\Lambda_1, \cdots, \Lambda_d)} (f)  \Big\|_{L^p (\mathbb{T}^d)} \lesssim_d  (p-1)^{-d}   \prod_{n=1}^d  \sigma_{\Lambda_n}^{1/2}   
$$
and
$$
\Big\|  S^{ (\Lambda_1, \cdots, \Lambda_d)}   \Big\|_{L^p (\mathbb{T}^d) \rightarrow L^p (\mathbb{T}^d)} \lesssim_d (p-1)^{-3d/2} \prod_{n=1}^d  \sigma_{\Lambda_n} ^{1/2} 
$$
and that the above estimates cannot be improved in general.
\end{rmk}
%%%%

%%%%%%%%%%%%%%%%%%%%%%%%%%%%
%%%%%%%%%%%%%%%%%%%%%%%%%%%%
%%%%%%%%%%%%%%%%%%%%%%%%%%%%
%%%%%%%%%%%%%%%%%%%%%%%%%%%%
\section{Euclidean Variants}\label{Euclidean}

Let $\Lambda = (\lambda_j)_{j \in \mathbb{Z}}$ be a countable collection of positive real numbers indexed by $\mathbb{Z}$ such that $\lambda_{j+1} > \lambda_j$ for all $j \in \mathbb{Z}$ as well as 
$$\lim_{j \rightarrow -\infty}\lambda_j = 0 \quad  \mathrm{and} \quad \lim_{j \rightarrow +\infty}\lambda_j = + \infty .$$
For $j \in \mathbb{Z}$ define the $j$-th ``rough" Littlewood-Paley projection $P^{(\Lambda)}_j$ to be the multiplier with symbol $\chi_{I_j \cup I'_j}$, where $I_j := [\lambda_{j-1}, \lambda_j ) $ and $I'_j := ( -\lambda_j , - \lambda_{j-1} ] $. 

Given a $d \in \mathbb{N}$ and sequences $\Lambda_1, \cdots, \Lambda_d$ as above, define the $d$-parameter ``rough" Littlewood-Paley square function 
$$ S_{\mathbb{R}^d }^{(\Lambda_1, \cdots, \Lambda_d)}  (f) : = \Bigg( \sum_{j_1, \cdots, j_d \in \mathbb{Z}} \Big| P^{(\Lambda_1)}_{j_1} \otimes \cdots \otimes P^{(\Lambda_d)}_{j_d} (f) \Big|^2 \Bigg)^{1/2}  $$
initially over the class of Schwartz functions on $\mathbb{R}^d$.

A Euclidean version of \eqref{main_result_d} is given by the following theorem.

%%%%
\begin{theo}\label{euclidean_Hp}
Given a  $d \in \mathbb{N}$, consider sequences $\Lambda_1, \cdots , \Lambda_d$ that are as above and moreover, satisfy
$$  \sigma_{\Lambda_n} := \sup_{ j \in \mathbb{Z} } \# \big( \Lambda_n \cap [2^{j-1} , 2^j) \big) < \infty $$
for all $n \in \{ 1, \cdots, d\}$.

Then, there exists an absolute constant $C_d >0$, depending only on $d$, such that 
\begin{equation}\label{main_result_eucl_d}
  \sup_{\substack{ \| f \|_{L^p (\mathbb{R}^d)} = 1 \\ f \in H^p_A (\mathbb{R}^d) } }  \Big\|  S^{ (\Lambda_1, \cdots, \Lambda_d)}_{\mathbb{R}^d} (f)  \Big\|_{L^p (\mathbb{R}^d)} \leq \frac{C_d}{(p-1)^d} \prod_{n=1}^d \sigma_{\Lambda_n}^{1/2}   
\end{equation}
for every $1<p<2$.
\end{theo}
%%%%

Similarly, a Euclidean version of \eqref{main_result_d_2} is the content of the following theorem.

%%%%
\begin{theo}\label{euclidean_Lp}
Given a  $d \in \mathbb{N}$, consider sequences $\Lambda_1, \cdots , \Lambda_d$ that are as above and moreover, satisfy
$$  \sigma_{\Lambda_n} := \sup_{ j \in \mathbb{Z} } \# \big( \Lambda_n \cap [2^{j-1} , 2^j)  \big) < \infty $$
for all $n \in \{ 1, \cdots, d\}$.

Then, there exists an absolute constant $C_d >0$, depending only on $d$, such that 
\begin{equation}\label{main_result_eucl_d_2}
\Big\|  S^{ (\Lambda_1, \cdots, \Lambda_d)}_{\mathbb{R}^d}  \Big\|_{L^p (\mathbb{R}^d) \rightarrow L^p (\mathbb{R}^d)} \leq \frac{C_d}{(p-1)^{3d/2} }     \prod_{n=1}^d \sigma_{\Lambda_n}^{1/2}   
\end{equation}
for every $1<p<2$.
\end{theo}
%%%%

The proofs of Theorems \ref{euclidean_Hp} and \ref{euclidean_Lp} are given in Subsections \ref{euclidean_1} and \ref{euclidean_2} respectively.

%%%%%%%%
%%%%%%%%
\subsection{Proof of Theorem \ref{euclidean_Hp}}\label{euclidean_1} 
Let $n \in \{ 1, \cdots, d \}$ be a fixed index. Note that if we consider a probability space $(\Omega, \mathcal{A}, \mathbb{P})$ and $\Lambda_n = (\lambda_{j_n}^{(n)})_{j_n \in \mathbb{Z}}$ is as in the hypothesis of the theorem, then it follows from the work of Tao and Wright \cite{TW}, see also Remark \ref{second_proof}, that there exists an absolute constant $C_0 > 0$ such that
\begin{equation}\label{euclidean_w-t}
\Bigg\| \sum_{j_n \in \mathbb{Z}} r_{j_n} (\omega_n) P_{j_n}^{(\Lambda_n)} (f) \Bigg\|_{L^{1, \infty} (\mathbb{R})} \leq C_0 \sigma_{\Lambda_n}^{1/2} \| f \|_{L^1 (\mathbb{R})} \quad (f \in H^1_A (\mathbb{R}))
\end{equation}
for every choice of $\omega_n \in \Omega$. 

We can now argue as in \cite[Section 5]{BRS}. Assume first that $p \in (1, 3/2]$ and note that by using \eqref{euclidean_w-t} together with the trivial bound
$$  \Bigg\| \sum_{j_n \in \mathbb{Z}} r_{j_n} (\omega_n) P_{j_n}^{(\Lambda_n)}  \Bigg\|_{H^2_A (\mathbb{R}) \rightarrow H^2_A (\mathbb{R})} = 1  $$
and a Marcinkiewicz decomposition for functions in analytic Hardy spaces on the real line, which is due to Peter Jones; see \cite[Theorem 2]{Peter_Jones}, one deduces that
\begin{equation}\label{euclidean_bound_Hp}
  \Bigg\| \sum_{j_n \in \mathbb{Z}} r_{j_n} (\omega_n) P_{j_n}^{(\Lambda_n)}  \Bigg\|_{H^p_A (\mathbb{R}) \rightarrow H^p_A (\mathbb{R})} \leq A_0 \sigma_{\Lambda_n}^{1/2} (p-1)^{-1}   ,
  \end{equation}
where $A_0 > 0$ is an absolute constant; see the proof of \cite[Proposition 8]{BRS}. To show that \eqref{euclidean_bound_Hp} also holds for $p \in (3/2 ,2)$, one can use, e.g., \eqref{euclidean_bound_Hp} for $p=3/2$ and duality so that one gets an appropriate $H^3_A (\mathbb{R})$ to $H^3_A (\mathbb{R})$ bound and then, as in the previous case, interpolate between this bound and \eqref{euclidean_w-t}.

Therefore, \eqref{euclidean_bound_Hp} holds for all $ p \in (1,2)$ and hence, by iterating \eqref{euclidean_bound_Hp} and then using multi-dimensional Khintchine's inequality  \eqref{Khintchine}, \eqref{main_result_eucl_d} follows. 

%%%%
\begin{rmk} By adapting the argument of Section \ref{finite_union} to the Euclidean setting, where one uses a variant of  the construction in the proof of \cite[Corollary 10]{BRS}, one can show that the exponents $r=1/2$ in $\sigma_{\Lambda_n}^{1/2} $ in \eqref{main_result_eucl_d} are optimal. 
\end{rmk}
%%%%

%%%%%%%%
%%%%%%%%
\subsection{Proof of Theorem \ref{euclidean_Lp}}\label{euclidean_2} The proof of \eqref{main_result_eucl_d_2} for $d=1$ is obtained by carefully examining  Lerner's argument that establishes \cite[Theorem 1.1]{Lerner}. To be more specific, the proof of \eqref{main_result_eucl_d_2} will be a consequence of the following weighted estimate
\begin{equation}\label{weighted_Lp}
\big\| S^{(\Lambda_n)}_{\mathbb{R}} \big\|_{L^2 (w) \rightarrow L^2 (w) }  \leq C \sigma_{\Lambda_n}^{1/2} [w]_{A_2}^{3/2} \quad (n \in \{1, \cdots, d \})
\end{equation}
combined with an extrapolation result, which is due to Duoandikoetxea; see \cite[Theorem 3.1]{Duoandikoetxea}. Recall that a non-negative locally integrable function on $\mathbb{R}$ is said to be an $A_2$ weight if, and only if,
$$ [w]_{A_2} :=  \sup_{\substack{ J \subseteq \mathbb{R}:\\ J \ \mathrm{interval} }} \langle w \rangle_J \langle w^{-1} \rangle_J $$
is finite. Here, we use the notation $\langle \sigma \rangle_J : = |J|^{-1} \int_J \sigma(y) dy$. 

For a fixed  index $n \in \{1, \cdots, d\}$, set $\widetilde{\Lambda}_n : = \Lambda \cup (2^j)_{j \in \mathbb{Z}}$
and observe that for every Schwartz function $f$ one has
 $$  S^{( \Lambda_n)}_{\mathbb{R}} (f) (x) \lesssim S^{(\widetilde{\Lambda}_n)}_{\mathbb{R}} (f) (x) $$
 for all $x \in \mathbb{R}$. Hence, in order to establish \eqref{weighted_Lp}, it suffices to prove that for every $A_2$ weight $w$ one has
 \begin{equation}\label{weighted_Lp_2}
\big\| S^{(\widetilde{\Lambda}_n)}_{\mathbb{R}} \big\|_{L^2 (w) \rightarrow L^2 (w) }  \leq C' \sigma_{\Lambda_n}^{1/2} [w]_{A_2}^{3/2} \quad (n \in \{1, \cdots, d \} )
\end{equation}
for each fixed index $n \in \{ 1, \cdots, d\}$. To prove \eqref{weighted_Lp_2}, note that by arguing exactly as in \cite{Lerner}, namely by using duality and \cite[Theorem 2.7]{Lerner}, it suffices to show that there exists an absolute constant $C''>0$ such that
\begin{equation}\label{weighted_bound}
  \Bigg\| S_{\phi, \mathcal{I}} \Bigg( \sum_{j \in \mathbb{Z}} P^{(\widetilde{\Lambda}_n)}_j (\psi_j) \Bigg) \Bigg\|_{L^2 (w)}  \leq C'' \sigma_{\Lambda_n}^{1/2} [w]_{A_2}  \Bigg\| \Bigg( \sum_{j \in \mathbb{Z}} |\psi_j|^2 \Bigg)^{1/2}  \Bigg\|_{L^2 (w)}
  \end{equation}
for every $A_2$ weight $w$ and for every countable collection of Schwartz functions $(\psi_j)_{j \in \mathbb{Z}}$, see also \cite[Remark 4.1]{Lerner}. Here, as in \cite{Lerner}, given a dyadic lattice $\mathcal{I}$ in $\mathbb{R}$; see \cite[Definition 2.1]{Lerner}, one sets
$$ S_{\phi, \mathcal{I}} (f) (x) : = \Bigg( \sum_{k \in \mathbb{Z}} \sum_{\substack{J \in \mathcal{I}:\\ | J | = 2^{-k}}}   \langle | f \ast \phi_k |^2 \rangle_J  \chi_J (x)  \Bigg)^{1/2} \quad (x \in \mathbb{R}), $$
where  $\phi$ is a Schwartz function satisfying $\mathrm{supp} (\widehat{\phi}) \subseteq [1/2,2]$ and for $k \in \mathbb{Z}$ we use the notation $\phi_k  (x):= 2^k \phi (2^k x )$, $x \in \mathbb{R}$. 

Write $\widetilde{\Lambda}_n = (\widetilde{\lambda}_j^{(n)})_{j \in \mathbb{Z}}$ and $\widetilde{I}^{(n)}_j : = [\widetilde{\lambda}_{j-1}^{(n)}, \widetilde{\lambda}_j^{(n)})$ for $j \in \mathbb{Z}$. To prove \eqref{weighted_bound}, observe that by using the Cauchy-Schwarz inequality one has 
\begin{align*}
& \Bigg| S_{\phi, \mathcal{I}} \Bigg( \sum_{j \in \mathbb{Z}} P^{(\widetilde{\Lambda}_n)}_j (\psi_j) \Bigg) (x) \Bigg|^2 \leq \\
& 4 \sigma_{\Lambda_n}  \sum_{k \in \mathbb{Z}} \sum_{ \widetilde{I}_{j_k}^{(n)}   \subseteq [2^{k-1}, 2^{k+1} ) }  \sum_{\substack{J \in \mathcal{I}: \\ |J| = 2^{-k} }} \Big \langle \Big|  P^{(\widetilde{\Lambda}_n)}_{j_k} (\psi_{j_k}) \ast \phi_k \Big|^2  \Big \rangle_J \chi_J (x) 
\end{align*}
and hence,
$$  \Bigg\| S_{\phi, \mathcal{I}} \Bigg( \sum_{j \in \mathbb{Z}} P^{(\widetilde{\Lambda}_n)}_j (\psi_j) \Bigg)  \Bigg\|^2_{L^2 (w)} \leq 4 \sigma_{\Lambda_n}    \sum_{k \in \mathbb{Z}} \sum_{ \widetilde{I}_{j_k}^{(n)} \subseteq [2^{k-1}, 2^{k+1} ) } \Big\| P^{(\widetilde{\Lambda}_n)}_{j_k} (\psi_{j_k}) \ast \phi_k \Big\|^2_{L^2 (w_{k, \mathcal{I}})}, $$
where, as in \cite{Lerner}, one sets
$$ w_{k, \mathcal{I}} : = \sum_{\substack{ J \in \mathcal{I}: \\ |J| = 2^{-k} }} \langle w \rangle_J \chi_J . $$
By using \cite[Lemma 3.2]{Lerner}, for every $k \in \mathbb{Z}$, one gets
$$  \Big\|  P^{(\widetilde{\Lambda}_n)}_{j_k} (\psi_{j_k}) \ast \phi_k \Big\|_{L^2 (w_{k, \mathcal{I}})}  \leq C  _0 [w]_{A_2} \| \psi_{j_k} \|_{L^2 (w)}  $$
for all $j_k \in \mathbb{Z}$ such that $\widetilde{I}_{j_k}^{(n)} \subseteq [2^{k-1}, 2^{k+1} )$, where $C_0 >0$ is a constant depending only on the function $\phi$. Hence, \eqref{weighted_bound} is obtained by combining the last two estimates. Therefore, we deduce that \eqref{weighted_Lp} holds. 

Hence, by arguing as in \cite[Remark 4.2]{Lerner}, \eqref{weighted_Lp} and  \cite[Theorem 3.1]{Duoandikoetxea} imply that
\begin{equation}\label{Euclidean_Lp}
\big\| S^{(\Lambda_n)}_{\mathbb{R}} \big\|_{L^p (\mathbb{R}) \rightarrow L^p (\mathbb{R}) }  \leq A \sigma_{\Lambda_n}^{1/2} (p-1)^{-3/2} \quad (1 < p <2)
\end{equation}
for each index $n \in \{1, \cdots, d \}$. 

To obtain the higher-dimensional case, given a probability space $(\Omega, \mathcal{A}, \mathbb{P}) $, observe that \eqref{Euclidean_Lp} combined with a Euclidean version of   \cite[(3.2)]{Bourgain_89} implies that
\begin{equation}\label{Euclidean_Lp_random}
\Bigg\| \sum_{j_n \in \mathbb{Z}} r_{j_n} (\omega_n) P^{(\Lambda_n)}_{j_n}  \Bigg\|_{L^p (\mathbb{R}) \rightarrow L^p (\mathbb{R}) }  \leq A' \sigma_{\Lambda_n}^{1/2} (p-1)^{-3/2} \quad (1 < p <2)
\end{equation}
for every $\omega_n \in \Omega$ and for all indices $n \in \{1, \cdots, d \}$. Therefore,  \eqref{main_result_eucl_d_2} is obtained by first iterating \eqref{Euclidean_Lp_random}  and then using multi-dimensional Khintchine's inequality  \eqref{Khintchine}. 

%%%%
\begin{rmk} We remark that either by adapting the modification of Lerner's argument  \cite{TW} presented above to the periodic setting or by using an appropriate variant of \eqref{Euclidean_Lp_random} and transference; see Theorem 3.8 in Chapter VII in \cite{SteinWeiss}, one obtains an alternative proof of \eqref{main_result_d_2} and in particular, of Theorem \ref{lac_2} as well as of the bound \eqref{variant_Lp} in Theorem \ref{union_variant}.
\end{rmk}
%%%%

%%%%
\begin{rmk} By adapting the argument of Section \ref{finite_union} to the Euclidean setting, where one uses a variant of the construction in the proof of \cite[Proposition 6.1]{Bakas}, one can show that the exponents $r=1/2$ in $\sigma_{\Lambda_n}^{1/2} $ in \eqref{main_result_eucl_d_2} cannot be improved in general. 
\end{rmk}
%%%%

%%%%
\begin{rmk} Assume that $\Lambda = (\lambda_j)_{j \in \mathbb{Z}}$ is a countable collection of non-negative real numbers satisfying the assumptions of Theorem \ref{euclidean_2}. Then, it follows from the work of Tao and Wright \cite{TW} that
\begin{equation}\label{local_bound}
 \big\| S^{(\Lambda)}_{\mathbb{R}} \big\|_{L^{1, \infty} (K) \rightarrow L \log^{1/2} L (K)} \lesssim_K \sigma_{\Lambda}^{1/2}  
 \end{equation}
for every compact subset $K$ in $ \mathbb{R}$. Hence,  by using \eqref{local_bound}, interpolation, and a version of Tao's converse extrapolation theorem for operators restricted to compact sets in $ \mathbb{R}$; see p. 2 in \cite{Tao}, it follows that
\begin{equation}\label{local_bound_p}
 \big\| S^{(\Lambda)}_{\mathbb{R}} \big\|_{L^p ( K) \rightarrow L^p (K)} \lesssim_K \sigma_{\Lambda}^{1/2}  (p-1)^{-3/2} \quad (1<p<2) 
 \end{equation}
for  every compact  subset $K$ of the real line. Notice that \eqref{local_bound_p} is weaker than \eqref{main_result_eucl_d_2} (for $d=1$).
\end{rmk}
%%%%

%%%%%%%%%%%%%%%%%%%%%%%%%%%%
%%%%%%%%%%%%%%%%%%%%%%%%%%%%
%%%%%%%%%%%%%%%%%%%%%%%%%%%%
%%%%%%%%%%%%%%%%%%%%%%%%%%%%
\section*{Acknowledgements} The author would like to thank Prof. James Wright for several discussions on matters related to the paper \cite{TW}, as well as Dr. Alan Sola and Prof. James Wright for their comments on an earlier version of this manuscript.

%%%%%%%%%%%%%%%%%%%%%%%%%%%%
%%%%%%%%%%%%%%%%%%%%%%%%%%%%
%%%%%%%%%%%%%%%%%%%%%%%%%%%%
%%%%%%%%%%%%%%%%%%%%%%%%%%%%

\end{document}